\documentclass[english]{amsart}

\usepackage{amsmath}

\usepackage{amssymb}

\usepackage{amscd} \usepackage{tabularx}
\usepackage[all,cmtip,line]{xy}

\newcommand{\ra}{\rightarrow}
\newcommand{\lra}{\longrightarrow}

\newcommand{\into}{\hookrightarrow}

\newcommand{\da}{\downarrow}
\newcommand{\iso}{\stackrel{\sim}{\ra}}
\newcommand{\liso}{\stackrel{\sim}{\lra}}

\newcommand{\pfbegin}{{{\em Proof:}\;}}
\newcommand{\pfend}{$\Box$ \medskip}
\newcommand{\mat}[4]{\left( \begin{array}{cc} {#1} & {#2} \\ {#3} & {#4}
\end{array} \right)}

\newlength{\ownl}

\newcommand{\norm}{{\mbox{\bf N}}}
\newcommand{\ndiv}{{\mbox{$\not| $}}}

\newcommand{\Art}{{\operatorname{Art}\,}}
\newcommand{\Aut}{{\operatorname{Aut}\,}}
\newcommand{\BC}{{\operatorname{BC}\,}}

\newcommand{\Fil}{{\operatorname{Fil}\,}}

\newcommand{\Frob}{{\operatorname{Frob}}}
\newcommand{\conju}{{\operatorname{conj}}}
\newcommand{\res}{{\operatorname{res}}}
\newcommand{\ind}{{\operatorname{ind}}}
\newcommand{\Gal}{{\operatorname{Gal}\,}}

\newcommand{\Hom}{{\operatorname{Hom}\,}}

\renewcommand{\Im}{{\operatorname{Im}\,}}
\newcommand{\Ind}{{\operatorname{Ind}\,}}

\newcommand{\Lie}{{\operatorname{Lie}\,}}

\renewcommand{\Re}{{\operatorname{Re}\,}}

\newcommand{\Rep}{{\operatorname{REP}}}
\newcommand{\REP}{{\operatorname{GG}}}
\newcommand{\GrG}{{\operatorname{Rep}}}
\newcommand{\Iw}{{\operatorname{Iw}}}

\newcommand{\WD}{{\operatorname{WD}}}

\newcommand{\Spec}{{\operatorname{Spec}\,}}

\newcommand{\Spp}{{\operatorname{Sp}\,}}

\newcommand{\ad}{{\operatorname{ad}\,}}

\newcommand{\gr}{{\operatorname{gr}\,}}

\newcommand{\rec}{{\operatorname{rec}}}
\newcommand{\tr}{{\operatorname{tr}\,}}

\newcommand{\wt}{{\operatorname{wt}}}

\newcommand{\sgn}{{\operatorname{sgn}\,}}

\newcommand{\diag}{{\operatorname{diag}}}

\newcommand{\crord}{{\operatorname{cr-ord}}}
\newcommand{\ssord}{{\operatorname{ss-ord}}}
\newcommand{\cris}{{\operatorname{cris}}}

\newcommand{\der}{{\operatorname{der}}}

\newcommand{\ab}{{\operatorname{ab}}}

\newcommand{\nr}{{\operatorname{nr}}}
\newcommand{\pnr}{{\operatorname{p-nr}}}

\newcommand{\semis}{{\operatorname{ss}}}
\newcommand{\Fsemis}{{\operatorname{F-ss}}}
\newcommand{\SC}{{\operatorname{sc}}}
\newcommand{\tf}{{\operatorname{tf}}}
\newcommand{\tor}{{\operatorname{tor}}}

\newcommand{\univ}{{\operatorname{univ}}}

\newcommand{\A}{{\mathbb{A}}}

\newcommand{\C}{{\mathbb{C}}}

\newcommand{\F}{{\mathbb{F}}}
\newcommand{\G}{{\mathbb{G}}}

\newcommand{\PP}{{\mathbb{P}}}
\newcommand{\Q}{{\mathbb{Q}}}
\newcommand{\R}{{\mathbb{R}}}

\newcommand{\Z}{{\mathbb{Z}}}

\newcommand{\CC}{{\mathcal{C}}}
\newcommand{\calD}{{\mathcal{D}}}

\newcommand{\CG}{{\mathcal{G}}}

\newcommand{\CI}{{\mathcal{I}}}

\newcommand{\CK}{{\mathcal{K}}}
\newcommand{\CL}{{\mathcal{L}}}
\newcommand{\CM}{{\mathcal{M}}}

\newcommand{\CO}{{\mathcal{O}}}
\newcommand{\CP}{{\mathcal{P}}}

\newcommand{\CR}{{\mathcal{R}}}
\newcommand{\CS}{{\mathcal{S}}}

\newcommand{\gG}{{\mathfrak{G}}}

\newcommand{\gog}{{\mathfrak{g}}}

\newcommand{\gl}{{\mathfrak{l}}}
\newcommand{\gm}{{\mathfrak{m}}}
\newcommand{\gn}{{\mathfrak{n}}}

\newcommand{\gq}{{\mathfrak{q}}}

\newcommand{\gs}{{\mathfrak{s}}}

\newcommand{\barF}{\overline{{F}}}

\newcommand{\barH}{\overline{{H}}}

\newcommand{\barK}{\overline{{K}}}
\newcommand{\barL}{\overline{{L}}}
\newcommand{\barM}{\overline{{M}}}

\newcommand{\barR}{\overline{{R}}}

\newcommand{\barV}{\overline{{V}}}

\newcommand{\barFF}{\overline{{\F}}}

\newcommand{\barQQ}{\overline{{\Q}}}

\newcommand{\bare}{\overline{{e}}}

\newcommand{\barr}{\overline{{r}}}
\newcommand{\bars}{\overline{{s}}}

\newcommand{\tC}{\widetilde{{C}}}

\newcommand{\tG}{\widetilde{{G}}}
\newcommand{\tH}{\widetilde{{H}}}

\newcommand{\tS}{\widetilde{{S}}}
\newcommand{\tT}{\widetilde{{T}}}

\newcommand{\tZ}{\widetilde{{Z}}}

\newcommand{\tg}{\widetilde{{g}}}

\newcommand{\ts}{\widetilde{{s}}}

\newcommand{\tu}{\widetilde{{u}}}
\newcommand{\tv}{{\widetilde{{v}}}}

\newcommand{\barepsilon    }{\overline{\epsilon}}

 \newcommand{\bartheta    }{\overline{\theta}}

 \newcommand{\barmu    }{\overline{\mu}}

 \newcommand{\barrho   }{{\overline{\rho}}}   
    
 \newcommand{\barsigma   }{\overline{\sigma}}

 \newcommand{\barpsi   }{\overline{\psi}}

\newcommand{\barPhi    }{\overline{\Phi}}

 \newcommand{\tmu    }{\widetilde{\mu}}

 \newcommand{\trho   }{\widetilde{\rho}}

 \newcommand{\ttau     }{\widetilde{\tau}}

 \newcommand{\tchi   }{\widetilde{\chi}}

\newcommand{\hatotimes}{{\widehat{\otimes}}}

\newcommand{\Fbar}{{\overline{\F}}}
\newcommand{\tbarr}{{\widetilde{\barr}}}
\newcommand{\hatQQ}{{\widehat{\Q}}}
\newcommand{\hatZZ}{{\widehat{\Z}}}

\def\RCS$#1: #2 ${\expandafter\def\csname RCS#1\endcsname{#2}}
\RCS$Revision: 400 $
\RCS$Date: 2013-12-06 14:30:06 +0000 (Fri, 06 Dec 2013) $

\newcommand{\psibar}{\overline{\psi}}
\newcommand{\onto}{\twoheadrightarrow}

\newcommand{\cL}{\mathcal{L}}

\newcommand{\Gn}{\mathcal{G}_n}

\newcommand{\Favoid}{F^{(\mathrm{avoid})}} 
\newcommand{\Kavoid}{K^{(\mathrm{avoid})}}

\newcommand{\thetabar}{\bar{\theta}}

\newcommand{\bb}{\mathbb} 
\newcommand{\mc}{\mathcal}
\newcommand{\mf}{\mathfrak}

\DeclareMathOperator{\ord}{ord}

\DeclareMathOperator{\GO}{GO}
\DeclareMathOperator{\FL}{FL}

\newcommand{\rbar}{\bar{r}}

\newcommand{\mubar}{\overline{\mu}}

\newcommand{\GL}{\operatorname{GL}}
\newcommand{\GSp}{\operatorname{GSp}}
\newcommand{\Sp}{\operatorname{Sp}}
\newcommand{\PGL}{\operatorname{PGL}}
\newcommand{\HT}{\operatorname{HT}}

\newcommand{\cC}{\mathcal{C}}

\newcommand{\cG}{\mathcal{G}}

\newcommand{\cI}{\mathcal{I}}

\newcommand{\cO}{\mathcal{O}}

 \newcommand{\Qp}{\Q_p}
\newcommand{\Ql}{\Q_l} 

\newcommand{\Qlbar}{\overline{\Q}_{l}}

\newcommand{\Flbar}{\overline{\F}_l}

\newcommand{\SL}{\operatorname{SL}}

\newcommand{\Sym}{\operatorname{Sym}}

\usepackage{hyperref}

\usepackage{amsthm}

 \newtheorem{ithm}{Theorem}
\newtheorem{thm}{Theorem}[subsection]

\newtheorem{cor}[thm]{Corollary}
 \newtheorem{lemma}[thm]{Lemma}
\newtheorem{lem}[thm]{Lemma} \newtheorem{prop}[thm]{Proposition}
 \theoremstyle{definition}
 \theoremstyle{definition}
 \theoremstyle{remark}

\numberwithin{equation}{subsection}

\theoremstyle{definition}

\begin{document}
\title[Potential automorphy]{Potential automorphy and change of weight.}

\author{Thomas Barnet-Lamb}\email{tbl@brandeis.edu}\address{Department of Mathematics, Brandeis University}
\author{Toby Gee} \email{toby.gee@imperial.ac.uk} \address{Department of
  Mathematics, Imperial College London} \author{David Geraghty}
\email{geraghty@math.princeton.edu}\address{Princeton University and
  Institute for Advanced Study} \author{Richard Taylor}
\email{rtaylor@ias.edu}\address{Institute for Advanced Study, Princeton} \thanks{The second author was partially supported
  by NSF grant DMS-0841491, the third author was partially supported
  by NSF grant DMS-0635607 and the fourth author was partially
  supported by NSF grants DMS-0600716 and DMS-1062759 and by the Oswald Veblen and Simonyi Funds at the IAS}  \subjclass[2000]{11F33.}
\begin{abstract}
We prove an automorphy lifting theorem for $l$-adic representations where we impose a new condition at $l$, which we call `potential diagonalizability'. This result allows for `change of weight' and  seems to be substantially more flexible than previous theorems along the same lines. We derive several applications. For instance 
we show that any irreducible, totally odd, essentially self-dual,
regular, weakly compatible system of $l$-adic representations of the
absolute Galois group of a totally real field is potentially automorphic, and hence is pure and its L-function has meromorphic continuation to the whole complex plane and satisfies the expected functional equation. 
\end{abstract}
\maketitle
\newpage

\section*{Introduction.}\label{sec:intro}
Suppose that $F$ and $M$ are
number fields, that $S$ is a finite set of primes of $F$ and that $n$
is a positive integer. By a {\em weakly compatible system} of $n$-dimensional $l$-adic representations of $G_F$ defined over $M$ and unramified outside $S$ we shall mean a family
of continuous semi-simple representations
\[ r_\lambda: G_F \lra \GL_n(\barM_\lambda), \]
where $\lambda$ runs over the finite places of $M$, with the following properties.
\begin{itemize}
\item If $v\notin S$ is a finite place of $F$, then for all $\lambda$ not dividing the residue characteristic of $v$, the representation $r_\lambda$ is unramified at $v$ and the characteristic
polynomial of $r_\lambda(\Frob_v)$ lies in $M[X]$ and is independent of $\lambda$. 
\item Each
representation $r_\lambda$ is de Rham at all places above the residue characteristic of $\lambda$, and in fact crystalline at any place $v \not\in S$ which divides
the residue characteristic of $\lambda$. 
\item For each embedding $\tau:F \into \barM$ the $\tau$-Hodge--Tate numbers of $r_\lambda$ are independent of $\lambda$.
\end{itemize}

In this paper we prove the following theorem (see Theorem \ref{mtcs}). 

\begin{ithm}\label{thma} Let $\{ r_\lambda\}$ be a weakly compatible system of $n$-dimensional $l$-adic representations of $G_F$ defined over $M$ and unramified outside $S$, where for simplicity we
assume  that $M$ contains the image of each embedding $F \into \barM$. Suppose that $\{ r_\lambda\}$ satisfies the following properties.
\begin{enumerate}
\item {\bf (Irreducibility)} Each $r_\lambda$ is irreducible.
\item {\bf (Regularity)} For each embedding $\tau:F \into M$ the representation $r_\lambda$ has $n$ distinct $\tau$-Hodge--Tate numbers. 
\item {\bf (Odd essential self-duality)} $F$ is totally real; and either each $r_\lambda$ factors through a map to $\GSp_n(\barM_\lambda)$ with a totally odd multiplier character; or each $r_\lambda$ factors through a map to $\GO_n(\barM_\lambda)$ with a totally even multiplier character. Moreover in either case the multiplier characters 
form a weakly compatible system. 
\end{enumerate}

Then there is a finite, Galois, totally real extension of $F$ over which all the $r_\lambda$'s become
automorphic. In particular for any embedding $\imath:M \into \C$ the partial L-function $L^S(\imath \{ r_\lambda \},s)$ converges in some right half plane and has meromorphic continuation to the whole complex plane.
\end{ithm}

This is not the first paper to prove potential automorphy results for compatible systems of $l$-adic
representations of dimension greater than $2$, see for example \cite{hsbt}, \cite{blght}, \cite{blgg}.
However previous attempts only applied to very specific, though well known, examples (e.g. symmetric
powers of the Tate modules of elliptic curves) and one had to exploit special properties of these
examples. We believe this is the first general potential automorphy theorem in dimension greater than $2$, and we are hopeful that it can be applied to many examples. We give an analogous theorem when $F$ is an imaginary CM field. Other than this we do not see how to improve much on this theorem using current methods.

As one example application, suppose that $\CK$ is a finite set of positive integers such that the $2^{\# \CK}$ possible partial
sums of elements of $\CK$ are all distinct. For each $k \in \CK$ let $f_k$ be an elliptic modular newform
of weight $k+1$ without complex multiplication. Then the $\# \CK$-fold tensor product of the $l$-adic representations associated to the $f_k$ is potentially automorphic and the $\# \CK$-fold product L-function for the $f_k$ has meromorphic continuation to the whole complex plane. (See Corollary \ref{product}.)

The proof of Theorem \ref{thma} follows familiar lines. One works with $r_\lambda$ for one suitably chosen $\lambda$. One finds a motive $X$ over some finite Galois totally real extension $F'/F$ which realizes the reduction
$\barr_\lambda$ in its mod $ l$ cohomology and whose mod $ l'$ cohomology is induced from a
character. One tries to argue that by automorphic induction the mod $ l'$ cohomology is automorphic over $F'$,
hence by an automorphy lifting theorem the $l'$-adic cohomology is automorphic over $F'$, hence
tautologically the mod $l$ cohomology is automorphic over $F'$ and hence, finally, by another 
automorphy lifting theorem $r_\lambda$ is automorphic over $F'$. To
find $X$ one uses a lemma of Moret-Bailly \cite{mb}, \cite{gpr} and
for this one needs a family of motives with distinct Hodge numbers, which has large monodromy. Griffiths transversality tells us that this will only be possible if the Hodge numbers of the motives are consecutive (e.g $0,1,2,\dots,n-1$). Thus the $l$-adic cohomology of $X$ may be 
automorphic of a different weight (infinitesimal character) than $r_\lambda$ and the second
automorphy lifting theorem needs to incorporate a `change of weight'. In addition it seems that we can in general only expect to find $X$ over an extension $F'/F$ which is highly ramified at $l$. Thus our
second automorphy lifting theorem needs to work over  a base which is highly ramified at $l$. 
These two, related problems were the principal difficulties we faced. The original higher dimensional automorphy lifting theorems (see \cite{cht}, \cite{tay}) could handle neither of them. In the
ordinary case one of us (D.G.) proved an automorphy lifting theorem that uses Hida theory and some new local calculations to handle both of these problems (see \cite{ger}). This has had important applications, but its applicability is still severely limited because we don't know how to prove that
many compatible systems of $l$-adic representations are ordinary infinitely often.  

The main innovation of this paper is a new automorphy lifting theorem that handles both these 
problems in significant generality. One of our key ideas is to introduce the notion of a potentially crystalline representation $\rho$ of the absolute Galois group of a local field $K$ being {\em potentially diagonalizable}: $\rho$ is potentially diagonalizable if there is a finite extension $K'/K$ such that
$\rho|_{G_{K'}}$ lies on the same irreducible component of the universal crystalline lifting ring of $\barrho|_{G_{K'}}$ (with fixed Hodge--Tate numbers) as a sum of characters lifting $\barrho|_{G_{K'}}$.
(We remark that this does not depend on the choice of integral model for $\rho$.) Ordinary crystalline
representations are potentially diagonalizable, as are crystalline representations in the Fontaine--Laffaille range (i.e.\ over an absolutely unramified base and with Hodge--Tate numbers in the range $[0,l-2]$). Potential diagonalizability is also preserved under restriction to the absolute Galois group of a finite extension. In this sense they behave better than `crystalline representations in the Fontaine--Laffaille range' which require the ground field to be absolutely unramified. Finally `potentially
diagonalizable' representations are perfectly suited to our method of proving automorphy lifting theorems that allow for a change of weight. It seems to us to be a very interesting question to clarify further the ubiquity of potential diagonalizability. Could every crystalline representation be potentially diagonalizable? (We have no reason to believe this, but we know of no counterexample.)

The following gives an indication of the sort of automorphy lifting
theorems we are able to prove. (See Theorem \ref{mainmlt} and also
section \ref{terminology} for the definition of any notation or terminology which may be unfamiliar.) 
     \begin{ithm}\label{thmb}     Let $F$ be an imaginary CM field with maximal totally real subfield $F^+$ and let $c$ denote the non-trivial element of $\Gal(F/F^+)$. Let $n$ denote a positive integer. Suppose that $l\geq 2(n+1)$ is a prime such that $F$ does not contain a primitive $l^{th}$ root of $1$. Let
\[ r: G_F \lra \GL_n(\overline{\Q}_l) \]
be a continuous irreducible representation and let $\barr$ denote the semi-simplification of the reduction
of $r$. Also let
\[ \mu:G_{F^+} \lra \barQQ_l^\times \]
be a continuous character.
Suppose that $r$ and $\mu$ enjoy the following properties:
   \begin{enumerate}
\item {\bf (Odd essential conjugate-self-duality)} $r^c \cong r^\vee \mu$ and $\mu(c_v)=-1$ for all $v|\infty$.
\item {\bf (Unramified almost everywhere)} $r$ ramifies at finitely many primes.
\item {\bf (Potential diagonalizability and regularity)} $r|_{G_{F_v}}$ is potentially diagonalizable (and so in particular potentially crystalline) for all $v|l$ and for each embedding $\tau:F \into \barQQ_l$ it has $n$ distinct $\tau$-Hodge--Tate numbers.
\item {\bf (Irreducibility)}  The restriction $ \barr|_{G_{F(\zeta_l)}}$ is irreducible.
\item {\bf (Residual ordinary automorphy)} There is a regular algebraic, cuspidal, polarized automorphic representation  $(\pi,\chi)$ of $\GL_n(\bb{A}_{F})$ such that
       \[ (\barr,\barmu)\cong(\barr_{l,\imath}(\pi), \barr_{l,\imath}(\chi)\barepsilon_l^{1-n} ) \]
       and $\pi$ is $\imath$-ordinary.
      \end{enumerate}
Then $(r,\mu)$ is automorphic.  
\end{ithm}

Theorem \ref{thmb} implies the following potential automorphy theorem for a single $l$-adic representation, from which Theorem \ref{thma} can be deduced. (See Corollary \ref{mtpmtotreal} and Theorem \ref{incompsyst}.)
\begin{ithm}\label{thmc} Suppose that $F$ is a totally real field.
Let $n$ be a positive integer and let $l\geq 2(n+1)$ be a prime. 
Let 
\[ r: G_{F} \ra \GL_{n}(\barQQ_{l}) \]
be a continuous representation. We will write $\barr$ for the semi-simplification of the reduction of $r$.
Suppose that the following conditions are satisfied.
\begin{enumerate}
\item {\bf (Unramified almost everywhere)} $r$ is unramified at all but finitely many primes.
\item {\bf (Odd essential self-duality)} Either $r$ maps to $\GSp_n$ with totally odd multiplier or it maps to $\GO_n$ with totally even multiplier.
\item {\bf (Potential diagonalizability and regularity)}
  $r$ is potentially diagonalizable (and hence potentially crystalline) at each prime $v$ of $F$ above $l$ and for each $\tau:F \into \barQQ_{l}$ it has $n$ distinct $\tau$-Hodge--Tate numbers.
  \item {\bf (Irreducibility)} $\bar{r}|_{G_{F(\zeta_{l})}}$ is irreducible.
  \end{enumerate}

Then we can find a finite Galois totally real extension $F'/F$ such
that $r|_{G_{F'}}$ is automorphic. Moreover $r$ is part of a weakly
compatible system of $l$-adic representations. (In fact, $r$ is part
of a strictly pure compatible system in the sense of section \ref{cs}.) 
\end{ithm}

This theorem has other applications besides Theorem \ref{thma}. For instance we mention the following irreducibility result (see Theorem \ref{irred}). 
\begin{ithm} \label{thmd} Suppose that $F$ is a CM field and that $\pi$ is a regular, algebraic, essentially conjugate
  self-dual, cuspidal automorphic representation of $\GL_n(\A_F)$. If
  $\pi_\infty$ has sufficiently regular weight (`extremely regular' in
  the sense of section \ref{terminology}), then for $l$ in a set of
  rational primes of Dirichlet density $1$ the $n$-dimensional $l$-adic representations associated to $\pi$ are irreducible. \end{ithm}
  
To prove Theorem \ref{thmb} we employ Harris' tensor product trick (see
\cite{harris:manin}), which was first employed in connection with
change of weight in \cite{blgg}. However the freedom that `potential
diagonalizability' gives us to make highly ramified base changes in
the non-ordinary case means that this method becomes more
powerful. More precisely, suppose that $r$ is potentially
diagonalizable, and that $r_0$ is a potentially diagonalizable,
automorphic lift of $\barr$ (with possibly different
Hodge--Tate numbers to $r$). In fact making a finite soluble
base change we can assume they
are diagonalizable, i.e.\ we can take $K'=K$ in the definition of
potential diagonalizability. We choose a cyclic extension $M/F$ of
degree $n$ in which each prime above $l$ splits completely, and two $l$-adic characters $\theta$
and $\theta_0$ of $G_M$ such that
\begin{itemize}
\item $\bar\theta=\bartheta_0$,
\item the restriction of $\Ind_{G_M}^{G_F} \theta$ to an inertia group at a prime $v|l$ realizes a diagonal point on the same component of the universal crystalline lifting ring of $\barr|_{G_{F_v}}$ as $r|_{G_{F_v}}$,
\item and the restriction of $\Ind_{G_M}^{G_F} \theta_0$ to an inertia group at a prime $v|l$ realizes a diagonal point on the same component of the universal crystalline lifting ring of $\barr|_{G_{F_v}}$ as $r_0|_{G_{F_v}}$.
\end{itemize}
Then $r_0 \otimes \Ind_{G_M}^{G_F} \theta$ is automorphic and has the same reduction as 
$r \otimes \Ind_{G_M}^{G_F} \theta_0$. Moreover the restrictions of these two representations
to the decomposition group at a prime $v|l$ lie on the same component of the universal crystalline lifting ring of $(\barr \otimes \Ind_{G_M}^{G_F} \bartheta_0)|_{G_{F_v}}$. This is enough for the
usual Taylor--Wiles--Kisin argument to prove that $r \otimes \Ind_{G_M}^{G_F} \theta_0$ is also
automorphic, from which we can deduce (as in \cite{blght}) the automorphy of $r$. 

Things are a little more complicated than this because it seems to be
hard to combine this with the `level changing' argument in
\cite{tay}. In addition a direct argument imposes minor, but unwanted, conditions
on the Hodge--Tate numbers of $r_0$ and $r$. So instead of
going directly from the automorphy of $r_0$ to that of $r$ we
create two ordinary lifts $r_1$ and $r_2$ of $\barr$ (at
least after a base change) where $r_1$ has the same local behavior
away from $l$ as $r_0$; $r_2$ has the same local behavior away from
$l$ as $r$; and where the Hodge--Tate numbers of $r_1$ and
$r_2$ are chosen suitably. Our new arguments allow us to deduce the
automorphy of $r_1$ from that of $r_0$. D.G.'s results in the ordinary
case \cite{ger} allow us to deduce the automorphy of $r_2$ from that
of $r_1$. Finally  applying our new argument again allows us to deduce
the automorphy of $r$ from the automorphy of $r_2$. To
construct $r_1$ and $r_2$ we use the method of Khare and Wintenberger
\cite{kw} based on potential automorphy (in the ordinary case, where it is already available: see for example \cite{blght}).  

Along the way we also prove a general theorem about the existence of $l$-adic lifts of a given mod $l$ Galois representation with prescribed local behavior (see Theorem \ref{diaglift}). We deduce a rather general theorem about change of weight and level (see Theorem \ref{cwl}) of which a very particular instance is the following.

\begin{ithm}\label{thmf}Let $n$ be a positive integer and let $l>2(n+1)$ be a prime. Fix $\imath:\barQQ_l \iso \C$.
  Let $F$ be a  CM field such that all primes of $F$ above $l$ are unramified over $\Q$ and split over the maximal totally real subfield of $F$. 
   Let $\pi$ be a regular, algebraic, polarizable, cuspidal automorphic representation of $\GL_n(\A_F)$ satisfying the following conditions:
  \begin{itemize}
  \item $\pi$ is unramified above $l$;
  \item $\pi_\infty$ has weight $(a_{\tau,i})_{\tau:F \into \C,\,\, i=1,\dots,n}$ with $l-n-1 \geq a_{\tau,1} \geq a_{\tau,2} \geq \dots \geq a_{\tau,n} \geq 0$ for all $\tau$;
  \item and the restriction to $G_{F(\zeta_l)}$ of the mod $l$ Galois representation $\barr_{l,\imath}(\pi)$ associated to $\pi$ and $\imath$ is irreducible.
  \end{itemize}
  Note that in this case $a_{\tau,i}+a_{\tau c, n+1-i}=w$ is independent of $\tau$ and $i$.
  Suppose that we are given a second weight $(a_{\tau,i}')_{\tau:F \into \C,\,\, i=1,\dots,n}$ with $l-n-1 \geq a_{\tau,1}' \geq a_{\tau,2}' \geq \dots \geq a_{\tau,n}' \geq 0$ for all $\tau$, such that 
  \begin{itemize}
  \item $a_{\tau,i}'+a_{\tau c, n+1-i}'=w$ is also independent of $\tau$ and $i$,
  \item and for all places $v|l$ of $F$ the restriction $\barr_{l,\imath}(\pi)|_{G_{F_v}}$ has a lift which is crystalline with $\tau$-Hodge--Tate numbers $\{ a_{\imath \tau, i}'+n-i \}$.
  \end{itemize}
 
 Then there is a second regular, algebraic, polarizable, cuspidal automorphic representation $\pi'$ of $\GL_n(\A_F)$ giving rise (via $\imath$) to the same mod $l$ Galois representation (i.e.\ `congruent to $\pi$ mod $l$') such that $\pi'$ is also unramified above $l$ and
 $\pi_\infty$ has weight $(a_{\tau,i}')$.
\end{ithm}

We remark that combining the results of this paper with work of Caraiani, one can deduce full local global compatibility of the $l$-adic  representations associated to regular algebraic, essentially conjugate self-dual, cuspidal automorphic representations of $\GL_n$ over a CM or totally real field. (See \cite{blggt2} and \cite{ana2}.)

 We also remark that Stefan Patrikis and one of us (R.T.) recently combined the methods and results of this paper with one further idea, originating in \cite{patrikis}, and obtained  
variants of theorems \ref{thma} and \ref{thmd} which are perhaps more useful in practice. (See \cite{pt}.) More specifically they proved a version of theorem \ref{thma} where the irreducibility assumption is replaced by a purity assumption. This is useful because for many compatible systems arising from geometry purity is known by Deligne's theorem, but irreducibility can be hard to check. In particular one can deduce the meromorphic continuation and functional equation of the L-function of any regular, pure, self-dual motive over a totally real field. They also prove a version of theorem \ref{thmd} in which the hypothesis that $\pi_\infty$ is `extremely regular' is weakened to `regular', but the conclusion is also weakened to give irreducibility only above a set of rational primes of positive Dirichlet density.

We now explain the structure of the paper. In section \ref{sec: local prerequisites} we collect some 
results about the deformation theory of Galois representations. These are mostly now fairly standard results but we recall them to fix notations and in some cases to make slight improvements. The main exception is the introduction of potential diagonalizability in section \ref{l=p}, which is new and of key importance for us. In section \ref{RAECSDC} we fix some notations and we recall the existing automorphy lifting theorems (or slight generalizations of them). Very little in this section is novel. Between the writing of the first and second versions of this paper, Jack Thorne \cite{jack} has found improved versions of these theorems which allow one to remove the troublesome `bigness' conditions from \cite{cht} and the papers that followed it. Moreover Ana Caraiani \cite{ana}, \cite{ana2} has proved local-global compatibility in all $l \neq p$ cases, as well as proving temperedness of all regular algebraic, polarizable, cuspidal automorphic representations, and the purity of all the Weil-Deligne representations associated to the $l$-adic representation associated to automorphic representations including the $l=p$ case. We have taken advantage of Caraiani's and Thorne's works to optimize our own results. In section \ref{sec:potential automorphy} we make use of the automorphy lifting theorems from section \ref{RAECSDC} and the Dwork family to prove a potential automorphy theorem (in the ordinary case) and a theorem about lifting mod $l$ Galois representations (again in the ordinary case). These arguments follow those of \cite{blght} and will not surprise an expert. 

In section \ref{alt2} we prove our main new theorems. Section \ref{prelalt} contains our main new argument. 
In section \ref{malt} we combine this with the results of sections \ref{RAECSDC} and \ref{sec:potential automorphy} to obtain our optimal automorphy lifting theorem. In section \ref{lgr2} we use the same ideas to deduce an improved result about the existence of $l$-adic lifts of mod $l$ Galois representations with specified local behavior. Combining the results of sections \ref{malt} and \ref{lgr2} we deduce in section \ref{sec:cwl} a general theorem about change of weight and level for mod $l$ automorphic forms on $\GL_n$. Then in section \ref{pa1} we use the automorphy lifting theorem of section \ref{malt} and our potential automorphy theorem 
from section \ref{poa} to deduce our main new potential automorphy result for a single $l$-adic representation. 

In section \ref{sec:main result} we turn to applications of our main results. In section \ref{cs} we recall definitions connected to compatible systems of $l$-adic representations. In sections \ref{rcs} and \ref{csl} we prove some group theoretic lemmas about the images of compatible systems of $l$-adic representations. Then in section \ref{pa2} we deduce from the potential automorphy theorem of section \ref{pa1} our main theorem -- a potential automorphy theorem for compatible systems of $l$-adic representations. Finally in section \ref{ir} we give further applications of our main results -- applications to fitting an $l$-adic representation into a compatible system and to the irreducibility of some $l$-adic representations associated to cusp forms on $\GL(n)$.

In the appendix we record some miscellaneous results which we use elsewhere in the paper. Some of these are results we suspect are `well known', but for which we couldn't find a reference. In these cases we give a proof. Others are results for which we know a reference, but which we hope it may assist the reader to recall here.
 \vspace{5mm}

 We would like to thank the anonymous referees for a thorough and intelligent reading of our paper,
 and for the numerous helpful suggestions they made to improve the
 exposition. We would also like to thank Kevin Buzzard, Florian
 Herzig, Wansu Kim and James Newton for their comments on an earlier
 draft of the paper. We are grateful to Ana Caraiani and Jack Thorne
 for sharing with us early drafts of their papers \cite{ana},  \cite{ana2} and
 \cite{jack}; and to Brian Conrad and Jiu-Kang Yu for answering our questions about commutative algebra and hyper-special maximal compact subgroups, respectively.  This paper was written at the same time as
 \cite{blggord} and there was considerable cross fertilization.  Our
 paper would have been impossible without Harris' tensor product trick
 and it is a pleasure to acknowledge our debt to him.

\subsection*{Notation}
We write all matrix transposes on the left; so ${}^t\!A$ is the
transpose of $A$. Let $\gog\gl_n$ denote the space of $n\times n$ matrices with the
adjoint action of $\GL_n$ and let $\gs\gl_n$ denote the subspace of
trace zero matrices. We will write $O(n)$ (resp.\ $U(n)$) for the group of matrices $g \in \GL_n(\R)$ (resp.\ $\GL_n(\C)$) such that ${}^tg^cg=1_n$.

If $R$ is a local ring we write $\mf{m}_{R}$ for the
maximal ideal of $R$. 

If $\Delta$ is an abelian group we will let $\Delta^\tor$ denote its maximal torsion subgroup and $\Delta^\tf$ its maximal torsion free quotient.
If $\Gamma$ is a profinite group then
$\Gamma^\ab$ will denote its maximal abelian quotient by a closed
subgroup. If $\rho:\Gamma \to \GL_n(\barQQ_l)$ is a continuous
homomorphism then we will let $\barrho:\Gamma \ra \GL_n(\barFF_l)$
denote the {\em semi-simplification} of its reduction, which is well defined
up to conjugacy.

If $M$ is a field, we let $\barM$ denote an algebraic closure of $M$
and $G_M$ the absolute Galois group $\Gal(\barM/M)$.  We will use
$\zeta_n$ to denote a primitive $n^{th}$-root of $1$.  Let
$\epsilon_l$ denote the $l$-adic cyclotomic character and
$\barepsilon_l$ its reduction modulo $l$. We will also let
$\omega_l:G_M \ra \mu_{l-1} \subset \Z_l^\times$ denote the
Teichmuller lift of $\barepsilon_l$.  If $N/M$ is a separable
quadratic extension we will let $\delta_{N/M}$ denote the non-trivial
character of $\Gal(N/M)$. 

If $\Gamma$ is a profinite group and $M$ is a topological abelian group with a  continuous action of $\Gamma$, then by $H^i(\Gamma,M)$ we shall mean the continuous cohomology. 

We will write $\Q_{l^r}$ for the unique unramified extension of $\Q_l$ of degree $r$ and $\Z_{l^r}$ for its ring of integers. We will write $\Q_l^\nr$ for the maximal unramified extension of $\Q_l$ and $\Z_l^\nr$ for its ring of integers. We will also write $\hatZZ_l^\nr$ for the $l$-adic completion of $\Z_l^\nr$ and $\hatQQ_l^\nr$ for its field of fractions.

If $K$ is a finite
extension of $\bb{Q}_p$ for some $p$, we write $K^\nr$ for its maximal unramified extension; $I_K$ for the inertia
subgroup of $G_K$; $\Frob_K \in G_K/I_K$ for the geometric Frobenius; and
$W_K$ for the Weil group. If $K'/K$ is a Galois extension we will write $I_{K'/K}$ for the inertia subgroup of $\Gal(K'/K)$.
We will write $\Art_K:K^\times \iso W_K^\ab$ for the Artin map normalized to send uniformizers to geometric Frobenius elements. We will let $\rec_K$ be the local Langlands correspondence of
\cite{ht}, so that if $\pi$ is an irreducible complex
admissible representation of $\GL_n(K)$, then $\rec_K(\pi)$ is a
Frobenius semi-simple Weil--Deligne representation of the Weil group $W_K$. 
We will write $\rec$ for $\rec_K$
when the choice of $K$ is clear.
If $(r,N)$ is a Weil--Deligne representation of $W_K$ we will write $(r,N)^{F-\semis}$ for its Frobenius semisimplification.
If $\rho$ is a continuous representation of $G_K$ over $\barQQ_l$ with $l\neq p$ then we will write $\WD(\rho)$ for the corresponding Weil--Deligne representation of $W_K$. (See for instance section 1
of \cite{ty}.) By a {\em Steinberg} representation of $\GL_n(K)$ we will mean a representation $\Spp_n(\psi)$  (in the notation of section 1.3 of \cite{ht}) where $\psi$ is an unramified character of $K^\times$.
If $\pi_i$ is an irreducible smooth representation of $\GL_{n_i}(K)$ for $i=1,2$ we will write $\pi_1 \boxplus \pi_2$ for the irreducible smooth representation of $\GL_{n_1+n_2}(K)$ with $\rec(\pi_1 \boxplus \pi_2)=\rec(\pi_1) \oplus \rec(\pi_2)$.
If $K'/K$ is a finite extension and if $\pi$ is an irreducible smooth representation of $\GL_n(K)$ we will write $\BC_{K'/K}(\pi)$ for the base change of $\pi$ to $K'$ which is characterized by $\rec_{K'}(\BC_{K'/K}(\pi))=
\rec_K(\pi)|_{W_{K'}}$.

If $\rho$ is a continuous de Rham representation of $G_K$ over $\barQQ_l$ then we will write $\WD(\rho)$ for the corresponding Weil--Deligne representation of $W_K$, and if $\tau:K \into \barQQ_l$ is a
continuous embedding of fields then we will write $\HT_\tau(\rho)$ for the multiset of Hodge--Tate numbers of $\rho$ with respect to $\tau$. Thus $\HT_\tau(\rho)$ is a multiset of $\dim \rho$ integers. 
 In fact if $W$ is a de Rham representation of $G_K$ over $\barQQ_l$ and if $\tau:K \into \barQQ_l$ then the multiset $\HT_\tau(W)$ contains
$i$ with multiplicity $\dim_{\barQQ_l} (W \otimes_{\tau,K} \widehat{\barK}(i))^{G_K} $. Thus for example
$\HT_\tau(\epsilon_l)=\{ -1\}$. 

We will let $c$ denote complex conjugation on $\C$. We will write $\Art_\R$ (resp.\ $\Art_\C$) for the unique continuous surjection
\[ \R^\times \onto \Gal(\C/\R) \]
(resp.\ $\C^\times \onto \Gal(\C/\C)$).
We will write $\rec_\C$ (resp.\ $\rec_\R$), or simply $\rec$, for the local Langlands correspondence from irreducible admissible $(\Lie GL_n(\R) \otimes_\R \C, O(n))$-modules (resp.\ $(\Lie GL_n(\C) \otimes_\R \C, U(n))$-modules) to continuous, semi-simple $n$-dimensional representations of the Weil group $W_\R$ (resp.\ $W_\C$). (See \cite{langlandsrg}.) If $\pi_i$ is an irreducible admissible $(\Lie GL_{n_i}(\R) \otimes_\R \C, O(n_i))$-module (resp.\ $(\Lie GL_{n_i}(\C) \otimes_\R \C, U(n_i))$-module) for $i=1,\dots,r$ and if $n=n_1+\dots+n_r$, then we define an irreducible admissible $(\Lie GL_n(\R) \otimes_\R \C, O(n))$-module (resp.\ $(\Lie GL_n(\C) \otimes_\R \C, U(n))$-module) $\pi_1 \boxplus \cdots \boxplus \pi_r$ by
\[ \rec(\pi_1 \boxplus \cdots \boxplus \pi_r) = \rec(\pi_1) \oplus \cdots \oplus \rec(\pi_r). \]
If $\pi$ is an irreducible admissible $(\Lie GL_{n}(\R) \otimes_\R \C, O(n))$-module then we define $\BC_{\C/\R}(\pi)$ to be the irreducible admissible $(\Lie GL_{n}(\C) \otimes_\R \C, U(n))$-module defined by
\[ \rec_\C(\BC_{\C/\R}(\pi))=\rec_\R(\pi)|_{W_\C}. \]

We will write $||\,\,\,||$ for the continuous homomorphism
\[ ||\,\,\,||=\prod_v |\,\,\,|_v: \A^\times/\Q^\times \lra \R^\times_{>0}, \]
where each $|\,\,\,|_v$ has its usual normalization, i.e.\ $|p|_p=1/p$. 

Now suppose that $K/\Q$ is a finite extension. 
We will write $||\,\,\,||_K$ (or simply $||\,\,\,||$) for $||\,\,\, || \circ \norm_{K/\Q}$. We will also write
\[ \Art_K = \prod_v \Art_{K_v}:\A_K^\times /\overline{K^\times (K_\infty^\times)^0} \liso G_K^\ab. \]
If $v$ is a finite place of $K$ we will write $k(v)$ for its residue field, $q_v$ for $\# k(v)$, and $\Frob_v$ for $\Frob_{K_v}$. If $v$ is a real place of $K$ then we will let
$[c_v]$ denote the conjugacy class in $G_K$ consisting of complex conjugations associated to $v$. 
If $K'/K$ is a quadratic extension of number fields we will denote by $\delta_{K'/K}$ the nontrivial
character of $\A_K^\times/K^\times \norm_{K'/K}\A_{K'}^\times$. (We hope that this will cause no confusion with the Galois character $\delta_{K'/K}$. One equals the composition of the other with the Artin map for $K$.) 
If $K'/K$ is a soluble, finite Galois extension and if $\pi$ is a
cuspidal automorphic representation of $\GL_n(\A_K)$ we will write
$\BC_{K'/K}(\pi)$ for its base change to $K'$, an (isobaric) automorphic representation
of $\GL_n(\A_{K'})$ satisfying 
\[ \BC_{K'/K}(\pi)_v=\BC_{K'_v/K_{v|_K}}(\pi_{v|_K}) \]
for all places $v$ of $K'$. If $\pi_i$ is an automorphic representation of $\GL_{n_i}(\A_K)$ for $i=1,2$ we will write $\pi_1 \boxplus \pi_2$ for the automorphic representation of $\GL_{n_1+n_2}(\A_K)$ satisfying
\[ (\pi_1 \boxplus \pi_2)_v=\pi_{1,v} \boxplus \pi_{2,v} \]
for all places $v$ of $K$.

We will call a number field $K$ a CM field if it has an automorphism
$c$ such that for all embeddings $i:K \into \C$ one has $c \circ i = i
\circ c$. In this case either $K$ is totally real, or a totally
imaginary quadratic extension of a totally real field. In either case
we will let $K^+$ denote the maximal totally real subfield of $K$.

Suppose that $K$ is a number field and 
\[ \chi: \A_K^\times /K^\times \lra \C^\times \]
is a continuous character. If there exists $a \in \Z^{\Hom(K,\C)}$ such that
\[ \chi|_{(K_\infty^\times)^0}: x \longmapsto \prod_{\tau \in \Hom(K,\C)} (\tau x)^{a_\tau} \]
 we will call $\chi$ algebraic. In this case we can attach to $\chi$ and a rational prime $l$ and an isomorphism $\imath:\barQQ_l \iso \C$, a unique continuous character
 \[ r_{l,\imath}(\chi): G_K \lra \barQQ_l^\times \]
 such that for all $v\ndiv l$ we have
 \[ \imath \circ r_{l,\imath}(\chi)|_{W_{K_v}} \circ \Art_{K_v} = \chi_v. \] 
 There is also an integer $\wt(\chi)$, the weight of $\chi$, such that
 \[ |\chi|= ||\,\,\,||_K^{-\wt(\chi)/2}. \]
 (See the discussion at the start of Section \ref{sbuild} for more details.)
 
 If $F$ is a totally real field we call a continuous character 
 \[ \chi:\A_K^\times/K^\times \lra \C^\times \]
 totally odd if $\chi_v(-1)=-1$ for all $v|\infty$. Similarly we call a continuous character
 \[ \mu:G_K \lra \barQQ_l^\times \]
totally odd if $\mu(c_v)=-1$ for all $v|\infty$.
 \newpage

\tableofcontents

\section{Deformations of Galois Representations.}\label{sec: local prerequisites}

\subsection{The group $\CG_n$.}\label{cgn}{$\mbox{}$} \newline

We let $\CG_n$ denote the semi-direct product of $\CG_n^0=\GL_n \times \GL_1$ by the group $\{1 , \jmath\}$
where 
\[ \jmath (g,a) \jmath^{-1}=(a{}^tg^{-1},a). \]
We let $\nu:\CG_n \ra \GL_1$ be the character which sends $(g,a)$ to $a$ and sends $\jmath$ to $-1$. We will also let $\GSp_{2n}\subset \GL_{2n}$ denote the symplectic similitude group defined by 
the anti-symmetric matrix
\[ J_{2n}=\mat{0}{1_n}{-1_n}{0}, \]
and we will again let $\nu:\GSp_{2n} \ra \GL_1$ denote the multiplier character. Finally let $\GO_n$ denote the orthogonal similitude group defined by the symmetric matrix $1_n$. 

There is a natural homomorphism
\[ \CG_n \times \CG_m \lra \CG_n/\CG_n^0 \times \CG_m/\CG_m^0 \lra \{ \pm 1\} \]
which sends both $(\jmath,1)$ and $(1,\jmath)$ to $-1$. Let $(\CG_n \times \CG_m)^+$ denote the kernel of this map. There is a homomorphism
\[ \begin{array}{rcl} \otimes: (\CG_n \times \CG_m)^+ &\lra &\CG_{nm} \\ (g,a) \times (g',a') & \longmapsto
& (g \otimes g', aa') \\ \jmath \times \jmath & \longmapsto & \jmath. \end{array} \]
There is also a homomorphism
\[ \begin{array}{rcl} I: \CG_n &\lra &\GSp_{2n} \\ (g,a)& \longmapsto
& \mat{g}{0}{0}{a{}^tg^{-1}} \\ \jmath & \longmapsto & \mat{0}{1_n}{1_n}{0}. \end{array} \]

Suppose that $\Gamma$ is a group with a normal subgroup $\Delta$ of index $2$ and that $\gamma_0 \in \Gamma -\Delta$. Suppose also that $A$ is a ring and that $r:\Gamma \ra \CG_n(A)$ is a homomorphism with $\Delta=r^{-1}\CG_n^0(A)$. Write $\breve{r}:\Delta \ra \GL_n(A)$ for the composition 
of $r|_\Delta$ with projection to $\GL_n(A)$. Write $r(\gamma_0)=(a,-(\nu \circ r)(\gamma_0)) \jmath$. Then
\[ \breve{r}(\gamma_0\delta\gamma_0^{-1}) a {}^t\breve{r}(\delta)=(\nu\circ r)(\delta)a \]
for all $\delta \in \Delta$, and 
\[ \breve{r}(\gamma_0^2){}^ta=-(\nu \circ r)(\gamma_0)a. \]

If $r:\Gamma \ra \GSp_{2n}(A)$ is a homomorphism with multiplier $\mu$, then it gives rise to a homomorphism $\hat{r}_\Delta:\Gamma
\ra \CG_{2n}(A)$ which sends $\delta \in \Delta$ to $(r(\delta), \mu(\delta))$ and $\gamma\in \Gamma-\Delta$ to $(r(\gamma)J_{2n}^{-1},-\mu(\gamma))\jmath$. Then $\Delta = \hat{r}_\Delta^{-1}\CG^0_{2n}(A)$ and $\nu \circ \hat{r}_\Delta=\mu$. Similarly if $r:\Gamma \ra \GO_n(A)$ is a homomorphism with multiplier $\mu$, then it gives rise to a homomorphism $\hat{r}_\Delta:\Gamma
\ra \CG_n(A)$ which sends $\delta \in \Delta$ to $(r(\delta), \mu(\delta))$ and $\gamma\in \Gamma-\Delta$ to $(r(\gamma),\mu(\gamma))\jmath$. Then $\Delta = \hat{r}_\Delta^{-1}\CG^0_{2n}(A)$ and $\nu \circ \hat{r}_\Delta$ equals the product of $\mu$ with the nontrivial character of $\Gamma/\Delta$. 

If $r:\Gamma \ra \CG_n(A)$ (resp.\  $r':\Gamma \ra \CG_m(A)$) is a homomorphism with $r^{-1}\CG_n^0(A)=\Delta$ (resp.\ $(r')^{-1}\CG_m^0(A)=\Delta$) then we define
\[ I(r)=I \circ r : \Gamma \lra \GSp_{2n}(A) \]
and
\[ r \otimes r'=\otimes \circ (r \times r'): \Gamma \lra \CG_{nm}(A). \]
Note that the multiplier of $I(r)$ equals the multiplier of $r$ and that the multiplier of $r \otimes r'$ differs from the product of the multipliers of $r$ and $r'$ by the non-trivial character of $\Gamma/\Delta$.

If $\chi:\Delta \ra A^\times$ and $\mu:\Gamma \ra A^\times$ satisfy 
\begin{itemize}
\item $\chi \chi^{\gamma_0}=\mu|_{\Delta}$, and
\item $\chi(\gamma_0^2)=-\mu(\gamma_0)$;
\end{itemize}
(i.e.\ the composition of $\chi$ with the transfer map $\Gamma^\ab \ra \Delta^\ab$ equals the product of $\mu$ and the non-trivial character of $\Gamma/\Delta$),
then there is a homomorphism
\[ \begin{array}{rcl}  (\chi,\mu): \Gamma & \lra & \CG_1(A) \\ \delta & \longmapsto & (\chi(\delta), \mu(\delta)) \\ \gamma & \longmapsto & (\chi(\gamma \gamma_0^{-1}),-\mu(\gamma)) \jmath, \end{array} \]
for all $\delta \in \Delta$ and $\gamma \in \Gamma-\Delta$. We have $\nu \circ (\chi,\mu)=\mu$.

(At the referee's suggestion we include a proof that $(\chi,\mu)$ is indeed a
homomorphism in an attempt to convince the reader that all the unsupported assertions
of this section can be checked in an entirely elementary way. Suppose that $\delta_1,\delta_2 \in \Delta$ and $\gamma_1,\gamma_2 \in \Gamma-\Delta$. Then we have
\[ (\chi,\mu)(\delta_1\delta_2)=(\chi,\mu)(\delta_1)(\chi,\mu)(\delta_2) \]
and
\[ \begin{array}{rcl} (\chi,\mu)(\delta_1\gamma_2)&=&(\chi(\delta_1\gamma_2\gamma_0^{-1}),-\mu(\delta_1\gamma_2))\jmath \\ &=&(\chi,\mu)(\delta_1)(\chi,\mu)(\gamma_2) \end{array} \]
and
\[ \begin{array}{rcl} (\chi,\mu)(\gamma_1\delta_2)&=&(\chi(\gamma_1\delta_2\gamma_0^{-1}),-\mu(\gamma_1\delta_2))\jmath\\ &=&(\chi(\gamma_1\gamma_0^{-1}(\gamma_0\delta_2\gamma_0^{-1})),-\mu(\gamma_1\delta_2))\jmath\\ &=&(\chi,\mu)(\gamma_1)\jmath(\chi(\gamma_0\delta_2\gamma_0^{-1}),\mu(\delta_2)) \jmath \\ &=&(\chi,\mu)(\gamma_1) (\mu(\delta_2)\chi^{\gamma_0}(\delta_2)^{-1},\mu(\delta_2)) \\ &=&
(\chi,\mu)(\gamma_1)(\chi,\mu)(\delta_2) \end{array} \]
and
\[ \begin{array}{rcl} (\chi,\mu)(\gamma_1\gamma_2)&=&(\chi(\gamma_1\gamma_0^{-1})\chi(\gamma_0 \gamma_2),\mu(\gamma_1)\mu(\gamma_2))\\ &=&(\chi,\mu)(\gamma_1)\jmath (\chi^{\gamma_0}(\gamma_2\gamma_0),-\mu(\gamma_2)) \\ &=& (\chi,\mu)(\gamma_1)(-\mu(\gamma_2)\chi^{\gamma_0}(\gamma_2\gamma_0)^{-1},-\mu(\gamma_2))\jmath \\
&=& (\chi,\mu)(\gamma_1)(-\chi(\gamma_2\gamma_0)\mu(\gamma_0)^{-1},-\mu(\gamma_2))\jmath \\ &=& (\chi,\mu)(\gamma_1)(\chi,\mu)(\gamma_2) \end{array} \]
as desired.)

In the case that $\Gamma=G_{F^+}$ and $\Delta=G_F$ where $F$ is an imaginary CM field with maximal totally
real subfield $F^+$, we call $r:G_{F^+} \ra \CG_n(A)$ (resp.\ $\GSp_{2n}(A)$, resp.\ $\GO_n(A)$) {\em totally odd}
if the multiplier character takes every complex conjugation to $-1$ (resp.\ $-1$, resp.\ $1$). Note that if
$r$ is totally odd so is $I(r)$ (resp.\ $\hat{r}_\Delta$, resp.\ $\hat{r}_\Delta$).

Suppose now that $A$ is a field, that $r:\Delta \ra \GL_n(A)$ is absolutely irreducible, and that
$\mu:\Gamma \ra A^\times$ is a character so that
\[ r^{\gamma_0} \cong r^\vee \otimes \mu|_{\Delta}. \]
More precisely if $\gamma \in \Gamma-\Delta$ there is a $b_\gamma \in \GL_n(A)$, unique up to scalar multiples, such that 
\[ r(\gamma\delta\gamma^{-1}) b_\gamma {}^tr(\delta)=\mu(\delta)b_\gamma \]
for all $\delta \in \Delta$. Computing $r(\gamma^2\delta\gamma^{-2})$ in two ways and using the absolute irreducibility of $r$, we deduce that $r(\gamma^2)$ is a scalar multiple
of $b_\gamma{}^tb_\gamma^{-1}$. (Write
$r(\gamma^2)r(\delta)r(\gamma^2)^{-1}=r(\gamma^2\delta\gamma^{-2})=\mu(\delta)b_\gamma{}^tr(\gamma\delta\gamma^{-1})^{-1}b_\gamma^{-1}=(b_\gamma{}^tb_\gamma^{-1})r(\delta)(b_\gamma{}^tb_\gamma^{-1})^{-1}$
and apply Schur's lemma.) Substituting $\delta=\gamma^2$ in the last displayed equation, we then deduce that 
\[ r(\gamma^2){}^tb_\gamma=\pm \mu(\gamma)b_\gamma. \]
One can check that the sign in the above equation is independent of $\gamma \in \Gamma-\Delta$, and we will denote it $-\sgn(r,\mu)$. (To see this one uses the fact that one can take $b_{\delta\gamma} 
=r(\delta)b_\gamma$ for $\delta \in \Delta$ and $\gamma\in \Gamma-\Delta$.) 
Then we get a homomorphism
\[ \tilde{r}_\mu:\Gamma \lra \CG_n(A) \]
which sends $\delta\in \Delta$ to $(r(\delta), \mu(\delta))$ and sends $\gamma_0$ to $(b_{\gamma_0},-\sgn(r,\mu) \mu(\gamma_0))\jmath$. In particular if $\sgn(r,\mu)=1$ then
$\nu \circ \tilde{r}_\mu=\mu$, while if $\sgn(r,\mu)=-1$ then $\mu^{-1} (\nu \circ \tilde{r}_\mu)$ is the non-trivial character of $\Gamma/\Delta$. Moreover $\breve{\tilde{r}}=r$.

\subsection{Abstract deformation theory.}\label{adt}{$\mbox{}$} \newline

Fix a rational prime $l$ and let $\CO$ denote the ring of integers 
of a finite extension $L$ of $\Q_l$ in $\barQQ_l$. Let $\lambda$ denote the maximal ideal of $\CO$ and
let $\F=\CO/\lambda$. Let $\Gamma$ denote a topologically finitely generated profinite group and let $\barrho:\Gamma \ra \GL_n(\F)$ be a continuous homomorphism.

We will denote by 
\[ \rho^\Box=\rho^\Box_\CO:\Gamma \ra \GL_n(R_{\CO,\barrho}^\Box) \]
the universal lifting (or `framed deformation') of $\barrho$ to a complete noetherian local $\CO$-algebra with residue field $\F$. (We impose no equivalence condition on lifts other than equality.) We will write
\[ R_{\barrho}^\Box \otimes \barQQ_l \]
for $R_{\CO,\barrho}^\Box \otimes_{\CO} \barQQ_l$.
The following lemma is presumably well known, but as we don't know a reference, we give a proof.
\begin{lem}\label{coeffchange} Suppose that $\CO'$ is the ring of integers of a finite extension $L'/L$. Then the map
\[ R_{\CO',\barrho}^\Box \lra R_{\CO,\barrho}^\Box \otimes_{\CO} \CO' \]
coming from the universal property of $R_{\CO',\barrho}^\Box$ and $\rho^\Box_\CO \otimes \CO'$ is an isomorphism. In particular, as the notation suggests, the ring $R_\barrho^\Box \otimes \barQQ_l$ does not depend on the choice of $L$.
\end{lem} 

\begin{proof} Let $R_{\CO',\F,\barrho}^\Box$ denote the closed subring of $R_{\CO',\barrho}^\Box$ consisting of elements which reduce to an element of $\F$ modulo the maximal ideal. Then $R_{\CO',\F,\barrho}^\Box$ is a complete, noetherian local $\CO$-algebra with residue field $\F$ and $\rho^\Box_{\CO'}$ is defined over $R_{\CO',\F,\barrho}^\Box$. Thus the universal property of $R^\Box_{\CO,\barrho}$ gives rise to a map $R^\Box_{\CO,\barrho} \ra R_{\CO',\F,\barrho}^\Box$ under which $\rho^\Box_\CO$ pushes forward to $\rho^\Box_{\CO'}$. This map extends to a $\CO'$-linear map
$R^\Box_{\CO,\barrho}\otimes_\CO \CO' \ra R_{\CO',\barrho}^\Box$. We claim this is an inverse to our map $R_{\CO',\barrho}^\Box \lra R_{\CO,\barrho}^\Box \otimes_{\CO} \CO'$. 

Under the composite $R_{\CO',\barrho}^\Box \lra R_{\CO',\barrho}^\Box$ the representation $\rho^\Box_{\CO'}$ pushes forward to itself, and so this map must be the identity. Consider the composite
\[ R^\Box_{\CO,\barrho} \lra R^\Box_{\CO',\barrho} \lra R^\Box_{\CO,\barrho} \otimes_\CO \CO'. \]
It factors through the subring $(R^\Box_{\CO,\barrho} \otimes_\CO \CO')_\F \subset R^\Box_{\CO,\barrho} \otimes_\CO \CO'$ consisting of elements which reduce modulo the maximal ideal to an element of $\F$. The representation $\rho^\Box_\CO$ pushes forward to itself, and so this composite must equal the canonical inclusion, and we have proved the lemma.
\end{proof}

The maximal ideals are dense in $\Spec R^\Box_{\CO,\barrho}[1/l]$ (see
\cite[10.5.7]{MR0217086}). A
prime ideal $\wp$ of $R^\Box_{\CO,\barrho}[1/l]$ is maximal if and
only if the residue field $k(\wp)=R^\Box_{\CO,\barrho}[1/l]_\wp/\wp$ is (topologically isomorphic to) a
finite extension of $L$. (For the `if' part note that the image of $R^\Box_{\CO,\barrho}$ in $k(\wp)$ is a compact $\CO$-submodule of $k(\wp)$ with field of fractions $k(\wp)$. Thus $R^\Box_{\CO,\barrho}[1/l] \onto k(\wp)$. For the `only if' part see for instance Lemma 2.6 of \cite{tay}.) We get a continuous representation
$\rho_\wp:\Gamma \ra \GL_n(k(\wp))$. The formal completion
$R_{\CO,\barrho}^\Box[1/l]_\wp^\wedge$ is the universal lifting ring
for $\rho_\wp$, i.e.\ if $A$ is an Artinian local $k(\wp)$-algebra with
residue field $k(\wp)$ and if $\rho:\Gamma \ra \GL_n(A)$ is a
continuous representation lifting $\rho_\wp$, then there is a unique
continuous map of $k(\wp)$-algebras
$R_{\CO,\barrho}^\Box[1/l]_\wp^\wedge \ra A$ so that $\rho^\Box$
pushes forward to $\rho$. (Let $R$ denote the image of
$R^\Box_{\CO,\barrho}$ in $A/\gm_A$. Let $A^0$ denote the
$R$-subalgebra of $A$ generated by the matrix entries of the image of
$\rho$. Then $A^0$ is a complete noetherian local $\CO$-algebra with
residue field $\F$ and $\rho:\Gamma \ra \GL_n(A^0)$. The assertion
follows easily.) The map which associates a
cocycle $c\in Z^1(\Gamma, \ad \rho_\wp)$ to the lift $(1_n +c
\epsilon)\rho_{\wp}$ of $\rho_{\wp}$ defines an isomorphism between
$Z^1(\Gamma, \ad \rho_\wp)$ and the tangent space
$\Hom_{k(\wp)}(R_{\CO,\barrho}^\Box[1/l]_\wp^\wedge,k(\wp)[\epsilon]/\epsilon^2)$.
If $H^2(\Gamma, \ad
\rho_\wp)=(0)$ then $R_{\CO,\barrho}^\Box[1/l]$ is formally smooth at
$\wp$ (by the argument of Proposition 2 of \cite{MR1012172}, which can
be easily adapted to our current situation) of dimension 
\[ \dim_{k(\wp)} Z^1(\Gamma, \ad \rho_\wp) = n^2+\dim_{k(\wp)} H^1(\Gamma, \ad
\rho_\wp)-\dim_{k(\wp)} H^0(\Gamma, \ad \rho_\wp), \]
where we use continuous cohomology. (We learned these observations from Mark Kisin.)

Let $H$ denote the subgroup of $\GL_n(R_{\CO,\barrho}^\Box)$ consisting of elements which reduce modulo the maximal ideal to an element that centralizes the image of $\barrho$. If $h \in H$ then there is a unique continuous homomorphism
\[ \phi_h: R_{\CO,\barrho}^\Box \ra R_{\CO,\barrho}^\Box \]
such that $\rho^\square$ pushes forward to $h^{-1} \rho^\square h$. We have that
\[ \phi_g \circ \phi_h = \phi_{g\phi_g(h)}. \]
(We remark that after definition 2.2.1 of \cite{cht} this action is defined for elements of $1_n+M_n(\gm_{R_{\CO,\barrho}^\Box})$, but it is incorrectly stated that this defines an action of the group $1_n+M_n(\gm_{R_{\CO,\barrho}^\Box})$ on $R_{\CO,\barrho}^\Box$. This is not important in the rest of \cite{cht}.)

\begin{lem}\label{conj} Keep the notation of the previous paragraph. If $h \in H$ then $\phi_h$ fixes each irreducible component of $\Spec
R_{\CO,\barrho}^\Box[1/l]$. \end{lem}

\begin{proof} Suppose that $h \in H$ and that $\CO'$ is the ring of integers of a
finite extension of $L'/L$ in $\barQQ_l$, with maximal ideal $\lambda'$. Suppose
that $\rho:\Gamma \ra \GL_n(\CO')$ lifts $\barrho$.  Then $h^{-1}\rho
h$ also lifts $\barrho$. (We are using $h$ both for an element of $\GL_n(R_{\CO,\barrho}^\Box)$ and for its image in $\GL_n(\CO')$ under the map $R_{\CO,\barrho}^\Box \ra \CO'$ induced by $\rho$.) Recall (from Lemma \ref{ID1}) that $A=\CO'\langle s,t\rangle/(s\det(t1_n+(1-t)h) -1)$ with the $\lambda'$-adic topology is a
complete topological domain. We have
a continuous representation
\[ \trho= (t1_n+(1-t)h)^{-1} \rho (t1_n+(1-t)h): \Gamma \lra \GL_n(A). \]

Let $A^0$ denote the closed subalgebra of $A$ generated by the matrix entries of elements of
the image of $\trho$ and give it the subspace topology. Then 
\[ \trho: \Gamma \lra \GL_n(A^0), \]
and $\trho \bmod (\lambda'A \cap A^0)=\barrho$, and $\trho$ pushes forward to $\rho$ (resp.\ $h^{-1}\rho h$) under the continuous homomorphism
 $A^0 \ra \CO'$ induced by $t \mapsto 1$ (resp.\ $t \mapsto 0$).
We will show that  $A^0$ is a complete, noetherian local $\CO$-algebra
with residue field $\F$. It will follow that there is a natural map $R_{\CO,\barrho}^\Box \ra A^0$ through
which the maps $R_{\CO,\barrho}^\Box \ra \CO'$ corresponding to $\rho$ and $h^{-1}\rho h$ both factor.
As $A^0$ is a domain (being a sub-ring of $A$) we conclude that the points corresponding to
$\rho$ and $h^{-1}\rho h$ lie on the same irreducible component of $\Spec R^\Box_{\CO,\barrho}[1/l]$.
As any irreducible component contains an $\CO'$-point which lies on no other irreducible component
for some $\CO'$ as above (because such points are Zariski dense in
$\Spec R^\Box_{\CO,\barrho}[1/l]$), we
see that the lemma follows. (The referee remarks that it may be
helpful to think of this argument as an instance of homotopy.)

It remains to show that $A^0$ is a complete, noetherian, local $\CO$-algebra with residue field $\F$.
Let $\gamma_1,\dots,\gamma_r$ denote topological generators of $\Gamma$. Write 
\[ \trho(\gamma_i)=a_i+b_i, \] where $a_i \in \GL_n(\CO)$ lifts
$\barrho(\gamma_i)$ and where $b_i \in M_{n \times n}(\lambda' A)$.
Then $A^0$ is the closure of the $\CO$-subalgebra of $A$ generated by
the entries of the $b_i$.  As these entries are topologically nilpotent
in $A$ we get a continuous $\CO$-algebra homomorphism
\[ \CO[[X_{ijk}]]_{i=1,\dots,r; \,\, j,k=1,\dots,n} \lra A \]
which sends $X_{ijk}$ to the $(j,k)$-entry of $b_i$. Let $J$ denote the kernel. As $\CO[[X_{ijk}]]/J$
is compact (and $A$ is Hausdorff) this map is a topological isomorphism of $\CO[[X_{ijk}]]/J$ with its image in $A$ and this image is closed. Thus the image is just $A^0$ and we have $\CO[[X_{ijk}]]/J \iso
A^0$, so that $A^0$ is indeed a complete, noetherian, local $\CO$-algebra with residue field $\F$.
\end{proof}

Now suppose that 
\[ \barr: \Gamma \lra \CG_n(\F) \]
is a continuous homomorphism such that $\Gamma \onto
\CG_n/\CG^0_n$. Let $\Delta$ denote the kernel of $\Gamma \onto
\CG_n/\CG_n^0$ and suppose that $\breve{\barr}: \Delta \lra \GL_n(\F)$ is absolutely irreducible.
 Then there is a universal deformation
\[ r^\univ: \Gamma \lra \CG_n(R^\univ_{\CO,\barr}) \]
to a complete noetherian local $\CO$-algebra with residue field $\F$, 
where now we consider two liftings as equivalent deformations if they are conjugate. (See section 2.2 of \cite{cht}.) 

\begin{lem}\label{cdr} \begin{enumerate}
\item Suppose that $\Sigma \subset \Gamma$ has finite index, that $\Sigma$ is not contained in $\Delta$ and that $\breve{\barr}|_{\Delta \cap \Sigma}$ is absolutely irreducible. Then the natural 
map $R^\univ_{\CO,\barr|_{\Sigma}} \ra R^\univ_{\CO,\barr}$ induced by $r^\univ|_\Sigma$ makes $R^\univ_{\CO,\barr}$ a finitely generated $R^\univ_{\CO,\barr|_{\Sigma}}$-module.
\item Suppose that $s:\Gamma \ra \CG_m(\CO)$ is a continuous homomorphism such that $\Delta
= s^{-1}\CG_m^0(\CO)$ and $\breve{\bars} \otimes \breve{\barr}$ is absolutely irreducible. Then
the natural map $R^\univ_{\CO,\barr\otimes \bars} \ra R^\univ_{\CO,\barr}$ induced by $r^\univ\otimes s$ makes $R^\univ_{\CO,\barr}$ a finitely generated $R^\univ_{\CO,\barr\otimes \bars}$-module.
\item Suppose $I(\barr)$ is absolutely irreducible and that $\Sigma$
  is another open subgroup of index two in $\Gamma$ which does not
  contain $\ker I(\barr)$. Then the natural map
  $R^\univ_{\CO,\widehat{I(\barr)}_{\Sigma}} \ra R^\univ_{\CO,\barr}$
  induced by $\widehat{I(r^\univ)}_\Sigma$ makes $R^\univ_{\CO,\barr}$
  a finitely generated
  $R^\univ_{\CO,\widehat{I(\barr)}_{\Sigma}}$-module.\end{enumerate} \end{lem}

\begin{proof}
  This is essentially an abstraction of Lemma 3.2.1 of
  \cite{blggord}. 
  
  Write $R$ for $R^\univ_{\CO,\barr|_{\Sigma}}$
  resp.\ $R^\univ_{\CO,\barr\otimes \bars}$
  resp.\ $R^\univ_{\CO,\widehat{I(\barr)}_{\Sigma}}$ and write $\gm$
  for the maximal ideal of $R$. We first verify  that the
  image of $\Gamma$ in $\CG_n(R^\univ_{\CO,\barr}/\gm
  R^\univ_{\CO,\barr})$ is finite. In the first case we use the inclusions
  \[ \ker (r^\univ \bmod \gm) \supset \ker (r^\univ|_\Sigma \bmod \gm)=\ker (\barr|_\Sigma); \] 
  in the second we use the inclusions
  \[ \ker( r^\univ \bmod \gm) \supset \ker( (r^\univ \otimes s) \bmod \gm) \cap \ker (s \bmod \gm) = \ker (\barr \otimes \bars) \cap \ker \bars; \]
  and in the third case the inclusions
  \[ \ker (r^\univ \bmod \gm) \supset \ker(\widehat{I(r^\univ)}_\Sigma \bmod \gm)=\ker \widehat{I(\barr)}_\Sigma. \]

  Let $m$ denote the order of the
  image of $\Gamma$ in $\CG_n(R^\univ_{\CO,\barr}/\gm
  R^\univ_{\CO,\barr})$, and let $\gamma_1,\dots,\gamma_m$ be elements of $\Gamma$ chosen so
  that their images in $\CG_n(R^\univ_{\CO,\barr}/\gm
  R^\univ_{\CO,\barr})$ exhaust the image of $\Gamma$. Let
\[ f(T) = \prod_{(\zeta_1,\dots,\zeta_n) \in \mu_m(\barFF)^n} (T-(\zeta_1+\dots+\zeta_n)) \in \F[T] \]
and let $A$ denote the maximal quotient of $\F[X_{i,j}]_{i,j=1,\dots,n}$ over which the $m^{th}$-power
of the matrix $(X_{i,j})$ is $1_n$.  If $\wp$ is a prime ideal of $A$ then all the roots of the characteristic
polynomial of $(X_{i,j})$ over $A_\wp/\wp$ are $m^{th}$ roots of unity and hence $f(\tr (X_{i,j}))=0$ in
$A/\wp \subset A_\wp/\wp$. Thus there is a positive integer $a$ such that $f(\tr (X_{i,j}))^a=0$ in $A$.
Then we get a map
\[ \begin{array}{rcl} \F[T_1,\dots,T_m]/(f(T_1)^a,\dots,f(T_m)^a) &\lra & R^\univ_{\CO,\barr}/\gm R^\univ_{\CO,\barr} \\ T_i & \longmapsto & \tr \breve{r}^\univ(\gamma_i). \end{array} \]
By Lemma 2.1.12 of \cite{cht} (which shows that $R^\univ_{\CO,\rbar}$
is topologically generated as a $\CO$-algebra by the $\tr r^\univ(\gamma_i)$) we see that this map has dense image. On the other hand the source has finite cardinality. We conclude that the map is surjective and that $R^\univ_{\CO,\barr}/\gm R^\univ_{\CO,\barr}$ is finite over $\F$. Hence by Nakayama's Lemma we conclude that that $R^\univ_{\CO,\barr}$ is finite over $R$, as desired.
\end{proof}

\subsection{Local theory: $l \neq p$.} \label{lnotp}{$\mbox{}$} \newline 
Continue to fix a rational prime $l$ and let $\CO$ denote the ring of integers 
of a finite extension $L$ of $\Q_l$ in $\barQQ_l$. Let $\lambda$ denote the maximal ideal of $\CO$ and
let $\F=\CO/\lambda$. However in this section we specialize our discussion to the case $\Gamma=G_K$, where $K/\Q_p$ is a finite extension and $p\neq l$. Thus $\barrho:G_K \ra \GL_n(\F)$ is continuous. Write $q$ for the order of the residue field of $K$.

In this case the 
tangent space to $R^\Box_{\CO,\barrho}[1/l]$ at a maximal ideal $\wp$
has dimension 
\[ \begin{array}{rl} & n^2+\dim_{k(\wp)} H^1(G_K, \ad
\rho_\wp)-\dim_{k(\wp)} H^0(G_K, \ad \rho_\wp) \\ = & n^2+\dim_{k(\wp)} H^2(G_K,\ad \rho_\wp) \\ =& n^2+\dim_{k(\wp)} H^0(G_K,(\ad \rho_\wp)(1)), \end{array} \]
by the local Euler characteristic formula for $\Q_l$-modules and local duality for $\Q_l$-modules. (The proof of
Lemma 9.7 of \cite{MR1992017} shows that the usual Euler
characteristic formula with finite coefficients implies the analogous
statement in the case where the coefficients are finite $\Q_l$-modules. Theorem 1.4.1
of \cite{MR1749177} provides a reference for local duality for
$\Q_l$-modules.) Moreover $R^\Box_{\CO,\barrho}[1/l]$ is formally smooth at a maximal ideal $\wp$ if  $H^0(G_K,(\ad \rho_\wp)(1))=(0)$.

We will call
a continuous representation $\rho:G_K \ra \GL_n(\barQQ_l)$ {\em robustly smooth} (resp.\ {\em smooth}) if $H^0(G_{K'},(\ad \rho_\wp)(1))=(0)$ for all finite extensions $K'/K$ (resp.\ for $K'=K$). Our next aim is to show that the set of closed points of 
$\Spec R_{\CO,\barrho}^\Box[1/l]$ which are robustly smooth is Zariski dense, which will imply that all irreducible components of  $\Spec R_{\CO,\barrho}^\Box[1/l]$ are generically formally smooth of dimension $n^2$. (The corresponding result for smooth points can be found in the proof of Theorem 2.1.6 of \cite{gee061} or  in \cite{suhhyun}, but our proof seems to be different even in this special case. This case is already sufficient to deduce that all irreducible components of  $\Spec R_{\CO,\barrho}^\Box[1/l]$ are generically formally smooth of dimension $n^2$.)

Define a partial order on $\barL^\times$ by $a\geq b$ if $a$ equals
$\sigma(b)\zeta q^m$ where $\sigma \in \Gal(\barL/L)$, where $\zeta$
is a root of unity and where $m \in \Z_{\geq 0}$. We will write $a
\approx b$ ($a$ `equivalent' to $b$) if $a \geq b$ and $b \geq a$; and we will write $a \sim b$ ($a$ `comparable' to $b$) 
if $a\geq b$ or $b \geq a$. These are both equivalence relations. Further we will write $a>b$ for $a \geq b$ but $a \not\approx b$. Choose $\phi \in W_K$ a lift of $\Frob_K$. If $V$ is a finite dimensional $L$-vector space with an action of $W_K$ with open kernel and if $a \in \barL^\times$ then we define $V((a))$ (resp.\ $V(a)$) to be the
$L$-subspace of $V$ such that $V((a)) \otimes_L \barL$ (resp.\ $V(a) \otimes_L  \barL$) is the sum of the $b$-generalized eigenspaces of $\phi$ in $V$ as $b$ runs over all elements of $\barL^\times$ with $a \approx b$ (resp.\ $a \sim b$). This is independent of the  choice of $\phi$. (If $\phi'$ is another choice then the actions of $\phi^m$ and $(\phi')^m$ on $V$ are equal for some $m \in \Z_{>0}$.) Thus $V(a)$ and $V((a))$ are $W_K$-invariant. We have decompositions
\[ V = \bigoplus V((a)) \]
where $a$ runs over $\barL^\times/\approx$, and
\[ V = \bigoplus V(a) \]
where $a$ runs over $\barL^\times/\sim$. We will say that $V$ has type $a$ if $V=V((a))$.

\begin{lem}\label{filter} Suppose that $(r,N)$ is a Weil--Deligne representation of $W_K$ on a finite dimensional $L$-vector space $V$. Then we can write
\[ V = \bigoplus_{i=1}^u \bigoplus_{j=1}^{s_i} V_{i,j} \]
where
\begin{itemize}
\item $V_{i,j}$ is invariant under $I_K$;
\item $N:V_{i,j} \iso V_{i,j+1}$ unless $j=s_i$ in which case $NV_{i,s_i}=(0)$;
\item $W_K V_{i,j} \subset V_{i,j} \oplus \bigoplus_{i'=1}^{i-1}
  \bigoplus_{j'} V_{i',j'}$and so we get an induced action of $W_K$ on $V_{i,j}$;
\item the action of $W_K$ on $(V_{i,j} \oplus \bigoplus_{i'=1}^{i-1}
  \bigoplus_{j'} V_{i',j'})/(\bigoplus_{i'=1}^{i-1}
  \bigoplus_{j'} V_{i',j'})$ is irreducible;
\item $V_{i,j}$ has type $a_iq^{1-j}$ for some $a_i \in \barQQ_l^\times$;
\item and if $i'<i$ then $a_i \not> a_{i'}$.\end{itemize} \end{lem}

\begin{proof}
We may suppose that $V=V(b)$ for some $b$ (because $N$ must take any $V(b)$ to itself). 
We will construct the $V_{i,j}$ by recursion on $i$. Suppose that we have constructed $V_{i,j}$ for $i<t$. Choose $a_t \in \barL^\times$ such that 
\[ \left( V/ \bigoplus_{i=1}^{t-1} \bigoplus_j V_{i,j}\right) ((a_t)) \neq (0) \]
and such that if $a>a_t$ then
\[ \left( V/ \bigoplus_{i=1}^{t-1} \bigoplus_j V_{i,j}\right) ((a))= (0) \]
Also choose an irreducible $W_K$-submodule 
\[ \barV_{t,1} \subset \left( V/ \bigoplus_{i=1}^{t-1} \bigoplus_j V_{i,j}\right) ((a_t)) \]
and choose $s_t$ minimal such that $N^{s_t}\barV_{t,1}=(0)$. Lift $\barV_{t,1}$  to an $I_K$-submodule $V_{t,1}^0 \subset V((a_t))$. Then $N^{s_t} V_{t,1}^0 \subset \bigoplus_{i=1}^{t-1} \bigoplus_j V_{i,j}$. For each $i<t$ choose $j_i \in \Z_{> s_t}$ such that $a_t q^{-s_t} \approx a_i q^{1-j_i}$. (To see that $j_i > s_t$ we are using the fact that for $i \leq t$ we have $a_i \geq a_t$.) Then
\[ N^{s_t} V_{t,1}^0 \subset \bigoplus_{i<t} V_{i,j_i}. \]
Thus if $v \in V_{t,1}^0$ we can write
\[ N^{s_t} v = \sum_{i<t} N^{s_t} v_i \]
for unique elements $v_i \in V_{i,j_i-s_t}$. Set $V_{t,1}$ to be the set of 
\[ v - \sum_{i<t} v_i. \]
We see that $V_{t,1} \subset V((a_t))$ is a $\barQQ_l$-sub-vector space lifting $\barV_{t,1}$, which is $I_K$-invariant and satisfies $N^{s_t}V_{t,1}=(0)$. Set $V_{t,j}=N^{j-1}V_{t,1}$. It is not hard to see that these $V_{i,j}$ have all the desired properties.
\end{proof}

We remark that if we define an increasing filtration on $V$ by 
\[ \Fil_i V = \bigoplus_{i' \leq i} \bigoplus_j V_{i,j}\]
then 
\[ V^{F-\semis} \cong \bigoplus_{i=1}^{u} \gr_i V. \]

\begin{lem}\label{gen} \begin{enumerate}
\item Suppose $\rho:G_K \ra \GL_n(\barQQ_l)$ is a continuous representation; that $\imath:\barQQ_l \iso \C$ and that $\pi$ is an irreducible smooth representation of $\GL_n(K)$ over $\C$ with $\imath\WD(\rho)^{F-\semis}\cong \rec_K(\pi)$. If $\pi$ is generic then $\rho$ is smooth.
\item The closed points $\wp$ in $\Spec R_{\CO,\barrho}^\Box[1/l]$  for which $\rho_\wp$ is robustly smooth are Zariski dense.
\end{enumerate} \end{lem}

\begin{proof}
For the first part write $\pi = \Sp_{s_1}(\pi_1)\boxplus \dots \boxplus
\Sp_{s_t}(\pi_t)$ for some supercuspidal representations $\pi_i$ of
$\GL_{n_i}(K)$ and positive integers $s_i,n_i$ with $\sum s_i n_i
=n$. (We are using the notation of \cite{ht}.)  Then $\rho$ has a filtration with graded pieces $\rho_{i}$ satisfying $\imath \WD(\rho_{i})^{F-\semis}=\rec_K(\Sp_{s_i}(\pi_i))$, possibly after reordering the
$i$'s. Thus $(\ad \rho)(1)$ has a filtration with graded pieces $\Hom( \rho_{i}, \rho_{j}(1))$. If this
had non-zero invariants, then $\pi_i \cong \pi_j \otimes |\det|^m$ for some
$\max\{1, 1+s_j-s_i\} \leq m \leq s_j$.
Thus $(\pi_i,s_i)$ and $(\pi_j,s_j)$ are linked, contradicting the assumption
that $\pi$ is generic (see page 36 of \cite{ht}).

For the second part suppose that $\wp$ is a closed point of $\Spec R^\Box_{\CO,\barrho}[1/l]$. Set $\CO'$ equal to the image of $R^\Box_{\CO,\barrho}$ in $L'=R^\Box_{\CO,\barrho}[1/l]/\wp$ and let $\lambda'$ denote the maximal ideal of $\CO'$. By Lemma \ref{filter} we can find a decomposition $(L')^n = \bigoplus_{i=1}^u V_i$ such that
\begin{itemize}
\item for each $i$ the sub-space $V_i$ is invariant under $I_K$;
\item for each $i$ the sub-space $\Fil_i V= \bigoplus_{i'=1}^i V_{i'}$ is invariant under $G_K$;
\item and for each $i$ we have $\WD(\gr_i  V) = \Spp_{s_i} (W_i)$, where $W_i$ has some type $a_i$.
\end{itemize}
(By $\Spp_s(W)$ we mean the Weil--Deligne representation of $W_K$ whose
underlying representation of $W_K$ is $W \oplus W(1) \oplus \dots \oplus W(s-1)$ and where $N:W(i) \iso W(i+1)$ for $i=0,\dots,s-2$.)
Choose $M\in \Z_{>0}$ so that
\[ \bigoplus_{i=1}^u ((\CO')^n \cap V_i) \supset (l^{M-1} \CO')^n. \]
Let 
\[ A \in \ker (\GL_n(\CO'[[X_1,\dots,X_u]]) \lra \GL_n(\F)) \]
be the unique element which preserves each $(V_i \cap (\CO')^n) \otimes_{\CO'} \CO'[[X_1,\dots,X_u]]$ and acts on it by multiplication by $(1+l^M X_i)$. Note that $A$ commutes with $\rho_\wp(I_K)$. Then there is a unique continuous representation
\[ \rho: G_K \lra \GL_n(\CO'[[X_1,\dots,X_u]]) \]
such that $\rho|_{I_K}=\rho_\wp|_{I_K} \otimes_{\CO'} \CO'[[X_1,\dots,X_u]]$ and such that for any lift $\phi$ of $\Frob_K$ to $W_K$ we have $\rho(\phi)=\rho_\wp(\phi) A$. Then $\rho$ is a lift of $\barrho$. 
If $x \in (\lambda')^u$ write $\rho_x$ for $\rho \bmod
(X_1-x_1,\dots,X_u-x_u)$. Note that $\rho_0=\rho_\wp$. We will show
that for (Zariski) generic $x$ that $\rho_x$ is robustly smooth and
the second part of the lemma will follow. (Note that if $0 \neq f \in
\CO'[[X_1,\dots,X_u]]$ then $f$ can not vanish on all of
$(\lambda')^u$, by, for instance, Lemma 3.1 of \cite{blght}.)

If $y \in (\CO')^\times$ let $\nu_y:G_K/I_K \ra (\CO')^\times$ be the unramified character taking $\Frob_K$ to $y$. Then $(\ad \rho_x)(1)$ has a filtration with graded pieces 
\[ \Hom(V_i, V_j(\nu_{(1+l^Mx_j)/q(1+l^Mx_i)})).\]
Note that if $i=j$ then 
\[ \Hom_{G_{K'}}(V_i,V_i(\nu_{1/q}))=(0) \]
for any finite $K'/K$, because 
\[ \Hom_{W_{K'}}(W_i, W_i(\nu_{1/q^j})) =(0) \]
for $j=1,\dots,s_j+1$ (because, in turn, $W_i$ and $W_i(\nu_{1/q^j})$ will have different types). So it remains to show that for general $x$ we will also have  
\[ \Hom_{G_{K'}}(V_i, V_j(\nu_{(1+l^Mx_j)/q(1+l^Mx_i)}))=(0) \]
for all $i \neq j$ and all finite $K'/K$. Let $\phi \in W_K$ denote a Frobenius lift and let $L''$ denote the compositum of all extensions of $L'$ of degree less than or equal to $n$. Then $L''/L'$ is finite. It will do to choose $(x_i) \in (\lambda')^u$ so that if $i \neq j$ and if $\alpha$ (resp.\ $\beta$) is an eigenvalue of $\phi$ on $V_i$ (resp.\ $V_j$) and if $\zeta$ is a root of unity then $q \alpha \zeta (1+l^Mx_i) \neq \beta (1+l^Mx_j)$. However if such an equality were to hold then $\zeta \in L''$. As $L''$ contains only finitely many roots of unity, the $x_i$'s need only satisfy finitely many inequalities, as desired.
\end{proof}

Suppose that $\CC$ is a set of irreducible components of $\Spec R^\Box_{\CO,\barrho}[1/l]$ and let $R^\Box_{\CO,\barrho,\CC}$ denote the maximal quotient of $R^\Box_{\CO,\barrho}$, which is reduced, $l$-torsion free and has
$\Spec R^\Box_{\CO,\barrho,\CC}[1/l]$ supported on the components in $\CC$. Also
let $\calD_\CC$ denote the set of liftings of $\barrho$ to complete local noetherian $\CO$-algebras $R$ with residue field $\F$ such that the induced map $R^\Box_{\CO,\barrho}
\ra R$ factors through $R^\Box_{\CO,\barrho, \CC}$. By Lemma \ref{conj} above and Lemma 3.2 of \cite{blght} we see that $\calD_\CC$ is a deformation problem in the sense of Definition 2.2.2 of \cite{cht}. 

If $K'/K$ is finite and Galois we will let
$R^\Box_{\CO,\barrho,K'-\nr}$ denote the maximal quotient of
$R^\Box_{\CO,\barrho}$ over which $\rho^\Box(I_{K'}) = \{ 1_n
\}$.
\begin{lem}The ring  $R^\Box_{\CO,\barrho,K'-\nr}[1/l]$ is either the zero ring
or is formally smooth of dimension $n^2$.
  
\end{lem}
\begin{proof}
 It suffices to show that 
  for any
  maximal ideal $\wp$ of  $R^\Box_{\CO,\barrho,K'-\nr}[1/l]$ we have
  \[ H^2(\Gal((K')^\nr/K), \ad \rho_\wp)=(0) \]
  and
  \[ \dim_{k(\wp)} H^1(\Gal((K')^\nr/K), \ad \rho_\wp)=\dim_{k(\wp)} H^0(\Gal((K')^\nr/K), \ad \rho_\wp). \]
  However
  \[ \begin{array}{rcl} H^i(\Gal((K')^\nr/K), \ad \rho_\wp)&=&H^i(G_K/I_K, (\ad \rho_\wp)^{I_{(K')^\nr/K}})\\ &=& \left\{ \begin{array}{ll} (\ad \rho_\wp)^{G_K} &{\rm if}\,\,\, i=0 \\ ((\ad \rho_\wp)^{I_K})_{G_K/I_K}& {\rm if}\,\,\, i=1 \\ (0) & {\rm otherwise}.\end{array} \right. \end{array}\]
  The lemma follows.
\end{proof}
 Thus $R^\Box_{\CO,\barrho, K'-\nr}[1/l]=R^\Box_{\CO,\barrho,\CC_{K'-\nr}}[1/l]$ for some finite set of components $\CC_{K'-\nr}$ of $\Spec R^\Box_{\CO,\barrho}$. Let $\CC_\pnr$ denote the union of $\CC_{K'-\nr}$ over all finite Galois extensions $K'/K$ and set $R^\Box_{\CO,\barrho,\pnr}=R^\Box_{\CO,\barrho,\CC_\pnr}$. A $\barQQ_l$-point of $R^\Box_{\CO,\barrho}$ factors through $R^\Box_{\CO,\barrho,\pnr}$ if and only if it is potentially unramified. The ring
$R^\Box_{\CO,\barrho,\pnr}[1/l]$ is either the zero ring or is formally smooth of dimension $n^2$. 

For $i=1,2$, let 
\[ \rho_i: G_K \lra \GL_n(\CO_{\barQQ_l}) \]
be a continuous representation. 
We say that $\rho_1$ {\em connects to} $\rho_2$, which we denote $\rho_1
\sim \rho_2$, if and only if 
\begin{itemize}
\item the reduction $\barrho_1=\rho_1 \bmod \gm_{\barQQ_l}$ is equivalent to the reduction $\barrho_2= \rho_2 \bmod \gm_{\barQQ_l}$, and
\item $\rho_1$ and $\rho_2$ define points on a common irreducible component of $\Spec (R_{\barrho_1}^\Box \otimes \barQQ_l)$.
\end{itemize}
We say that $\rho_1$ {\em strongly connects to} $\rho_2$, which we write $\rho_1 \leadsto \rho_2$, if $\rho_1
\sim \rho_2$ and $\rho_1$ lies on a unique irreducible component of $\Spec (R_{\barrho_1}^\Box \otimes \barQQ_l)$. 

We make the following remarks.
\begin{enumerate}
\item By Lemma \ref{conj} the relations $\rho_1 \sim \rho_2$ and $\rho_1 \leadsto \rho_2$ do not depend 
on the equivalence chosen between the reductions $\barrho_1$ and $\barrho_2$, nor on the $\GL_n(\CO_{\barQQ_l})$-conjugacy class of $\rho_1$ or $\rho_2$.
\item `Connects' is a symmetric relationship, but `strongly connects' may not be. 
\item `Strongly connects' is a transitive relationship, whereas 
`connects' may not be. 
\item If $\rho_1 \sim \rho_2$ and $\rho_2 \leadsto \rho_3$ then $\rho_1 \sim \rho_3$. 
\item If $\rho_1 \sim \rho_2$ and $H^0(G_K, (\ad \rho_1)(1))=(0)$ then $\rho_1 \leadsto \rho_2$.
\item\label{locconst} Write $\WD(\rho_i)=(r_i,N_i)$. If $\rho_1 \sim \rho_2$ then
  $r_1|_{I_K} \cong r_2|_{I_K}$. If $\rho_1\leadsto \rho_2$ and
  $\rho_2\leadsto \rho_1$ then $(r_1|_{I_K},N_1) \cong
  (r_2|_{I_K},N_2)$. \item If $\rho_1$ and $\rho_2$ are unramified and have the same reduction then $\rho_1 \sim \rho_2$.
\item If $K'/K$ is a finite extension and $\rho_1 \sim \rho_2$ then $\rho_1|_{G_{K'}} \sim \rho_2|_{G_{K'}}$.
\item If $\rho_1 \sim \rho_2$ and $\rho_1' \sim \rho_2'$ then $\rho_1\oplus \rho_1'  \sim \rho_2 \oplus \rho_2'$ and $\rho_1\otimes \rho_1'  \sim \rho_2 \otimes \rho_2'$ and $\rho_1^\vee \sim \rho_2^\vee$.
\item\label{twis} If $\mu:G_K \ra \barQQ_l^\times$ is a continuous character and if $\rho_1 \leadsto \rho_2$ then $\rho_1^\vee \leadsto \rho_2^\vee$ and 
$\rho_1 \otimes \mu \leadsto \rho_2 \otimes \mu$. 
 \item\label{nrtwist} If $\mu:G_K \ra \barQQ_l^\times$ is a continuous unramified
  character with $\mubar=1$ then $\rho_1 \sim \rho_1 \otimes \mu$.
\item\label{dia} Suppose that $\barrho_1$ is semisimple and let $\Fil^i$ be an invariant decreasing filtration on $\rho_1$ by $\CO_{\barQQ_l}$-direct summands, then $\rho_1 \sim \bigoplus_i \gr^i \rho_1$. \end{enumerate}

Two of these assertions (\ref{locconst}) and (\ref{dia}) require proof, so we separate them out into lemmas. The first was proved in \cite{suhhyun}, but as this is not yet easily available we give a proof here.

\begin{lem}[Choi] Suppose that $\barrho:G_K \lra \GL_n(\F)$. 
\begin{enumerate}
\item If $\wp_1$ and $\wp_2$ are two maximal ideals on the same connected component of $\Spec R_{\barrho}^\Box \otimes \barQQ_l$ then 
\[ (\rho^\Box \bmod \wp_1)|_{I_K}^\semis \cong (\rho^\Box \bmod \wp_2)|_{I_K}^\semis. \] 
\item Suppose that $\wp_1$ and $\wp_2$ are two maximal ideals lying on the same irreducible component of $\Spec R_{\barrho}^\Box \otimes \barQQ_l$, and that neither $\wp_1$ nor $\wp_2$ lie on any other irreducible component of $\Spec R_{\barrho}^\Box \otimes \barQQ_l$. Then 
\[ (\rho^\Box \bmod \wp_1)|_{I_K} \cong (\rho^\Box \bmod \wp_2)|_{I_K}. \] 
\end{enumerate} \end{lem}

\begin{proof}
Let $I_K^{(l)}$ denote the unique Sylow-pro-$l$-complement in $I_K$, which is a closed normal subgroup of $G_K$. Let $H_1=\ker \barrho|_{I_K^{(l)}}$. Then $H_1$ is also a closed normal subgroup of $G_K$, and $\rho^\Box$ factors through $G_K/H_1$. Then $(\ker \barrho|_{I_K})/H_1 \cong \Z_l$. Let $\varphi \in G_K$ denote a lift of $\Frob_K$ and let $H_0$ denote the unique closed subgroup of $\ker \barrho|_{I_K}$ which contains $H_1$ and satisfies
\[ H_0/H_1 = (\varphi^{n!} -1) ((\ker \barrho|_{I_K})/H_1). \]
We see that $H_0$ is normal in $G_K$ and that $H_0/H_1 \cong \Z_l$ and that $I_K/H_0$ is finite. We will let $\sigma$ denote a generator of $H_0/H_1$.

We will first show that for any prime ideal $\wp$ of $R_{\barrho}^\Box \otimes \barQQ_l$ the restriction $(\rho^\Box \bmod \wp)(\sigma)$ is unipotent.
To see this first note that, because $\ker \barrho|_{I_K}$ is normal in $G_K$, the
restriction $(\rho^\Box \bmod \wp)^\semis|_{\ker \barrho|_{I_K}}$ is
semi-simple, and hence is the direct sum of $n$ characters, $\chi_1
\oplus \cdots \oplus \chi_n$. Conjugation by $\varphi$ must permute these characters and hence each $\chi_i$ is invariant by $\varphi^{n!}$. Thus $\chi_i(\sigma)=1$ for each $i$ and so $(\rho^\Box \bmod \wp)(\sigma)$ is unipotent.

Again using the fact that $H_0$ is normal in $I_K$ we deduce that for any prime ideal $\wp$ of $R_{\barrho}^\Box \otimes \barQQ_l$ the representation $(\rho^\Box \bmod \wp)|_{I_K}^\semis$ factors through the finite group $I_K/H_0$. For each of the finitely many isomorphism classes $[r]$ of $n$-dimensional representations of $I_K/H_0$ over $\barQQ_l$ consider the locus in $\Spec R_{\barrho}^\Box \otimes \barQQ_l$ where $\tr \rho^\Box|_{I_K} = \tr r$. This gives finitely many disjoint closed subsets of $\Spec R_{\barrho}^\Box \otimes \barQQ_l$, whose union contains every prime of $\Spec R_{\barrho}^\Box \otimes \barQQ_l$. Thus our closed sets are also all open, and hence a union of connected components of $\Spec R_{\barrho}^\Box \otimes \barQQ_l$. The first part of the lemma follows.

Now suppose that $\wp$ is a maximal prime ideal of $R_{\barrho}^\Box \otimes \barQQ_l$ and let $S$ denote the formal power series ring in $n^2$ variables over $\barQQ_l$. We will first exhibit a surjection
\[ \phi:(R_{\barrho}^\Box \otimes \barQQ_l)^\wedge_\wp \onto S \]
such that $(\rho^\Box)|_{I_K}$ pushes forward to a conjugate of $(\rho^\Box \bmod \wp)|_{I_K}$. Let $e_1,\dots,e_d$ denote a basis of the centralizer of $(\rho^\Box \bmod \wp)(I_K)$ in $M_{n \times n}(\barQQ_l)$ and extend it to a basis $e_1,\dots,e_d,f_1,\dots,f_{n^2-d}$ of $M_{n \times n}(\barQQ_l)$. Because 
\[ \bigcap_{\tau \in I_K} \ker(\ad (\rho^\Box \bmod \wp)(\tau)^{-1} -1) = \bigoplus_j e_j \]
we see that we can find $\tau_{ik} \in I_K$ and $\mu_{ik} \in \Hom(M_{n \times n}(\barQQ_l),\barQQ_l)$ such that for all $a \in M_{n \times n}(\barQQ_l)$ and all $i=1,\dots,n^2-d$,
\[ \sum_k \mu_{ik} (\ad  (\rho^\Box \bmod \wp)(\tau_{ik})^{-1} -1) a \]
is the coefficient of $f_i$ when $a$ is written in terms of the basis $e_1,\dots,e_d,f_1,\dots,f_{n^2-d}$. 

Consider the continuous representation 
\[ \begin{array}{rcl} \rho: G_K &\lra& \GL_n(\barQQ_l[[X_1,\dots,X_d,Y_1,\dots,Y_{n^2-d}]])=\GL_n(S) \\
\sigma & \longmapsto & (1+\sum Y_i f_i) (\rho^\Box\bmod \wp)(\sigma) (1+ \sum X_je_j)^{v(\sigma)} (1+\sum Y_i f_i)^{-1},
\end{array} \]
where $v:G_K/I_K \iso \widehat{\Z}$ sends $\Frob_K$ to $1$. It is a lifting of $\rho^\Box\bmod \wp$, and so gives rise to a map 
\[ \phi:(R_{\barrho}^\Box \otimes \barQQ_l)^\wedge_\wp \lra S \]
under which $\rho^\Box|_{I_K}$ pushes forward to a conjugate of $(\rho^\Box \bmod \wp)|_{I_K}$. It just remains to show that $\phi$ is surjective. Modulo $\gm_S^2$ we have 
\[ \rho(\sigma)=(\rho^\Box \bmod \wp)(\sigma)\left(1+(\ad (\rho^\Box \bmod \wp)(\sigma)^{-1} -1)\sum_i Y_if_i + v(\sigma) \sum_jX_je_j \right). \]
Looking at
\[ \sum_k \mu_{ik} ((\rho^\Box \bmod \wp)(\tau_{ik})^{-1}\rho(\tau_{ik})-1_n) \]
we see that, for each $i$, the element $Y_i$ is in the image of $\wp \ra \gm_S/\gm_S^2$. Next looking at 
\[ ((\rho^\Box \bmod \wp)(\varphi)^{-1}\rho(\varphi)-1_n) \]
we see that for each $j$ the element $X_j$ is in the image of
\[ \wp \lra \gm_S/(\gm_S^2,Y_1,\dots,Y_{n^2-d}). \]
We conclude that 
\[ \phi: \wp \onto \gm_S/\gm_S^2 \]
and hence that
\[ \phi: (R_{\barrho}^\Box \otimes \barQQ_l)^\wedge_\wp \onto S, \]
as desired.

Write $R$ for $R_{\barrho}^\Box \otimes \barQQ_l$ and let $I$ denote the kernel of $\phi|_{R}$. It is a prime ideal contained in $\wp$.  Let $T$ denote the integral closure of $R/I$ in its field of fractions, which is finite over $R/I$ and hence has the same Krull dimension as $R/I$. Then $\phi|_{R/I}$ extends to a map $\phi_T:T \into S$. Let $\gn$ denote the contraction to $T$ of the maximal ideal of $S$. We have $\gn \cap R/I=\wp$. Moreover $T_\gn^\wedge$ is a domain. The map $\phi_T$ gives a map
\[ T_\gn^\wedge \onto S \]
extending $\phi$. We have
\[ n^2 = \dim R \geq \dim R/I=\dim T \geq \dim T_\gn^\wedge \geq \dim S = n^2, \]
where $\dim$ denotes the Krull dimension. Thus the inequalities are equalities and we deduce that $I$ is a minimal prime ideal of $R$ and that $T_\gn^\wedge \iso S$. Thus 
$R/I \into S$ and so the residue field $k(I)$ of $I$ injects into the field of fractions of $S$. As 
$(\rho^\Box)|_{I_K}$ and $(\rho^\Box \bmod \wp)|_{I_K}$ are equivalent over $S$ they are also equivalent over $k(I)$.  

We deduce that if $\wp_1$ and $\wp_2$ are two maximal primes of  $R_{\barrho_1}^\Box \otimes \barQQ_l$ each containing a minimal prime $I$ and each containing no other minimal prime, then 
\[ (\barrho_1^\Box \bmod \wp_1)|_{I_K} \cong (\barrho_1^\Box \bmod \wp_2)|_{I_K} \]
over $k(I)$ and hence over $\barQQ_l$.
\end{proof}

\begin{lem}\label{l12} Suppose that $\barrho_1$ is semisimple and let $\Fil^i$ be an invariant decreasing filtration on $\rho_1$ by $\CO_{\barQQ_l}$-direct summands, then $\rho_1 \sim \bigoplus_i \gr^i \rho_1$. 
\end{lem}

\begin{proof}
We may suppose that $L$ is chosen large enough that $\rho_1:G_K \ra \GL_n(\CO)$ and that $\Fil^i$ is defined over $L$. Then we may choose a basis
$\{ e_{i,j} \}$ of $\CO^n$ such that 
\begin{itemize}
\item for each $i$ the set $\{ e_{i',j}: \,\,i \leq i'\}$ is a basis of $\Fil^i \CO^n$;
\item and for each $i$ the set $\{ \overline{e}_{i,j}\}$ (with only $j$ varying) spans a $G_K$-submodule of $\F^n$.
\end{itemize}
(Use reverse induction on $i$.) Let $\CO\langle t \rangle$ denote the algebra of power series over $\CO$ with coefficients tending to $0$, so that $\CO\langle t \rangle$ is complete in the $l$-adic topology. Let $h$ denote the element of $M_{n \times n}(\CO\langle t \rangle)$ such that $h e_{i,j}=t^i e_{i,j}$ for all $i$ and $j$. Consider the continuous representation 
\[ \rho=h\rho_1h^{-1}: G_K \lra \GL_n(\CO\langle t \rangle). \]
Let $A^0$ denote the closed subalgebra of $\CO\langle t \rangle$ generated by the matrix entries of the image of $\rho$. As in the proof of Lemma \ref{conj}, we see that $A^0$ is a complete, noetherian local $\CO$-algebra with residue field $\F$ and that there is a continuous homomorphism $R^\Box_{\CO,\barrho_1} \ra A^0$ under which the universal lifting of $\barrho_1$ pushes forward to $\rho$. Under the map $A^0 \ra \CO$ which sends $t$ to $1$, we see that $\rho$ pushes forward to $\rho_1$. Under the map $A^0 \ra \CO$ which sends $t$ to $0$, we see that $\rho$ pushes forward to $ \bigoplus_i \gr^i \rho_1$. As $A^0$ is a domain, the claim follows.
\end{proof}

{\em Important convention:} Suppose that $F$ is a global field and
that $r:G_F \ra \GL_n(\barQQ_l)$ is a continuous representation with
irreducible reduction $\barr$. In this case there is a model $r^\circ:G_F \ra \GL_n(\CO_{\barQQ_l})$
of $r$, which is unique up to $\GL_n(\CO_{\barQQ_l})$-conjugation. If $v|p$ is a place of $F$ we write
$r|_{G_{F_v}} \sim \rho_2$ (resp.\ $r|_{G_{F_v}} \leadsto \rho_2$, resp.\ $\rho_1 \leadsto r|_{G_{F_v}}$)
to mean $r^\circ|_{G_{F_v}} \sim \rho_2$ (resp.\ $r^\circ|_{G_{F_v}} \leadsto \rho_2$, resp.\ $\rho_1 \leadsto r^\circ|_{G_{F_v}}$).\vspace{3mm}

We end this section by recalling some facts about Weil-Deligne representations.
Recall (from the paragraph before Lemma 1.4 of \cite{ty}) that a Weil--Deligne representation $(r,N)$ of the Weil group $W_K$ of a $p$-adic field $K$ on a complex vector space $W$ is called {\em pure of weight $w$} if
there is an exhaustive and separated ascending filtration $\Fil_i$ of $W$
such that
\begin{itemize}
\item each $\Fil_iW$ is invariant under $r$;
\item if $\sigma \in W_K$ maps to $\Frob_K^{v(\sigma)}$ then all eigenvalues of $r(\sigma)$ on $\gr_iW$ are Weil $q^{iv(\sigma)}$-numbers;
\item and for all $j$ we have $N^j:\gr_{w+j}W \iso \gr_{w-j}W$. (Note that necessarily we have $N \Fil_i W \subset \Fil_{i-2}W$.)
\end{itemize}
The following Lemma is part of Lemma 1.4 of \cite{ty}.
\begin{lem} \label{tylem}\begin{enumerate}
\item A Weil--Deligne representation is pure of weight $w$ if and only if its Frobenius semi-simplification is.
\item Let $K'/K$ be a finite extension. A Weil--Deligne representation of $W_K$ is pure of weight $w$  if and only if its restriction  to $W_{K'}$ is.
\item Two pure, Frobenius semi-simple, Weil--Deligne representations $(r_1,N_1)$ and $(r_2,N_2)$ are equivalent if and only if $r_1$ and $r_2$ are equivalent.
\end{enumerate}\end{lem}

\subsection{Local theory: $l=p$.}\label{l=p}  {$\mbox{}$} \newline

Continue to fix a rational prime $l$ and let $\CO$ denote the ring of integers 
of a finite extension $L$ of $\Q_l$ in $\barQQ_l$. Let $\lambda$ denote the maximal ideal of $\CO$ and
let $\F=\CO/\lambda$. However in this section we specialize our discussion to the case $\Gamma=G_K$, where $K/\Q_l$ is a finite extension. Thus $\barrho:G_K \ra \GL_n(\F)$ is continuous.  We will assume that the image of each continuous embedding $K \into \barL$ is contained in $L$.
Let  $\{ H_\tau \}$ be a collection of $n$ element multisets of integers parametrized by $\tau \in \Hom_{\Q_l}(K,\barQQ_l)$.

We call a continuous representation $\rho:G_K \ra \GL_n(\barQQ_l)$
{\em ordinary} if the following conditions are satisfied:
\begin{itemize}
\item  there is a $G_K$-invariant decreasing 
filtration $\Fil^i$ on $\overline{\Q}_l^n$ such that for $i=1,\dots,n$ the graded piece $\gr^i \overline{\Q}_l^n$ is one dimensional and $G_K$ acts on it by a character $\chi_{i}$; 
\item and there are integers $b_{\tau,i} \in \Z$ for $\tau\in \Hom_{\Q_l}(K,\barQQ_l)$ and $i=1,\dots,n$ and  an open subgroup $U \subset K^\times$ such that
\begin{itemize}
\item $(\chi_{i}\circ \Art_{K})|_{U}( \alpha) = \prod_{\tau:K \into \barQQ_l} (\tau \alpha)^{b_{\tau,i}}$
\item and $b_{\tau,1}<b_{\tau,2}<\dots<b_{\tau,n}$ for all $\tau$.
\end{itemize}\end{itemize}
We will call $\rho$ {\em ss-ordinary} if we may take $U=\CO_K^\times$ in the above definition. We will
call $\rho$ {\em cr-ordinary} if it is ordinary and crystalline (in which case it is also ss-ordinary). 
If $\rho$ is ordinary (resp.\ ss-ordinary) then it is de Rham (resp.\ semi-stable) and
\[ \HT_\tau(\rho) =\{ -b_{\tau,1},\dots,-b_{\tau,n} \}. \]
If $\rho$ is ss-ordinary and if for each $i$ there exists a $\tau$ such that
$b_{\tau,i}+1<b_{\tau,i+1}$ then $\rho$ is cr-ordinary. (See Propositions 1.24, 1.26 and 1.28 of 
\cite{nekovar} or Lemma 3.1.4 of \cite{gg}.) 

Let $K'/K$ denote a finite extension.
The universal lifting ring $R_{\CO,\barrho}^\Box$ has various important quotients $R_{\CO,\barrho,\{ H_\tau\}, *}^\Box$ which are uniquely characterized by requiring that they are reduced without $l$-torsion and that a $\barQQ_l$-point of $R_{\CO,\barrho}^\Box$ factors through $R_{\CO,\barrho,\{ H_\tau\}, *}^\Box$ if and only if it corresponds to a representation $\rho:G_K \ra \GL_n(\barQQ_l)$ which is de Rham with Hodge--Tate numbers $\HT_\tau(\rho)=H_\tau$ for all $\tau:K \into \barQQ_l$ and which has a further specified property $\CP_*$. We will consider the following instances of this construction:
\begin{itemize}
\item $*=\cris$ and $\CP_*$ is `crystalline';
\item $*=\semis$ and $\CP_*$ is `semi-stable';
\item $*=K'-\cris$ and $\CP_*$ is `crystalline after restriction to $G_{K'}$';
\item $*=K'-\semis$ and $\CP_*$ is `semi-stable after restriction to $G_{K'}$';
\item $*=\crord$ and $\CP_*$ is `cr-ordinary';
\item $*=\ssord$ and $\CP_*$ is `ss-ordinary'.
\end{itemize}
(The existence of $R_{\CO,\barrho,\{ H_\tau\}, *}^\Box$ in the case $*=\semis$ follows from Corollary 2.6.2 of \cite{kisindefrings}. In the case $*=\cris$ it follows from this using Theorem 2.5.5(2) of \cite{kisindefrings}, because $R_{\CO,\barrho,\{ H_\tau\}, \cris}^\Box$ is the maximal reduced, $l$-torsion free quotient of $R_{\CO,\barrho,\{ H_\tau\}, \semis}^\Box$ over which the $N$ of Theorem 2.5.5(2) of \cite{kisindefrings} becomes $0$. Existence in the cases $*=K'-\cris$ and $*=K'-\semis$ also follow as in the first paragraph of the proof of Theorem 2.7.6 of \cite{kisindefrings}. In the cases $*=\crord, \ssord$ existence follows from Lemma 3.3.3(1) of \cite{ger}.)

We will write  $R_{\barrho,\{ H_\tau\}, *}^\Box \otimes \barQQ_l$ 
for  $R_{\CO,\barrho,\{ H_\tau\}, *}^\Box \otimes_{\CO} \barQQ_l$. 
This definition is independent of the choice of $\CO$ (using the same argument as in the proof of Lemma \ref{coeffchange}). 

If $H_\tau$ has $n$ distinct elements for each $\tau$ then each ring
$R^\Box_{\CO, \barrho, \{ H_\tau\}, *}$ is either zero or  equidimensional of dimension
\[ 1+n^2+[K:\Q_l]n(n-1)/2. \] 
(In the cases $*=\cris$, $\semis$ this is a special case of Theorem 3.3.4 of \cite{kisindefrings}, and the cases $K'-\cris$ and $K'-\semis$ are treated in the same way. The cases $*=\crord, \ssord$ follow from this and from Lemma 3.3.3(2) of \cite{ger}.)  
We deduce that if $K'' \supset K'$ then $\Spec
R_{\CO,\barrho,\{ H_\tau\}, K'-\cris}^\Box$ (resp.\ $\Spec
R_{\CO,\barrho,\{ H_\tau\}, K'-\semis}^\Box$) is a union of
irreducible components of $\Spec R_{\CO,\barrho,\{ H_\tau\},
  K''-\cris}^\Box$ (resp.\ $\Spec R_{\CO,\barrho,\{ H_\tau\},
  K''-\semis}^\Box$).  Each of the schemes $\Spec R_{\CO,\barrho,\{
  H_\tau\}, \cris}^\Box[1/l]$ and $\Spec R_{\CO,\barrho,\{ H_\tau\},
  K'-\cris}^\Box[1/l]$ and $\Spec R_{\CO,\barrho,\{ H_\tau\},
  \crord}^\Box[1/l]$ are formally smooth.  (The case $*=\cris$ is a special case of Theorem  3.3.8 of \cite{kisindefrings}, the case $*=K'-\cris$ is treated in the same way and the case $*=\crord$ follows from this and Lemma 3.3.3(2) of \cite{ger}.) Finally if $\barrho$ is
trivial then the scheme $\Spec R_{\CO,\barrho,\{ H_\tau\},
  \crord}^\Box[1/l]$ is geometrically irreducible. (See Lemma 3.4.3 of \cite{ger}.)

Choose a finite set $\CC$ of irreducible components of $\lim_{\ra K'} \Spec R^\Box_{\CO,\barrho, \{H_\tau\}, K'-\semis}$.  Let $R^\Box_{\CO,\barrho,\CC}$ denote the maximal quotient 
of $R^\Box_{\CO,\barrho, \{ H_\tau\}, K'-\semis}$ which is reduced, $l$-torsion free and has
$\Spec R^\Box_{\CO,\barrho,\CC}$ supported on the components in $\CC$, for $K'$ chosen sufficiently large. This is independent of the choice of $K'$ (as long as $K'$ is sufficiently large). Also
let $\calD_\CC$ denote the set of liftings of $\barrho$ to complete local noetherian $\CO$-algebras $R$ with residue field $\F$ such that the induced map $R^\Box_{\CO,\barrho}
\ra R$ factors through $R^\Box_{\CO,\barrho, \CC}$. Again we see that $\calD_\CC$ is a deformation problem in the sense of Definition 2.2.2 of \cite{cht}. 

If $\rho_1$ and $\rho_2$ are continuous representations $G_K \ra \GL_n(\CO_{\barQQ_l})$, we say that $\rho_1$ {\em connects to} $\rho_2$, which we denote $\rho_1
\sim \rho_2$, if and only if 
\begin{itemize}
\item the reduction $\barrho_1=\rho_1 \bmod \gm_{\barQQ_l}$ is equivalent to the reduction $\barrho_2= \rho_2 \bmod \gm_{\barQQ_l}$;
\item $\rho_1$ and $\rho_2$ are both potentially crystalline;
\item for each continuous $\tau:K \into \barQQ_l$ we have $\HT_\tau(\rho_1)=\HT_\tau(\rho_2)$;
\item and $\rho_1$ and $\rho_2$ define points on the same irreducible component of the scheme $\Spec (R_{\barrho_1,\{ \HT_\tau(\rho_1)\}, K'-\cris}^\Box \otimes \barQQ_l)$ for some (and hence all) sufficiently large $K'$.
\end{itemize}

We make the following remarks.
\begin{enumerate}
\item By the proof of Lemma \ref{conj} we see that the relation $\rho_1 \sim \rho_2$ does not depend 
on the equivalence chosen between the reductions $\barrho_1$ and $\barrho_2$, nor on the $\GL_n(\CO_{\barQQ_l})$-conjugacy class of $\rho_1$ or $\rho_2$.
\item `Connects' is an equivalence relation. (Because each $R_{\CO,\barrho_i,\{ H_\tau\}, K'-\cris}^\Box[1/l]$ is formally smooth.)
\item If $\rho_1 \sim \rho_2$ then
  $\WD(\rho_1)|_{I_K} \cong \WD(\rho_2)|_{I_K}$. (See the proof of
  Theorem 2.7.6 of \cite{kisindefrings}.)
\item If $K'/K$ is a finite extension and $\rho_1 \sim \rho_2$ then $\rho_1|_{G_{K'}} \sim \rho_2|_{G_{K'}}$.
\item If $\rho_1 \sim \rho_2$ and $\rho_1' \sim \rho_2'$ then $\rho_1\oplus \rho_1'  \sim \rho_2 \oplus \rho_2'$ and $\rho_1\otimes \rho_1'  \sim \rho_2 \otimes \rho_2'$ and $\rho_1^\vee \sim \rho_2^\vee$.
 \item If $\mu:G_K \ra \barQQ_l^\times$ is a continuous unramified
  character with $\mubar=1$ and $\rho_1$ is potentially crystalline
  then $\rho_1 \sim \rho_1 \otimes \mu$.
\item \label{semisimp} Suppose that $\rho_1$ is potentially crystalline and that $\barrho_1$ is semisimple. Let $\Fil^i$ be an invariant filtration on $\rho_1$ by $\CO_{\barQQ_l}$ direct summands, then $\rho_1 \sim \bigoplus_i \gr^i \rho_1$. (This is proved in the same way as Lemma \ref{l12} of the previous section.)  
   \end{enumerate}

We will call a representation $\rho:G_K \ra \GL_n(\CO_{\barQQ_l})$ {\em diagonalizable} if 
it is crystalline and connects to some representation $\chi_1 \oplus
\dots \oplus \chi_n$ with $\chi_i:G_K \ra
\CO_{\barQQ_l}^\times$ crystalline characters. We will call a representation $\rho:G_K \ra \GL_n(\CO_{\barQQ_l})$ {\em potentially diagonalizable} if there is a finite extension $K'/K$ such that $\rho|_{G_{K'}}$ is diagonalizable.
Note that if $K''/K$ is a finite extension and $\rho$ is diagonalizable (resp.\ potentially diagonalizable)
then $\rho|_{G_{K''}}$ is diagonalizable (resp.\ potentially diagonalizable). It seems to us an interesting, and important, question to determine which potentially crystalline representations are potentially diagonalizable. We know no example of a potentially crystalline representation, which is not potentially diagonalizable.

\begin{lem} If $\rho_1$ and $\rho_2$ are conjugate in $\GL_n(\barQQ_l)$ then $\rho_1$ is potentially diagonalizable if and only if $\rho_2$ is potentially diagonalizable. \end{lem}
Thus we can speak of a representation $\rho:G_K \ra \GL_n(\barQQ_l)$ being potentially diagonalizable
without needing to specify an invariant lattice.

\begin{proof}
If $\rho_1$ and $\rho_2$ are conjugate by an element of $\GL_n(\CO_{\barQQ_l})$ then after passing to a finite extension over which $\barrho_1=\barrho_2=1$ we see that $\rho_1 \sim \rho_2$. Thus we may
suppose that $\rho_1=g\rho_2g^{-1}$ where $g=\diag(d_1,\dots,d_n)$ with $d_i \in \barQQ_l^\times$
satisfying $d_n|d_{n-1}|\dots|d_1$. Choosing $L \subset \barQQ_l$ large enough we may assume that
$\rho_1$ and $\rho_2$ are defined over $\CO$ and that $d_1,\dots,d_n \in \CO$. (If $d_n$ is not integral multiply each $d_i$ by a suitable element of $L^\times$.)  Replacing $K$ by a finite extension we may also assume that $\rho_2 \equiv 1 \bmod ld_1/d_n$, in which case we also have $\barrho_1=1$.

Consider the complete topological domain
\[ A=\CO\langle t_1,s_1,t_2,s_2,\dots,t_{n-1},s_{n-1}\rangle / (s_1t_1-(d_1/d_2),\dots,s_{n-1}t_{n-1}-(d_{n-1}/d_n)). \]
(See Lemma \ref{ID}.)
Let $\tg=\diag(t_1\cdots t_{n-1},t_2\cdots t_{n-1},\dots,t_{n-1},1)$ and let
\[ \trho = \tg \rho_2 \tg^{-1}. \] If $j \geq i$ then the $(i,j)$
entry of $\trho(\sigma)$ is $t_i\dots t_{j-1}$ times the $(i,j)$ entry
of $\rho_2(\sigma)$.  If $i > j$ then the $(i,j)$ entry of
$\trho(\sigma)$ is $s_j\dots s_{i-1} d_i/d_j$ times the $(i,j)$ entry
of $\rho_2(\sigma)$, which is in $A$ by our assumption that $\rho_2
\equiv 1 \bmod ld_1/d_n$. Thus we see that $\trho:G_K \ra \GL_n(A)$ is
a continuous homomorphism. The specialization under $t_i \mapsto 1$
for all $i$ is $\rho_2$. The specialization under $s_i \mapsto 1$ for
all $i$ is $\rho_1$. As in the proof of Lemma \ref{conj} we conclude
that $\rho_1 \sim \rho_2$, and we are done.
\end{proof}

We will establish some cases of (potential) diagonalizability below,
but first we must recall some results from the theory of Fontaine and Laffaille
\cite{fl}, normalized as in Section 2.4.1 of \cite{cht}. Assume that $K/\Q_l$ is unramified and
denote its ring of integers by $\CO_K$.
Let $\mc{MF}_{\CO}$
denote the category of finite $\mc{O}_K \otimes_{\bb{Z}_l}
\CO$-modules $M$ together with
\begin{itemize}
\item a decreasing filtration $\Fil^i M$ by $\mc{O}_K
  \otimes_{\bb{Z}_l} \CO$-submodules which are $\mc{O}_K$-direct
  summands with $\Fil^0 M = M$ and $\Fil^{l-1}M=\{0\}$;
\item and $\Frob_p^{-1} \otimes 1$-linear maps $\Phi^i : \Fil^i M \rightarrow
  M$ with $\Phi^i|_{\Fil^{i+1}M} = l \Phi^{i+1}$ and $\sum_i \Phi^i
  \Fil^i M = M$.
\end{itemize}
Let $\Rep_{\CO}(G_K)$ denote the category of finite
$\CO$-modules with a continuous $G_K$-action.  There is an exact,
fully faithful, covariant functor of $\CO$-linear categories
$\mathbf{G}_K : \mc{MF}_{\CO} \rightarrow \Rep_{\CO}(G_K)$.
The essential image of $\mathbf{G}_K$ is closed under taking
sub-objects and quotients. If $M$ is an object of $\mc{MF}_{\CO}$,
then the length of $M$ as an $\CO$-module is $[K : \bb{Q}_l]$
times the length of $\mathbf{G}_K(M)$ as an $\CO$-module.  

Let $\F$ denote the residue field of $\CO$ and let
$\mc{MF}_{\F}$ denote the full subcategory of $\mc{MF}_\CO$ consisting of objects killed by
the maximal ideal $\lambda$ of $\CO$ and let $\Rep_{\F}(G_K)$ denote the category of
finite $\F$-modules with a continuous $G_K$-action. Then $\mathbf{G}_K$
restricts to a functor $\mc{MF}_{\F} \rightarrow \Rep_{\F}(G_K)$.
If $M$ is an object of $\mc{MF}_{\F}$ and $\tau$ is a continuous
embedding $K \into \barQQ_l$, we let $\FL_\tau(M)$ denote the multiset of
integers $i$ such that $ \gr^i
M\otimes_{\CO_K\otimes_{\bb{Z}_l}\CO, \tau\otimes 1}\CO \neq
\{0\}$ and $i$ is counted with multiplicity equal to the $\F$-dimension of this space. If $M$ is an $l$-torsion free object of 
$\mc{MF}_{\CO}$ then $\mathbf{G}_K(M) \otimes_{\Z_l}\Q_l$ is crystalline  
and for every continuous embedding $\tau:K \into \barQQ_l$ we have
\[ \HT_\tau(\mathbf{G}_K(M) \otimes_{\Z_l}\Q_l) = \FL_\tau(M \otimes_{\CO} \F). \]
Moreover, if $\Lambda$ is a $G_K$-invariant lattice in a crystalline representation $V$ of $G_K$
with all its Hodge--Tate numbers in the range $[0,l-2]$ then $\Lambda$
is in the image of $\mathbf{G}_K$. (See \cite{fl}.)

\begin{lem}
  \label{fllift} Let $K/\Ql$ be
  unramified. Let $\barM$ denote an object of $\mc{MF}_{\F}$ together with a filtration
  \[ \barM=\barM_0 \supset \barM_1 \supset \dots \supset \barM_{n-1} \supset \barM_n=(0) \]
  by $\mc{MF}_{\F}$-subobjects such that $\barM_{i}/\barM_{i+1}$ has $\F$-rank
  $[K:\Q_l]$ for $i=0,\dots,n-1$. Then we can find an object $M$ of $\mc{MF}_{\CO}$ which
  is $l$-torsion free together with a filtration by $\mc{MF}_{\CO}$-subobjects
  \[ M=M_0 \supset M_1 \supset \dots \supset M_{n-1} \supset M_n=(0) \]
  and an isomorphism
  \[ M \otimes_{\CO} \F \cong \barM \]
  under which $M_i \otimes_{\CO} \F$ maps isomorphically to $\barM_i$ for all $i$.
  \end{lem}
  
  \begin{proof} $\barM$ has an $\F$ basis $\bare_{i,\tau}$ for $i=1,\dots,n$ and $\tau \in \Hom(K, \barQQ_l)$ such that 
  \begin{itemize}
  \item the residue field $k_K$ of $K$ acts on $\bare_{i,\tau}$ via $\tau$;
  \item $\barM_j$ is spanned over $\F$ by the $\bare_{i,\tau}$ for $i>j$;
  \item  and for each $j$ there is a subset $\Omega_j \subset \{1,\dots,n\} \times \Hom(K,\barQQ_l)$ such that
  $\Fil^j \barM$ is spanned over $\F$ by the $\bare_{i,\tau}$ for $(i,\tau) \in \Omega_j$. 
  \end{itemize}
  Then we define $M$ to be the free $\CO$-module with basis $e_{i,\tau}$ for $i=1,\dots,n$ and $\tau \in \Hom(K, \barQQ_l)$. 
  \begin{itemize}
  \item We let $\CO_K$ act on $e_{i,\tau}$ via $\tau$; 
  \item we define $M_j$ to 
  the sub $\CO$-module generated by the  $e_{i,\tau}$ with $i>j$; 
  \item and
  we define $\Fil^j M$ to be the $\CO$-submodule spanned by the $e_{i,\tau}$ for $(i,\tau) \in \Omega_j$. 
  \end{itemize}
  We define $\Phi^j:\Fil^j M \ra M$ by reverse induction on $j$. If we have defined $\Phi^{j+1}$
  we define $\Phi^j$ as follows:
  \begin{itemize}
  \item If $(i,\tau) \in \Omega_{j+1}$ then $\Phi^j e_{i,\tau}=l\Phi^{j+1}e_{i,\tau}$.
  \item If $(i,\tau) \in \Omega_j-\Omega_{j+1}$ then $\Phi^j e_{i,\tau}$ is chosen to be an
  $\CO$-linear combination of the $e_{i',\tau\circ \Frob_p}$ for $i' \geq i$ which lifts
  $\barPhi^j\bare_{i,\tau}$. 
  \end{itemize}
  It follows from Nakayama's lemma that $M$ is an object of  $\mc{MF}_{\CO}$, and then it is easy to verify that it has the desired properties.
  \end{proof}
  
We can now state and prove our potential diagonalizability criteria.

\begin{lem}\label{locallift} Keep the above notation, including the assumption $l=p$. Suppose that $\rho:G_K \ra \GL_n(\barQQ_l)$ is a potentially crystalline representation.
\begin{enumerate}
\item If $\rho$ has a $G_K$-invariant filtration with one dimensional graded pieces, in particular if it is ordinary, then $\rho$ is potentially diagonalizable.
\item If $K/\Q_l$ is unramified, if $\rho$ is crystalline and if for each $\tau:K \into \barL$ the Hodge--Tate numbers $\HT_\tau(\rho) \subset [a_\tau,a_\tau+l-2]$ for some integer $a_\tau$, then $\rho$ is potentially
diagonalizable.
\end{enumerate}
\end{lem}

\begin{proof}
After passing to a finite extension so that $\barrho$ becomes trivial and each $\gr^i \rho$ becomes crystalline, the first part follows from item (\ref{semisimp}) of the first numbered list of this section.

For the second we may assume (by twisting) that $a_\tau=0$ for all $\tau$. Note that every irreducible subquotient of $\barrho|_{I_K}$ is trivial on wild inertia and
hence one dimensional. Choose a finite unramified extension $K'/K$
such that $\barrho(G_{K'})=\barrho(I_K)$.  Then
$\barrho|_{G_{K'}}$ has a $G_{K'}$ invariant filtration with
$1$-dimensional graded pieces. From Lemma \ref{fllift} (and the
discussion just proceeding it) we see that $\barrho|_{G_{K'}}$ has a
crystalline lift $\rho_2$ with the same Hodge--Tate numbers as
$\rho|_{G_{K'}}$ which also has a $G_{K'}$-invariant filtration with
one dimensional graded pieces. It follows from Lemma 2.4.1 of
\cite{cht} that $\rho|_{G_{K'}} \sim \rho_2$. (Note that in Section 2.4.1 of \cite{cht} there is a running assumption that the Hodge--Tate numbers are distinct. However this assumption is not used in the proof of Lemma 2.4.1 of \cite{cht}.) From the first part of this lemma we see that $\rho_2$ is
potentially diagonalizable. Hence $\rho$ is also potentially diagonalizable.
\end{proof}

{\em Important convention}: Suppose that $F$ is a global field and that $r:G_F \ra \GL_n(\barQQ_l)$ is a continuous representation with irreducible reduction $\barr$. In this case there is model $r^\circ:G_F \ra \GL_n(\CO_{\barQQ_l})$
of $r$, which is unique up to $\GL_n(\CO_{\barQQ_l})$-conjugation. If $v|l$ is a place of $F$ we write
$r|_{G_{F_v}} \sim \rho_2$ to mean $r^\circ|_{G_{F_v}} \sim
\rho_2$. We will also say that $r|_{G_{F_v}}$ is (potentially)
diagonalizable to mean that $r^\circ|_{G_{F_v}}$ is.

\subsection{Global theory.}\label{global}{$\mbox{}$} \newline

Fix an odd rational prime $l$ and let $\CO$ denote the ring of integers 
of a finite extension $L$ of $\Q_l$ in $\barQQ_l$. Let $\lambda$ denote the maximal ideal of $\CO$ and
let $\F=\CO/\lambda$.
Let $F$ denote an imaginary CM field with maximal totally real subfield $F^+$. We suppose that each prime of $F^+$ above $l$ splits in $F$ and that $L$ contains the image of each embedding $F \into \barL$. Let $S$ denote a finite
set of places of $F^+$ which split in $F$ and suppose that $S$ contains all places of $F^+$ above $l$.
For each $v \in S$ choose once and for all a prime $\tv$ of $F$ above $v$, and let $\tS$ denote the set of $\tv$ for $v \in S$. 

Let
\[ \barr: G_{F^+} \lra \CG_n(\F) \]
be a continuous representation such that $G_F=\barr^{-1}\CG^0_n(\F)$ and such that $\barr$ is unramified outside $S$. Let
\[ \mu:G_{F^+} \lra  \CO^\times \]
be a continuous character lifting $\nu \circ \barr$. We suppose that $\mu$ is de Rham, so that there is
an integer $w$ such that $\HT_\tau(\mu)=\{w\}$ for all $\tau:F^+ \into L$. For $\tau:F \into L$ let $H_\tau$ denote a multiset of $n$ integers such that 
\[ H_{\tau \circ c}=\{ w-h:\,\, h \in H_\tau \}. \]

For $v \in S$ with $v\ndiv l$ choose a set $\CC_v$ of irreducible
components of the scheme $\Spec R^\Box_{\CO,\breve{\barr}|_{G_{F_\tv}}}[1/l]$,
and let $\calD_{v}$ denote the corresponding local deformation
problem. For $v \in S$ with $v| l$ choose a finite set $\CC_v$ of
irreducible components of $\lim_{\ra K'} \Spec R^\Box_{\CO,\breve{\barr}|_{G_{F_\tv}},
  \{H_\tau\}, K'-\semis}$, and let $\calD_v$ denote the corresponding
local deformation problem.

Let
\[ \CS=(F/F^+,S, \tS, \CO,\barr,\mu,\{ \calD_v\}_{v \in S}). \]
The following proposition is established in \cite{cht} (see Proposition 2.2.9 and Corollary 2.3.5 of that paper).
\begin{prop}\label{drb} Keep the notation and assumptions established in this section. Suppose moreover that
$\breve{\barr}$ is absolutely irreducible.
\begin{enumerate}
\item There is a universal deformation 
\[ r_\CS^\univ:G_{F^+} \lra \CG_n(R^\univ_\CS) \]
of $\barr$ of type $\CS$ in the sense of Section 2.3 of \cite{cht}. 
\item If $\mu(c_v)=-1$ for all $v|\infty$ and if each $H_\tau$ has $n$ distinct elements then 
  $R^\univ_\CS$ has Krull dimension at least $1$.
\end{enumerate}
\end{prop}

\newpage

\section{Automorphy Lifting Theorems.}\label{RAECSDC}

\subsection{Terminology.} \label{terminology}{$\mbox{}$} \newline

Continue to fix a rational prime $l$ and an isomorphism $\imath:\barQQ_l \iso \C$.

Suppose that $F$ is a CM (or totally real) field and that $F^+$ is the maximal totally real subfield of $F$. By a {\em polarized} $l$-adic (resp.\ mod $l$) representation of $G_F$ we will mean a pair $(r,\mu)$ (resp.\ $(\barr,\barmu)$), where
\[ r: G_F \lra \GL_n(\barQQ_l) \]
(resp.\ $\barr: G_F \lra \GL_n(\barFF_l)$) and
\[ \mu: G_{F^+} \lra \barQQ_l^\times \]
(resp.\ $\barmu: G_{F^+} \lra \barFF_l^\times$) are continuous homomorphisms, such that for some infinite place $v$ of $F^+$ there is $\varepsilon_v \in \{ \pm 1\}$  and a non-degenerate pairing $\langle\,\,\, ,\,\,\,\rangle_v$ on $\barQQ_l^n$ (resp.\ $\barFF_l^n$) such that
\[ \langle x,y \rangle_v = \varepsilon_v \langle y,x \rangle_v \]
and
\[ \langle r(\sigma) x, r(c_v\sigma c_v) y \rangle_v = \mu(\sigma) \langle x,y \rangle_v \]
(resp.\ 
\[\langle \barr(\sigma) x, \barr(c_v\sigma c_v) y \rangle_v = \barmu(\sigma) \langle x,y \rangle_v) \]
for all $x,y \in \barQQ_l^n$ (resp.\ $\barFF_l^n$) and all $\sigma \in G_F$. In the case that $F$ is imaginary we further require that $\varepsilon_v = - \mu(c_v)$. (This last condition can always be achieved by replacing $\mu$ by $\mu \delta_{F/F^+}$.)

Note the following.
\begin{itemize}
\item If the condition is true for one place $v|\infty$ it will be true for
all places $v|\infty$: take $\varepsilon_{v'}=\mu(c_vc_{v'})\varepsilon_v$ and $\langle x ,y\rangle_{v'} = \langle x,
r(c_v c_{v'})y\rangle_v$ (resp.\ $\langle x ,y\rangle_{v'} = \langle x,
\barr(c_v c_{v'})y\rangle_v$).
\item If $F$ is imaginary then $(r,\mu)$
(resp.\ $(\barr,\barmu)$) is a polarized $l$-adic (resp.\ mod $l$) representation if and only if there
is a continuous homomorphism $\widetilde{r}:G_{F^+} \ra \CG_n(\barQQ_l)$
(resp.\ $\CG_n(\barFF_l)$) with $\breve{\widetilde{r}}=r$ (resp.\ $\barr$)
and with multiplier $\mu$ (resp.\ $\barmu$). 
\item If $F$ is totally real
then $(r,\mu)$ is a polarized $l$-adic representation if and only if $r$ factors
through $\GSp_n(\barQQ_l)$ (if $\mu(c_v)=-\varepsilon_v$) or $\GO_n(\barQQ_l)$ (if
$\mu(c_v)=\varepsilon_v$) with multiplier $\mu$. [Define the pairing on
$\barQQ_l^n$ by $\langle x ,y\rangle = \langle x, r(c_v)y\rangle_v$.]
A similar assertion is true in the case of $\barFF_l$ if $l>2$.
\end{itemize}

We will call $(r,\mu)$ (resp.\ $(\barr,\barmu)$) {\em totally odd} if $\varepsilon_v=1$ for all $v|\infty$. (Equivalently, if $\mu(c_v)$ (resp.\ $\barmu(c_v)$) is independent of $v|\infty$ and $\varepsilon_v=1$ for some $v|\infty$.) We call an $l$-adic (resp.\ mod $l$) representation $r$ (resp.\ $\barr$) {\em polarizable} if there exists a $\mu$ (resp.\ $\barmu$) such that $(r,\mu)$ (resp.\ $(\barr,\barmu)$) is a polarized representation. We call an $l$-adic (resp.\ mod $l$) representation $r$ (resp.\ $\barr$) {\em totally odd polarizable} if there exists a $\mu$ (resp.\ $\barmu$) such that $(r,\mu)$ (resp.\ $(\barr,\barmu)$) is a totally odd polarized representation. 

We will call $(r,\mu)$ (resp.\ $r$) {\em algebraic} if $r$ is unramified at all but finitely many primes and if it is de Rham at all primes above $l$. We will call it {\em regular algebraic} if it is algebraic and if for all $\tau:F \into \barQQ_l$ the multiset $\HT_\tau(r)$ has $n$ distinct elements.

We now recall from \cite{cht} and \cite{blght} the notions of RAESDC
and RAECSDC automorphic representations. In fact, it will be
convenient for us to work with a slight variant of these definitions,
where we keep track of the character which occurs in the essential (conjugate) self-duality. Moreover the referee suggested we use a less cumbersome name. We have followed this advice with some reservations. It is less cumbersome, but we have a slight worry it could be misleading. We hope not.

If $F$ is a number field and $\pi$ is an automorphic representation of $\GL_n(\A_F)$ we will call $\pi$ {\em regular algebraic} if $\pi_\infty$ has the same infinitesimal character as an irreducible algebraic representation of the restriction of scalars from $F$ to $\Q$ of
$\GL_n$.

Let $F$ be a CM (or totally real) field with maximal totally real subfield $F^+$. 
By a {\em polarized automorphic  representation} of
$\GL_n(\A_{F})$ we mean a pair $(\pi,\chi)$ where
\begin{itemize}
\item $\pi$ is an automorphic representation of $\GL_n(\A_F)$,
\item $\chi:\A_{F^+}^\times/(F^+)^\times \ra \C^\times$ is a
  continuous character such that $\chi_v(-1)$ is independent of $v|\infty$,
\item and $\pi^c \cong \pi^\vee \otimes (\chi \circ \norm_{F/F^+} \circ \det)$.
\end{itemize}
(Here $\pi^c$ denotes the composition of $\pi$ with complex conjugation on $\GL_n(\A_F)$. If $F$ is totally real then $\pi^c=\pi$.)
In the case that $F$ is imaginary we further suppose that $\mu_v(-1)=(-1)^n$ for all $v|\infty$. (This last condition can always be achieved by replacing $\mu$ by $\mu \delta_{F/F^+}$.)

We will call an automorphic representation $\pi$ of $\GL_n(\A_F)$ {\em polarizable} if $F$ is CM (or totally real) and there is a character $\mu$ such that $(\pi,\mu)$ is a polarized automorphic representation.

We will say that $(\pi,\chi)$ is {\em cuspidal} if $\pi$ is. We will say that $(\pi,\chi)$ is {\em regular algebraic} if $\pi$ is, in which case $\chi$ is also algebraic. 
We will say that $(\pi,\chi)$ has {\em level prime to $l$} (resp.\ {\em level potentially prime to $l$}) if
for all $v|l$ the representation $\pi_v$ is unramified (resp.\ becomes unramified after a finite base change).

If $\Omega$ is an algebraically closed field of characteristic $0$ we will
write $(\Z^n)^{\Hom(F,\Omega),+}$ for the set of $a=(a_{\tau,i}) \in (\Z^n)^{\Hom(F,\Omega)}$ satisfying
\[ a_{\tau,1} \geq \dots \geq a_{\tau,n}. \]
Let $w \in \Z$. If $F$ is totally real or CM (resp.\ if $\Omega=\C$) we will write $(\Z^n)^{\Hom(F,\Omega)}_w$
for the subset of elements $a \in (\Z^n)^{\Hom(F,\Omega)}$ with 
\[ a_{\tau,i}+a_{\tau \circ c, n+1-i}=w \]
(resp.\ 
\[ a_{\tau,i}+a_{c \circ \tau, n+1-i}=w.) \]
(These definitions are consistent when $F$ is totally real or CM and $\Omega=\C$.)
If $F'/F$ is a finite extension
we define $a_{F'} \in (\Z^n)^{\Hom(F',\Omega),+}$ by
\[ (a_{F'})_{\tau,i} = a_{\tau|_F,i}. \]
We will call $a$ {\em extremely regular} if for some $\tau$ the $a_{\tau,i}$ have the following property: for any subsets $H$ and $H'$ of $\{ a_{\tau,i}+n-i\}_{i=1}^n$ of the same cardinality, if $\sum_{h \in H} h = \sum_{h \in H'} h$ then $H=H'$.

If $a \in (\Z^n)^{\Hom(F,\C),+}$, let $\Xi_a$ denote the irreducible algebraic representation of
$\GL_n^{\Hom(F,\C)}$ which is the tensor product over $\tau$ of the
irreducible representations of $\GL_n$ with highest weights $a_\tau$. We will
say that a regular algebraic polarized (resp.\ regular algebraic) automorphic representation $(\pi,\chi)$ (resp.\ $\pi$) of $\GL_n(\A_F)$ has
{\em weight $a$} if $\pi_\infty$ has the same infinitesimal character as
$\Xi_a^\vee$. Note that in this case $a$ must lie in $(\Z^n)^{\Hom(F,\C)}_w$ for some $w \in \Z$.

Suppose that $\pi$ is a regular algebraic automorphic representation of $\GL_n(\A_F)$ of weight $a \in (\bb{Z}^n)^{\Hom(F,\bb{C}),+}$.
 Let $v$ be a place of $F$ dividing
$l$ and $\varpi_{v}$ a uniformizer in $\mc{O}_{F_v}$.  For each integer $b>0$, let
$\Iw(v^{b,b})$ denote the subgroup of $\GL_n(\CO_{F_v})$ consisting of
elements that reduce to upper triangular unipotent matrices modulo $v^b$. The space
$(\iota^{-1}\pi_v)^{\Iw(v^{b,b})}$ carries commuting actions of the
Hecke operators
\[ U_{\varpi_v}^{(j)} = \left[ \Iw(v^{b,b}) \left(\begin{matrix} \varpi 1_j & 0 \\
  0 & 1_{n-j}\end{matrix}\right) \Iw(v^{b,b}) \right] .\]
  (See for instance Lemma 2.3.3 of \cite{ger}.)
We define rescaled Hecke operators
\[ U_{\iota^* a,\varpi_v}^{(j)} := \left( \prod_{ \tau : F_v
    \hookrightarrow \overline{\bb{Q}}_l}\prod_{i=1}^{j}
  \tau(\varpi_{v})^{- a_{ \iota\circ \tau|_F, n-i+1}} \right)  U_{\varpi_v}^{(j)}
\]
for $j=1,\ldots,n$. We define the ordinary part $(\iota^{-1}
\pi_v)^{\Iw(v^{b,b}),\ord}$ of $(\iota^{-1}\pi_v)^{\Iw(v^{b,b})}$ to be
the maximal subspace which is invariant under each
$U_{\iota^* a,\varpi_v}^{(j)}$ and such that every eigenvalue of each
$U_{\iota^* a,\varpi_v}^{(j)}$ is an $l$-adic unit. This definition
does not depend on the choice of uniformizer $\varpi_v$. (If we choose
another uniformizer $\varpi'_v$, then $U_{\varpi'_v}^{(j)}=\langle u\rangle U_{\varpi_v}^{(j)}$
  where $\langle u \rangle$ is an
  operator on $\pi_v^{\Iw(v^{b,b})}$ that commutes with $U_{\varpi_v}^{(j)}$ and whose
  $b$-th power is trivial.)

We say that $\pi$ is \emph{$\iota$-ordinary} if for each $v|l$ there
is an integer $b>0$ such that the space
$(\iota^{-1}\pi_v)^{\Iw(v^{b,b}),\ord}$ is non-zero. Note that if $\psi$ is an algebraic character of $\A_F^\times/F^\times$ then $\pi$ is $\imath$-ordinary if and only if $\pi \otimes (\psi \circ \det)$ is $\imath$-ordinary. Also recall (from Lemma 5.1.5 of \cite{ger}) that if $\pi$ has weight $0$ and if $\pi_v$ is Steinberg for all $v|l$ then $\pi$ is $\imath$-ordinary.

We recall that thanks to the work of many people (we mention in
particular Bellaiche, Caraiani, Chenevier, Clozel, Harris, Kottwitz, Labesse, Shin and
R.T., and the references \cite{kott}, \cite{MR1044819}, \cite{ht}, \cite{belchen}, \cite{shin}, \cite{bookproj}, \cite{clozelpurity}, \cite{ana} and \cite{ana2}) we can associate $l$-adic representations to regular algebraic cuspidal polarized automorphic representations:
\begin{thm}\label{grfaf} Suppose that $(\pi,\chi)$ is a regular algebraic, cuspidal, polarized automorphic representation of $\GL_n(\A_F)$. Then there is a continuous 
semi-simple representation
\[ r_{l,\imath}(\pi): G_F \lra \GL_n(\barQQ_l) \]
and an integer $w$ 
with the following properties.
\begin{enumerate}
\item $(r_{l,\imath}(\pi), \epsilon_l^{1-n}r_{l,\imath}(\chi))$ is a totally odd, polarized $l$-adic representation.
\item If $v\ndiv l$ is a place of $F$ then
\[ \imath \WD(r_{l,\imath}(\pi)|_{G_{F_v}})^{F-\semis} \cong \rec(\pi_v \otimes |\det|_v^{(1-n)/2}), \]
and these Weil--Deligne representations are pure of weight $w$.
\item $r_{l,\imath}(\pi)$ is de Rham and if $\tau:F \into \barQQ_l$ then
\[ \HT_\tau(r_{l,\imath}(\pi))=\{ a_{\imath\tau,1}+n-1, a_{\imath \tau,2}+n-2,\dots,a_{\imath \tau,n} \}. \]
Moreover 
\[ \HT_{\tau \circ c}(r_{l,\imath}(\pi))= \{ w-h:\,\, h \in \HT_{\tau}(r_{l,\imath}(\pi))\}. \]
\item If $v|l$ and $\pi_v$ has an Iwahori fixed vector then
\[ \imath \WD(r_{l,\imath}(\pi)|_{G_{F_v}})^{F-\semis} \cong \rec(\pi_v \otimes |\det|_v^{(1-n)/2}). \]
In particular $r_{l,\imath}(\pi)$ is semi-stable at $v$, and if $\pi_v$ is unramified then it is crystalline.
\end{enumerate} \end{thm}
(For most of this Theorem we refer to Theorems 1.1 and 1.2 of \cite{blght}, strengthened by incorporating the main theorems of \cite{ana} and \cite{ana2}. We warn the reader that the main theorem of \cite{ana2} depends on the current paper, however there is no circularity, because it does not do so in the case that $\Pi_y$ has an Iwahori fixed vector, which is the only case we are quoting here. The first part follows from Theorem 1.2 and Corollary 1.3 of
\cite{belchen}. Note that Theorem 1.2 of \cite{belchen} can easily be extended to the case $\chi$ non-trivial by a twisting argument. Also note that irreducible factors $r$ of $r_{l,\imath}(\pi)$ which do not satisfy $r^c \cong r^\vee \otimes \epsilon_l^{1-n} r_{l,\imath}(\chi)$ occur in pairs $r$ and  $(r^c)^\vee \otimes 
\epsilon_l^{1-n} r_{l,\imath}(\chi)$, and it is straightforward to put a pairing of the desired form on $r \oplus (r^c)^\vee \otimes 
\epsilon_l^{1-n} r_{l,\imath}(\chi)$. When $F$ is CM, note that by definition $\epsilon_l^{1-n} r_{l,\imath}(\chi)$ takes every complex conjugation to $-1$.) We also have the following remarks.
\begin{enumerate}
\setcounter{enumi}{5}
\item {\em If $\pi$ is $\imath$-ordinary and $v|l$ then $r_{l,\imath}(\pi)|_{G_{F_v}}$ is ordinary. }(This follows from Lemma 5.2.1 of \cite{ger} using the same twisting argument as in section 1 of \cite{blght}.)
\item {\em If $\pi$ has level potentially prime to $l$ and if $r_{l,\imath}(\pi)$ is ordinary then $\pi$ is $\imath$-ordinary.} (See Lemmas 5.1.6 and 5.2.1 of \cite{ger}.)
\end{enumerate}
We remark that it is presumably both true and provable that $\pi$ is
$\imath$-ordinary if and only if $r_{l,\imath}(\pi)$ is ordinary, but
to work out the details here would take us too far afield.

We will let $\barr_{l,\imath}(\pi)$ denote the semisimplification of the reduction of $r_{l,\imath}(\pi)$. 
If $F$ is totally real and if $\chi_v(-1)=(-1)^{n-1}$ for all $v|\infty$ then 
$\barr_{l,\imath}(\pi)$ factors through a map
\[ \tbarr_{l,\imath}(\pi):G_F \lra \GO_n(\barFF_l)  \]
with multiplier $\barepsilon_l^{1-n} \barr_{l,\imath}(\chi)$.
If $F$ is totally real and if $\chi_v(-1)=(-1)^{n}$ for all $v|\infty$ then 
$n$ is even and $\barr_{l,\imath}(\pi)$ factors through a map
\[ \tbarr_{l,\imath}(\pi):G_F \lra \GSp_n(\barFF_l)  \]
with multiplier $\barepsilon_l^{1-n} \barr_{l,\imath}(\chi)$.
If $F$ is imaginary CM  then
it extends to a continuous homomorphism
\[ \tbarr_{l,\imath}(\pi): G_{F^+} \lra \CG_n(\barFF_l) \]
with multiplier $\barepsilon_l^{1-n} \barr_{l,\imath}(\chi)$.

We will call $(r,\mu)$ (resp.\ $(\barr,\barmu)$, resp.\ $r$, resp.\ $\barr$)  {\em automorphic} if there is a regular algebraic, polarized cuspidal automorphic representation $(\pi,\chi)$ such that
\[ (r,\mu)\cong (r_{l,\imath}(\pi),r_{l,\imath}(\chi)\epsilon_l^{1-n})\]
(resp.\ 
\[ (\barr,\barmu)\cong (\barr_{l,\imath}(\pi),\barr_{l,\imath}(\chi)\barepsilon_l^{1-n}), \]
resp.
\[ r \cong r_{l,\imath}(\pi), \]
resp.
\[ \barr \cong \barr_{l,\imath}(\pi)). \]
We will say  that $(r,\mu)$ or $(\barr,\barmu)$ or $r$ or $\barr$ is {\em automorphic of level prime to $l$} (resp.\ {\em automorphic of level potentially prime to $l$}, resp.\ {\em ordinarily automorphic}, resp.\ {\em potentially diagonalizably automorphic}) if $(\pi,\chi)$ has level prime to $l$ (resp.\ has level potentially prime to $l$, resp.\ is $\imath$-ordinary, resp.\ has level potentially prime to $l$ and $r_{l,\imath}(\pi)$ is potentially diagonalizable). By Theorem 3.13 of \cite{MR1044819} these definitions do not depend on the choice of $\imath$.

Finally recall the following definition from \cite{jack}. We will call a
subgroup $H \subset \GL_n(\barFF_l)$ {\em adequate} if the following
conditions are satisfied.
\begin{itemize}
\item $H^1(H,\barFF_l)=(0)$ and $H^1(H,\gs\gl_n(\barFF_l))=(0)$.
\item $H^0(H,\gs\gl_n(\barFF_l))=(0)$.
\item The elements of $H$ with order coprime to $l$ span $M_{n \times
    n}(\barFF_l)$ over $\barFF_l$. 
\end{itemize}
Note that this is not exactly the definition given in Definition 2.3 of \cite{jack}, however it is equivalent to it by Lemma 1 of \cite{jackapp}. Note also that if $H$ is adequate then $\barFF_l^n$ is an irreducible $H$-module (by the third condition) and $l \ndiv n$ (by the second condition, as when $l|n$ we have $\barFF_l 1_n \subset \gs\gl_n(\barFF_l)$).  The following proposition is Theorem 9 of \cite{jackapp}.

\begin{prop}\label{ghtt} Suppose that $H$ is a finite subgroup of $\GL_n(\barFF_l)$ such that the tautological representation of $H$ is irreducible. Let $H^0$ denote the subgroup of $H$ generated by all elements of $l$-power order and let $d$ denote the maximal dimension of an irreducible $H^0$-submodule of $\barFF_l^n$. If $l \geq 2(d+1)$ then $H$ is adequate. \end{prop}

(R.T.\ would like to take this opportunity to make two corrections to \cite{cht}. Robert Guralnick points out that the assumption in Corollary 2.5.4 of \cite{cht} should be $l>2n+1$ and not $l>2n-1$. The application of Lemma (2.7) c) of \cite{cps} in the penultimate sentence of the proof of Corollary 2.5.4 of \cite{cht} requires $l>2n+1$, not $l>n+1$ as was stated there. The correct form of the Corollary was used in \cite{hsbt}. 

Florian Herzig points out that in the proof of Corollary 4.4.4 of \cite{cht} we should have written $\operatorname{Sp}_n(\F_l)/\{ \pm 1\}$ and not $\operatorname{PSp}_n(\F_l)$. R.T.\ would like to thank Guralnick and Herzig for these observations.)

\subsection{Lemmas on Automorphy}{$\mbox{}$} \newline

Our first lemma is elementary.
\begin{lem}\label{untwist} Suppose that $F$ is a CM (or totally real) field and that $\psi$ is an algebraic character of $G_F$. If $F$ is imaginary let $\phi$ denote the composition of $\psi$ with the transfer map $G_{F^+}^\ab \ra G_F^\ab$. If $F$ is totally real let $\phi=\psi^2$. Then $(r,\mu)$ is automorphic if and only if $(r \otimes \psi, \mu \phi)$ is.\end{lem}

The next lemma is proven in the same way as Lemma 4.2.2 of \cite{cht}. (See also Lemmas 4.3.2 and 4.3.3 of \cite{cht}.)
\begin{lem}\label{bc} Suppose that $M/F$ is a soluble Galois extension of fields and that $M$ is CM (or totally real). Suppose that $(r,\mu)$ is a polarized $l$-adic representation of $G_F$ with $r|_{G_M}$ irreducible. Then $(r,\mu)$ is automorphic if and only if $(r|_{G_M},\mu|_{G_{M^+}})$ is automorphic. \end{lem}

The next lemma is a generalization of Lemma 7.1 of \cite{blght}. 

\begin{lem} Suppose that $\pi$ is an irreducible, unitary, admissible module for $((\Lie \GL_n(\R)) \otimes_\R \C,\operatorname{O}(n))$ or $((\Lie \GL_n(\C)) \otimes_\R \C,\operatorname{U}(n))$ with half integral Harish-Chandra parameter (i.e.\ this parameter lies in $1/2$ the co-character group of a maximal torus in the complexified group). Then
\[ \pi^c \cong \pi^\vee. \]
\end{lem}

\begin{proof} In the second case this is Lemma 7.1 of \cite{blght}. However, we will give here a uniform proof in both cases. From the classification of irreducible unitary admissible $((\Lie \GL_n(\R)) \otimes_\R \C,\operatorname{O}(n))$ and $((\Lie \GL_n(\C)) \otimes_\R \C,\operatorname{U}(n))$-modules, we see that $\pi$ is of the form
\[ \pi_1 \boxplus \cdots \boxplus \pi_r, \]
where
\[ \pi_i = (\sigma_i \otimes |\det|^{(1-m_i)/2}) \boxplus (\sigma_i
\otimes |\det|^{(3-m_i)/2}) \boxplus \cdots \boxplus (\sigma_i \otimes |\det|^{(m_i-1)/2}) \]
with $m_i \in \Z_{>0}$ and $\sigma_i$ an irreducible, unitary, discrete series, admissible module for 
$((\Lie \GL_{n_i}(\R)) \otimes_\R \C,\operatorname{O}(n_i))$ or $((\Lie \GL_{n_i}(\C)) \otimes_\R \C,\operatorname{U}(n_i))$ with half-integral Harish-Chandra parameter.  (We are using the form of the classification proposed by Tadic in 1985 and proved in \cite{tadic}. A similar classification was given earlier in \cite{vogan}. \cite{clozelb} indicates how the formulation we are using can be deduced from the one given in \cite{vogan}.)

The only possibilities for $\sigma_i$ are
\begin{itemize}
\item the $((\Lie \GL_{1}(\R)) \otimes_\R \C,\operatorname{O}(1))$-module associated to the trivial or sign representation of $\R^\times$;
\item the unique discrete series $((\Lie \GL_{2}(\R)) \otimes_\R \C,\operatorname{O}(2))$-module with the same infinitesimal character as the representation $\Sym^a \otimes |\det|^{-a/2}$ of $\GL_2(\R)$ for some $a\in \Z_{\geq 0}$;
\item the $((\Lie \GL_{1}(\C)) \otimes_\R \C,\operatorname{U}(1))$-module associated to one of the representations $z \mapsto (z/|z|)^a$ of $\C^\times$ for some $a \in \Z$.
\end{itemize}
In each case we see that $\sigma_i^\vee \cong \sigma_i^c$. Thus $\pi^\vee \cong \pi^c$, as desired. (Note that, by known properties of the Langlands local reciprocity map $\rec$, we have $(\pi_1 \boxplus \pi_2)^\vee \cong \pi_1^\vee \boxplus \pi_2^\vee$ and $(\pi_1\boxplus \pi_2)^c \cong \pi_1^c \boxplus \pi_2^c$.)
\end{proof}

The next lemma formalizes and generalizes the argument of step 3 of the proof of Theorem 7.5 of \cite{blght}. Partial generalizations of this argument have already been given in Proposition 5.2.1 of \cite{blgg} and Proposition 5.1.1 of \cite{blggord}.

\begin{lemma}\label{htcmtr}
  Let $F$ be a CM (or totally real) field and $M/F$ a
  soluble Galois, CM (or totally real) extension of degree $m$. Let $r:G_M\to
  \GL_n(\Qlbar)$ be an irreducible continuous representation and $\mu:G_{F^+} \ra \barQQ_l^\times$ a continuous character, such that $(\Ind_{G_M}^{G_F} r,\mu)$ is an automorphic polarized $l$-adic representation of $G_F$. Then $(r,\mu|_{G_{M^+}})$ is also automorphic and polarized. 
\end{lemma}

\begin{proof} 
  Inductively we may reduce to the case that $m$ is prime. 
  
  Let $\sigma$ denote a generator of $\Gal(M/F)$, and $\kappa$ a
  generator of $\Gal(M/F)^\vee$. Let $(\Pi,\chi)$ be a regular algebraic, cuspidal polarized automorphic representation of $\GL_{mn}(\A_F)$ such that $r_{l,\imath}(\Pi) \cong \Ind_{G_M}^{G_F} r$ and $\mu=\epsilon_l^{1-mn}r_{l,\imath}(\chi)$. Because 
  \[ r_{l,\imath}(\Pi) \otimes \kappa \cong r_{l,\imath}(\Pi) \]
  we deduce that
  \[ \Pi \otimes (\kappa \circ \Art_F \circ \det) \cong \Pi, \]
  and so from Theorems 3.4.2 and 3.5.1 of \cite{MR1007299} there is a cuspidal automorphic representation $\pi$ of $\GL_n(\A_M)$ such that
  \[ \pi \boxplus \pi^{\sigma} \boxplus\cdots \boxplus \pi^{\sigma^{m-1}} \]
  is a strong base change of $\Pi$ in the sense of definitions 1.6.1 and 3.1.2 and section 1.7 of \cite{MR1007299}. Lemma VII.2.6 of \cite{ht} then shows that
  \[ \BC_{M/F}(\Pi) \cong \pi \boxplus \pi^{\sigma} \boxplus\cdots \boxplus \pi^{\sigma^{m-1}}. \]
  Note that, because $\Pi$ is regular algebraic, $\pi \otimes ||\det||_M^{n(1-m)/2}$ is regular algebraic and that the representations $\pi^{\sigma^i}$ for $i=0,1,\dots,m-1$ are pairwise non-isomorphic. 
  
  Because $\Pi^c \cong \Pi^\vee \otimes (\chi \circ \norm_{F/F^+} \circ \det)$ we see that
  \[ \pi^c \cong \pi^{\sigma^i, \vee} \otimes  (\chi \circ \norm_{M/F^+} \circ \det) \]
  for some $i=0,\dots,m-1$. If $\psi$ denotes the central character of $\pi$, which must also be algebraic, then we see that
  \[ 2 \wt(\psi)=n \wt (\chi). \]
  (See the discussion at the start of Section \ref{sbuild}.) Thus the central character of $\pi \otimes ||\det ||_{F}^{-\wt(\chi)/4}$ (i.e.\ $\psi ||\,\,\,||_F^{-\wt(\psi)/2}$) is unitary, and so $\pi \otimes ||\det ||_{F}^{-\wt(\chi)/4}$ is also unitary. 
  
  Moreover $\wt(\chi)$ is even (as $F^+$ is totally real) and so for all $v|\infty$ the representation $\pi_v \otimes |\det|_v^{-\wt(\chi)/4}$ has half integral Harish-Chandra parameter. We conclude from the previous lemma that for $v|\infty$ we have
  \[ \pi_v^c  \cong  \pi_v^\vee \otimes |\det|_v^{\wt(\chi)/2} . \]
  Thus 
  \[ (\pi^{\sigma^i})_v \cong \pi_v \otimes ((\chi ||\,\,\,||_{F^+}^{-\wt(\chi)/2})\circ \norm_{M/F^+} \circ \det)_v. \]
  By the regularity of $\BC_{M/F}(\Pi)$ we deduce that we must have $i=0$, i.e.
   \[ \pi^c \cong \pi^{\vee} \otimes  (\chi \circ \norm_{M/F^+} \circ \det). \]
   In particular $(\pi \otimes ||\det||^{n(1-m)/2}, (\chi ||\,\,\,||_{F^+}^{n(1-m)}) \circ \norm_{M^+/F^+} )$ is a regular algebraic, cuspidal, polarized automorphic representation of $\GL_n(\A_M)$. Moreover
   \[\begin{array}{r}  r_{l,\imath}(\Pi)|_{G_M} \cong  r_{l,\imath}(\pi \otimes ||\det||^{n(1-m)/2}) \oplus r_{l,\imath}(\pi \otimes ||\det||^{n(1-m)/2})^\sigma \oplus \cdots \\ \oplus r_{l,\imath}(\pi \otimes ||\det||^{n(1-m)/2})^{\sigma^{m-1}}. \end{array} \]
   
   On the other hand
   \[ r_{l,\imath}(\Pi)|_{G_M} \cong (\Ind_{G_M}^{G_F} r)|_{G_M} \cong r \oplus r^\sigma \oplus \cdots \oplus r^{\sigma^{m-1}}. \]
   As $r$ is irreducible we deduce that 
   \[ r \cong r_{l,\imath}(\pi \otimes ||\det||^{n(1-m)/2})^{\sigma^j} \cong r_{l,\imath}(\pi^{\sigma^j} \otimes ||\det||^{n(1-m)/2}) \]
   for some $j$. Moreover
   \[ r_{l,\imath}((\chi ||\,\,\,||_{F^+}^{n(1-m)}) \circ \norm_{M^+/F^+} ) \epsilon_l^{1-n} = r_{l,\imath}(\chi)|_{G_{M^+}} \epsilon_l^{1-nm}=\mu|_{G_{M^+}}. \]
   The lemma follows.
 \end{proof}

\subsection{Automorphy lifting: the `minimal' case} \label{min} {$\mbox{}$} \newline

In this section we present an automorphy lifting theorem which represents the natural output of the 
Taylor--Wiles--Kisin method. We incorporate improvements due to Thorne
\cite{jack} and Caraiani \cite{ana}. This result is essentially Theorem 7.1 of
\cite{jack}.

\begin{thm}
\label{twkmlt}
  Suppose that $F$ is an imaginary CM field with maximal totally real subfield $F^+$,
   that $l$ is an odd prime and that $n\in
  \bb{Z}_{\geq 1}$. Assume also that $\zeta_l\notin F$. Let $(r,\mu)$ be an $n$-dimensional, algebraic, polarized $l$-adic representation of $G_F$ satisfying the following properties:
   \begin{enumerate}
  \item The reduction $\barr$ is irreducible and 
 $\barr(G_{F(\zeta_l)}) \subset \GL_n(\barFF_l)$ is adequate.
  \item $(\barr,\barmu)$ is automorphic of level potentially prime to $l$, arising from a regular algebraic, cuspidal, polarized automorphic representation $(\pi,\chi)$ of level potentially prime to $l$, such that
  \[  r_{l,\imath}(\pi)|_{G_{F_v}}\sim r|_{G_{F_v}} \]
  for each finite place $v$ of $F$. (In particular, $r|_{G_{F_v}}$ is potentially crystalline for $v|l$, and $r$ has the same Hodge--Tate numbers as $r_{l,\imath}(\pi)$.)
\end{enumerate}
    Then $(r,\mu)$ is automorphic of level potentially prime to $l$.
    
    If further $\pi$ has level prime to $l$ 
     and if $r$ is crystalline at all primes above $l$, then $(r,\mu)$ is automorphic of level prime to $l$.
    \end{thm}

\begin{proof}
The result follows from Theorem 7.1 of \cite{jack} on noting that for $v \ndiv l$ we have $r_{l,\imath}(\pi)|_{G_{F_v}}\leadsto r|_{G_{F_v}}$ (by Lemma \ref{gen}, because $\imath \WD(r_{l,\imath}(\pi)|_{G_{F_v}})^{F-\semis} \cong \rec(\pi_v \otimes |\det|_v^{(1-n)/2})$ and $\pi_v$ is generic). 
\end{proof}

\begin{thm}\label{fdmin} Let $F$ be an imaginary CM field with maximal totally real subfield $F^+$. Suppose that 
$n \in \Z_{\geq 1}$ and that $l$ is an odd prime. Assume also that
$\zeta_l\notin F$. Let $S$ be a finite
set of primes of $F^+$ including all primes above $l$.  Suppose moreover that each prime in $S$ splits in $F$ and choose a prime $\tv$ of $F$ above each $v\in S$. Write $\tS$ for the set of $\tv$ with $v \in S$.

Let $(\pi,\chi)$ be a regular algebraic, cuspidal, polarized automorphic representation of
$\GL_n(\A_F)$ which is unramified outside
$S$ and has level potentially prime to $l$. Let $a\in (\Z^n)^{\Hom(F,\C),+}$ be the weight of
$\pi$. Suppose that the image $ \barr_{l,\imath}(\pi)
(G_{F(\zeta_l)})$ is adequate.

Suppose, for each $v|l$, that $\cC_v$ is an irreducible
component of 
\[ \lim_{\ra K'}\Spec R_{\rbar_{l,\imath}(\pi)|_{G_{F_\tv}}, \{ \{a_{\imath \tau,i}+n-i \}_i \}_\tau,K'-\cris}^\Box\otimes\Qlbar\]
 containing $r_{l,\imath}(\pi)|_{G_{F_\tv}}$. 
Suppose also that for each $v\in S$ with $v \ndiv l$, $\cC_v$ is the irreducible
component of $\Spec R_{\rbar_{l,\imath}(\pi)|_{G_{F_\tv}}}^\Box\otimes\Qlbar$ containing $r_{l,\imath}(\pi)|_{G_{F_\tv}}$.

Let $L$ denote a finite extension of $\Q_l$ in $\barQQ_l$ such that $L$ contains the image of each
embedding $F \into  \barQQ_l$; and $L$ contains the image of $r_{l,\imath}(\chi)$;
and $r_{l,\imath}(\pi)$ is defined over $L$; and each $\cC_v$ is
defined over $L$.
For $v\in S$ let $\calD_v$ be the deformation problem corresponding to
$\cC_v$. Also let\[ \CS=(F/F^+,S,\tS,\CO_L,\widetilde{\barr}_{l,\imath}(\pi), r_{l,\imath}(\chi)\epsilon_l^{1-n}, \{ \calD_v\}_{v \in S}). \]

Then the ring $R_\CS^\univ$ is a finitely generated $\CO_L$-module.
\end{thm}
\begin{proof}
Note that $r_{l,\imath}(\pi)|_{G_{F_\tv}}$ lies on a unique component of $\Spec R_{\rbar_{l,\imath}(\pi)|_{G_{F_\tv}}}^\Box\otimes\Qlbar$ for each $v \in S$ with $v \ndiv l$. (Use Lemma \ref{gen} and the fact that $\imath \WD(r_{l,\imath}(\pi)|_{G_{F_\tv}})^{F-\semis} \cong \rec(\pi_\tv \otimes |\det|_\tv^{(1-n)/2})$.)
Also by making a base change to a finite, soluble, Galois, CM extension
  $F'/F$ which is linearly disjoint from $\barF^{\ker
    \barr_{l,\imath}(\pi)}(\zeta_l)$ over $F$ we may suppose that
  $\pi$ is unramified above $l$ and that $\cC_v$ is a component of the spectrum
  $\Spec R_{\rbar_{l,\imath}(\pi)|_{G_{F_\tv}}, \{ \{a_{\imath
      \tau,i}+n-i \}_i \}_\tau,\cris}^\Box\otimes\Qlbar$ for each
  $v|l$. (Use Lemma \ref{cdr}).  In particular the character $\chi$ is
  unramified above $l$ (as $F/F^+$ is unramified above $l$). The
  result now follows from Theorem 10.1 of \cite{jack}.
\end{proof}

\subsection{Automorphy lifting: the ordinary case.}\label{ordc}{$\mbox{}$} \newline

One can combine the Taylor--Wiles--Kisin method with the level changing method of \cite{tay} and Hida theory, to derive a stronger theorem in the 
ordinary case. This theorem allows for changes of level and weight. The first such theorem was obtained by D.G.\ (see
Theorem 5.3.2 of \cite{ger}). The `bigness' condition in Theorem 5.3.2
of \cite{ger} was relaxed by Thorne. The theorem we present below, in the case that $F$ is imaginary, is 
Theorem 9.1 of \cite{jack}. The case that $F$ is totally real follows immediately from the case that $F$ is imaginary by base change.

\begin{thm} \label{ger} 
 Suppose that $F$ is a CM (or totally real) field; 
that $l$ is an odd prime and that $n\in
  \bb{Z}_{\geq 1}$.  Let $(r,\mu)$ be an $n$-dimensional, algebraic, polarized $l$-adic representation of $G_F$ satisfying the following properties:
   \begin{enumerate}
  \item The reduction $\barr$ is irreducible and 
 $\barr(G_{F(\zeta_l)}) \subset \GL_n(\barFF_l)$ is adequate.
 \item $\zeta_l \not\in F$.
 \item $r$ is ordinary at all primes above $l$.
  \item $(\barr,\barmu)$ is ordinarily automorphic.
\end{enumerate}
Then $(r,\mu)$ is ordinarily automorphic.  If $r$ is also crystalline (resp.\ potentially crystalline) then $(r,\mu)$ is ordinarily automorphic of level prime to $l$ (resp.\ potentially level prime to $l$). 
\end{thm}

The next result is Theorem 10.2 of \cite{jack}, which  generalizes Corollary 4.3.3 of
\cite{gg}. It provides a finiteness theorem for universal deformation rings.

\begin{thm}\label{fdrord} Let $F$ be an imaginary CM field with maximal totally real subfield $F^+$. Suppose that 
  $n \in \Z_{\geq 1}$ and that $l$ is an odd prime with $\zeta_l
  \not\in F$. Let $S$ be a finite set of primes of $F^+$ including all
  primes above $l$. Suppose moreover that each prime in $S$ splits in
  $F$ and choose a prime $\tv$ of $F$ above each $v \in S$. Write $\tS$ for the set of $\tv$ for $v \in S$.

Let $(\pi,\chi)$ be an $\imath$-ordinary, regular algebraic, cuspidal, polarized  automorphic representation of $\GL_n(\A_F)$ which is unramified outside $S$. Suppose that the image $\barr_{l,\imath}(\pi) (G_{F(\zeta_l)})$ is adequate.

Let 
 \[ \mu: G_{F^+} \lra \barQQ_l^\times \]
be an algebraic character satisfying $\barmu = \barr_{l,\imath}(\chi)\barepsilon_l^{1-n}$. Note that $\HT_\tau(\mu)=\{ w \}$ is independent of $\tau:F^+ \into \barQQ_l$. For each $\tau:F \into \barQQ_l$ choose a multiset of $n$ distinct integers $H_\tau$ such that 
 \[ H_{\tau \circ c} = \{ w-h: \,\, h \in H_\tau\}. \]
 
Let $L$ denote a finite extension of $\Q_l$ in $\barQQ_l$ such that $L$ contains the image of each
embedding $F \into  \barQQ_l$; and $L$ contains the image of $\mu$; and $r_{l,\imath}(\pi)$ is defined over $L$.
For $v\in S$ with $v\ndiv l$ let $\calD_v$ consist of all lifts of $\barr_{l,\imath}(\pi)|_{G_{F_\tv}}$. If $v|l$ let $\calD_v$ consist of all lifts which factor through $R^\Box_{\CO, \barr_{l,\imath}(\pi)|_{G_{F_\tv}}, \{ H_\tau\}, \ssord}$. 
Also let
\[ \CS=(F/F^+,S,\tS,\CO_L,\widetilde{\barr}_{l,\imath}(\pi), \mu, \{ \calD_v\}_{v \in S}). \]

Then the ring $R_\CS^\univ$ is a finitely generated
$\CO_L$-module.
\end{thm}

\newpage

     \section{Potential Automorphy.}\label{sec:potential automorphy}

\subsection{The Dwork family.}\label{sec:Dwork}  {$\mbox{}$} \newline

In this section we show that a suitable symplectic, $\bmod l$ representation is potentially automorphic. 
The theorem and its proof are slight generalizations of section
6 of \cite{blght} The arguments are also simpler because of the stronger automorphy lifting theorems that we now have available, particularly \cite{jack}. (See in particular step 2 of the proof of Theorem 6.3 of \cite{blght}.) We start with another minor variant of a result of Moret-Bailly \cite{mb} (see also \cite{gpr} and Proposition 6.2 of \cite{blght}).

\begin{prop}\label{mbp} Let $\Kavoid/K/K_0$ be number fields with $\Kavoid/K$ and $K/K_0$ Galois. Suppose also that $S$ is a finite set of places of $K_0$ and let $S^K$ denote the set of places of $K$ above $S$. For $v\in S^K$ let $L_v'/K_v$
be a finite Galois extension with $L_{\sigma v}'=\sigma L_v'$
  for $\sigma \in G_{K_{0,v|_{K_0}}}$. Suppose also
that $T/K$ is a smooth, geometrically connected variety and that for each $v \in S^K$ we are given a non-empty, $\Gal(L_v'/K_v)$-invariant, open subset $\Omega_v \subset T(L_v')$. 

Then there is a finite Galois extension $L/K$ and a point $P \in T(L)$ such that
\begin{itemize}
\item $L/K_0$ is Galois;
\item $L/K$ is linearly disjoint from $\Kavoid/K$;
\item if $v \in S^K$ and $w$ is a prime of $L$ above $v$ then $L_w/K_v$ is isomorphic to $L_v'/K_v$ and $P\in \Omega_v \subset  T(L_v') \cong T(L_w)$. (This makes sense as $\Omega_v$ is $\Gal(L_v'/K_v)$-invariant.)
\end{itemize} \end{prop}

\begin{proof} Let $\Kavoid_1,\dots,\Kavoid_r$ denote the intermediate fields between $\Kavoid$ and $K$ with $\Kavoid_i/K$ Galois with simple Galois group.
  Combining Hensel's lemma with the Weil bounds we see that $T$ has a
  $K_v$-rational point for all but finitely many primes $v$ of
  $K$. Thus enlarging $S$ we may assume that for each $i=1,\dots,r$
  there is $v \in S^K$ with $L_v'=K_v$ and $v$ not split completely in
  $\Kavoid_i$.  Then we may suppress the second condition on $L$.

Let $K'/K$ be a finite extension such that
\begin{itemize}
\item $K'/K_0$ is Galois;
\item if $v \in S^K$ and $w|v$ is a place of $K'$ then $K_w'/K_v$ is isomorphic to $L_v'/K_v$.
\end{itemize}
(Apply Lemma \ref{l412} with $F$ of that lemma our $K$ and $S$ of that lemma our $S^K$. This
produces a soluble extension $K''/K$. Then we take $K'$ to be the normal closure of $K''$ over $K_0$.)
Thus we may assume that $L_v'=K_v$ for all $v \in S^K$. 

Then Theorem 1.3 of \cite{mb} tells us that we can find a finite Galois extension $K'/K$ and a point $P \in T(K')$ such that 
 every place $v$ of $S^K$ splits completely in $K'$ and if $w$ is a prime of $K'$
above $v $ then $P \in \Omega_v \subset T(K_w')$.
Now take $L$ to be the normal closure of $K'$ over $K_0$.
\end{proof}

\begin{thm} \label{prop:Dwork}Suppose that:
\begin{itemize}
\item $F/F_0$ is a finite, Galois extension of totally real fields, 
\item $\CI$ is a finite set,
\item for each $i \in \CI$, $n_i$ is a positive even integer,
  $l_i$ is an odd rational prime, and $\imath_i:\barQQ_{l_i} \iso
  \C$,\item $\Favoid/F$ is a finite Galois extension, and 
\item $\bar{r}_i: G_{F} \to \GSp_{n_i}(\Fbar_{l_i})$ is a mod $l_i$ Galois representation with open kernel 
and multiplier $\barepsilon_{l_i}^{1-n_i}$.
\end{itemize}
Then we can find a finite totally real extension $F'/F$ and for each $i \in \CI$ a regular algebraic, cuspidal, polarized automorphic representation $(\pi_i,\chi_i)$ of $\GL_{n_i}(\A_{F'})$ such that
\begin{enumerate}
\item $F'/F_0$ is Galois,
\item $F'$ is linearly disjoint from $\Favoid$ over $F$,
\item $(\barr_{l_i,\imath_i}(\pi_i),\barr_{l_i,\imath_i}(\chi_i)\barepsilon_{l_i}^{1-n_i})\cong ({\barr}_i|_{G_{F'}},
  \barepsilon_{l_i}^{1-n_i})$ for each $i\in\CI$;
\item $\pi_i$ is $\imath_i$-ordinary of weight 0 for each $i\in\CI$.
\end{enumerate}
\end{thm}

\begin{proof} The proof follows the proof of Theorem 6.3 of \cite{blght}, although the proof here is simpler. 

First note that $\barr_i$ is actually valued in $\GSp_n(\F^{(i)})$ for some finite cardinality
subfield $\F^{(i)} \subset \barFF_{l_i}$. We choose a positive integer $N$ such that 
\begin{itemize}
\item $N$ is coprime to $2 \prod_i l_i$,
\item $N>n_i+1$ for all $i$,
\item $N$ is not divisible by any rational prime which ramifies in $\Favoid$,\item and for each $i \in \CI$ there is a prime $\lambda_i$ of $\Q(\zeta_N)^+$ above $l_i$ and an embedding $\F^{(i)} \into \Z[\zeta_N]^+/\lambda_i$. 
\end{itemize}
(Use Lemma 6.1 of \cite{blght}.) Note that in particular $\Favoid$ is linearly disjoint from $\Q(\zeta_N)$ over $\Q$.

We next choose an imaginary CM field $M_i$ for each $i \in \CI$ such that
 $M_i/\Q$ is cyclic of degree $n_i$ and unramified at all rational primes which ramify in $\Favoid$. For each $i$ let $\tau_i$ denote a generator of $\Gal(M_i/\Q)$. Choose a rational prime $q$ such that
\begin{itemize}
\item $q$ splits completely in $\prod_i M_i$,
\item and $q$ is unramified in $F (\zeta_{4N})$. 
\end{itemize}
Also choose primes $\gq_i$ of $M_i$ above $q$ for each $i$. Choose $M'$ containing the compositum of the $M_i$'s and, for each $i$, a character $\phi_i:\A_{M_i}^\times \ra (M')^\times$ with open kernel such that
\begin{itemize}
\item if $\alpha \in M_i^\times$ then
\[ \phi_i(\alpha)=\prod_{j=0}^{n_i/2-1} \tau^j_i(\alpha)^j\tau_i^{n_i/2+j}(\alpha)^{n_i-1-j}; \]
\item $\phi_i\phi_i^c=\prod_{v \ndiv \infty} |\,\,\,|_v^{1-n_i}$;
\item $\phi_i$ is unramified above all rational primes which ramify in $F/\Q$;
\item $\phi_i$ is unramified above $N$;
\item $\phi_i$ is unramified at all primes above $q$ except $\gq_i$ and $\gq_i^c$, but $q|\# \phi_i(\CO_{M_i,\gq_i}^\times)$.
\end{itemize}
(Apply Lemma \ref{l22}. We take $S$ to be the set of primes of $M_i$ above $Nq$ or any rational prime that ramifies in $F$. We also take $\psi_S=\prod_{v \in S} \psi_v$, where $\psi_v$ is the trivial character unless $v|\gq_i\gq_i^c$. Moreover we choose $\psi_{\gq_i}$ to be wildly ramified and take $\psi_{\gq^c_i}=(\psi_{\gq_i}^c)^{-1}$. )

Next choose a rational prime $l'$ such that
\begin{itemize}
\item $l'$ splits completely in $M'(\zeta_N)$ (so that in particular $l' \equiv 1 \bmod N$);
\item $l'$ is unramified in $F$;
\item $l'\ndiv 6qN\prod_i(l_in_i)$.
\end{itemize}
In particular $\zeta_{l'} \not\in F$.
Also choose primes $\lambda_{M'}'$ of $M'$ and $\lambda'$ of $\Q(\zeta_N)^+$ above $l'$. For each $i$ let
\[ \theta_i: G_{M_i} \lra \CO_{M',\lambda_{M'}'}^\times = \Z_{l'}^\times \]
be the algebraic character defined by
\[ \theta_i(\Art_{M_i} \alpha)=\phi_i(\alpha) \prod_{j=0}^{n_i/2-1}\tau_i^j(\alpha_{l'})^{-j} \tau_i^{n_i/2+j}(\alpha_{l'})^{j+1-n_i}. \]
Note that $\theta_i\theta_i^c=\epsilon_{l'}^{1-n_i}$. Let $r'_i=\Ind_{G_{M_i}}^{G_\Q} \theta_i$ and note that
\[ r'_i \cong (r_i')^\vee \otimes \epsilon_{l'}^{1-n_i}. \]

Choose a lifting $\ttau_i$ of $\tau_i$ to $\ker \epsilon_{l'} \subset G_\Q$. Then
\[ \theta_i(\ttau_i^n) = \theta_i(c(c \ttau_i^{n_i/2})c(c \ttau_i^{n_i/2})) = (\theta \theta^c)(c \ttau_i^{n_i/2}) = \epsilon_{l'}^{1-n_i}(c \ttau_i^{n_i/2})=-1. \]
The module underlying $r_i'$ has a $\Z_{l'}$-basis $e_0,e_1,\dots,e_{n_i-1}$ such that
\begin{itemize}
\item $r_i'(\sigma) e_j = \theta_i^{\tau_i^j}(\sigma) e_j$ for all $\sigma \in G_M$;
\item $r_i'(\ttau_i) e_i = e_{i-1}$ for $i=1,\dots,n_i-1$;
\item and $r_i'(\ttau_i)e_0 = \theta_i(\ttau_i^{n_i})e_{n_i-1}=- e_{n_i-1}$.
\end{itemize}
If we define a perfect pairing on $r_i'$ by setting
\[ \langle e_{j_1},e_{j_2} \rangle = \left\{ \begin{array}{ll} 1 & {\rm if}\,\, j_2=j_1+n_i/2 \\ -1 & {\rm if}\,\, j_1=j_2+n_i/2 \\ 0 & {\rm otherwise}, \end{array} \right. \]
we see that this pairing is preserved by $r_i'$ up to scalar multiplication by $\epsilon_{l'}^{1-n_i}$. 
Thus $r_i'$ factors through $\GSp_{n_i}(\Z_{l'})$ with multiplier $\epsilon_{l'}^{1-n_i}$.

Let $\bartheta_i:G_{M_i} \ra \F_{l'}^\times$ denote the reduction of $\theta_i$ and let
$\barr_i'$ denote the reduction of $r_i'$. We have the following observations.
\begin{itemize}
\item $\bartheta_i^{\tau^j}|_{G_{M_i(\zeta_{l'})}} \neq \bartheta_i^{\tau^{j'}}|_{G_{M_i(\zeta_{l'})}}$ for $j \neq j'$ in the range $0,1,\dots,n_i-1$. (Look at the ramification above $q$.)
\item $\barr_i'|_{\Q(\zeta_{l'})}$ is irreducible. 
\item $l' \ndiv \# \barr_i' G_\Q$.
\item $\barr'_i(G_{\Q(\zeta_{l'})})$ is adequate. (By Proposition \ref{ghtt}.)
\item $\Q(\zeta_N)$ is linearly disjoint over $\Q$ from $\Favoid/\Q$. (Because no rational prime ramifies in both fields.)
\item $\barr_i'(G_{\Q(\zeta_{l'})})=\barr_i'(G_{F(\zeta_{Nl'})})$. (Because $\barr_i'$ is only ramified at primes which are unramified in $F(\zeta_N)$.)
\end{itemize} 

Let $T_0/\Spec F(\zeta_N)^+$ denote the scheme $\PP^1-(\{ \infty\} \cup \mu_N)$. For each $i$ there are
\begin{itemize}
\item lisse $\Z[\zeta_N]^+_{\lambda_i}$ (resp.\ $\Z[\zeta_N]^+_{\lambda'}$) sheaves $V_{n_i,\lambda_i}((N-1-n_i)/2)$ (resp.\ $V_{n_i,\lambda'}((N-1-n_i)/2)$) over $T_0$;
\item locally free etale sheaves $V_{n_i}[\lambda_i]((N-n_i-1)/2)$ (resp.\ $V_{n_i}[\lambda']((N-n_i-1)/2)$) of  $\Z[\zeta_N]^+/\lambda_i$ (resp.\ $\Z[\zeta_N]^+/\lambda'$) modules;
\item a finite cover $T_{\barr_i \times \barr_i'}/(T_0\times \Spec F(\zeta_N)^+)$; 
\end{itemize}
constructed as in section 4 of \cite{blght} using $N=N$ and $n=n_i$. (We have added the subscript $n_i$ to the notation of Section 4 of \cite{blght} as a reminder that we are taking $n=n_i$ in the constructions of that Section. We would like to point out that in the definition of $T_W$ on page 54 of \cite{blght} we should have specified that it represents isomorphisms $W_S \iso V[\gn]((N-1-n)/2)_S$ {\em compatible with the symplectic structures}. We would also like to point out that the fourth occurrence of $S_2$ in the statement of Proposition 6.2 of \cite{blght} should be an $M$. We thank Kevin Buzzard for pointing out these corrections to \cite{blght}.)

 Let $\tT$ denote the product of the $T_{\barr_i \times \barr_i'}$ over $\Spec F(\zeta_N)^+$ and let $t_i$ denote the $i^{th}$ projection to $T_0$. By Proposition 4.2 of \cite{blght} we see that $\tT$ is geometrically irreducible. By Proposition \ref{mbp} we can find a
finite extension of $F'/F(\zeta_N)^+$ and a point $P \in \tT(F')$ such that
\begin{itemize}
\item $F'/F_0$ is Galois;
\item $F'$ is totally real;
\item $F'$ is linearly disjoint from $\Favoid \barF^{\cap_i \ker
     \barr_i'}(\zeta_{Nl'})$ over
  $F(\zeta_N)^+$;
\item $v(t_i(P))<0$ for all places $v|l_i$ of $F'$;
\item $v(t_i(P))>0$ for all places $v|l'$ of $F'$.\end{itemize}
(The only thing we need to check is that $T_{\barr_i \times \barr_i'}(F(\zeta_N)^+_v)\neq \emptyset$ for each $i$ and each $v|\infty$. However, because $\GSp_{n_i}(\Z/l_il'\Z)$ has a unique conjugacy class of elements of order $2$ and multiplier $-1$, we see that every $F(\zeta_N)^+_v$-point of $T_0$ lifts to a $F(\zeta_N)^+_v$-point of $T_{\barr_i \times \barr_i'}$.)

Then we have the following observations.
\begin{itemize}
\item $F'$ is linearly disjoint from $\Favoid$ over $F$ (as $F(\zeta_N)^+$ is linearly disjoint from $\Favoid$ over $F$).
\item $V_{n_i}[\lambda_i]((N-1-n_i)/2)_{t_i(P)} \cong \barr_i|_{G_{F'}}$.
\item $V_{n_i}[\lambda']((N-1-n_i)/2)_{t_i(P)} \cong \barr_i'|_{G_{F'}}$.
\item $\barr_i'(G_{F'(\zeta_{l'})})$ is adequate.
\item $\zeta_{l'} \not\in F'$ (as $F'$ is totally real and $l'>2$).
\item $V_{n_i,\lambda'}((N-1-n_i)/2)_{t_i(P)}$ is ordinary at all primes above $l'$. (See Lemma 5.3(3) of \cite{blght}.)
\item $\HT_\tau(V_{n_i,\lambda'}((N-1-n_i)/2)_{t_i(P)})=\{0,1,\dots,n_i-1\}$ for all $\tau:F' \into \barQQ_{l'}$. (See Lemma 5.3(1) of \cite{blght}.)
\item If $v$ is a place of $F'$ above $l_i$ then 
\[ \imath'\WD(V_{n_i,\lambda'}((N-1-n_i)/2)_{t_i(P)}|_{G_{F'_v}})\cong  \rec_{F'_v} (\Spp_{n_i}(\phi_i)) \]
for some unramified character $\phi_i$  (and for any isomorphism $\imath':\barQQ_{l'} \iso \C$). (See Lemma 5.1(2) of \cite{blght}.)
\end{itemize}

From Theorem 4.2 of \cite{MR1007299} we see that $(r_i'|_{G_{F'}},\epsilon_{l'}^{1-n_i})$ is automorphic of level potentially prime to $l'$, and hence ordinarily automorphic. By Theorem \ref{ger}  we conclude that $(V_{n_i,\lambda'}((N-1-n_i)/2)_{t_i(P)},\epsilon_{l'}^{1-n_i})$ is automorphic over $F'$ of weight $0$, arising from a regular algebraic, cuspidal, polarized automorphic representation $(\pi_i,1)$ with $\pi_{i,v}$ Steinberg for all $v|l_i$, and from an isomorphism $\imath_i':\barQQ_{l_i} \iso \C$. Thus  $\pi_i$ is $\imath_i$-ordinary. The Theorem follows.
\end{proof}

\subsection{Lifting Galois representations I}\label{lgr1} {$\mbox{}$} \newline

We now use the method of Khare and Wintenberger \cite{kw} to show that
certain mod $l$ representations have ordinary lifts with prescribed
local behavior. We will later improve upon this by weakening the
ordinary hypothesis (see Theorem \ref{diaglift}), but we will need to
use this special case before we are in a position to prove the more
general result.

 Let $n$ be a positive integer and $l$ an odd  prime. Suppose that $F$
 is an imaginary CM field not containing $\zeta_l$ and with maximal totally real subfield $F^+$.
Let $S$ be a finite set of finite places of $F^+$ which split in $F$ and suppose that $S$ includes all places above $l$. For each $v \in S$ choose a prime $\tv$ of $F$ above $v$. 

Let $\mu:G_{F^+} \ra \barQQ_l^\times$ be a continuous, crystalline character unramified outside $S$ such that $\mu(c_v)=-1$ for all $v|\infty$. Then there is a $w \in \Z$ such that for each $\tau:F^+ \into \barQQ_l$ we have $\HT_\tau(\mu)=\{ w \}$. For each $\tau:F \into \barQQ_l$ let $H_\tau$ be a set of $n$ distinct integers such that $H_{\tau \circ c} = \{ w-h: \,\,\, h\in H_\tau\}$. 

Let
\[ \barr: G_{F^+} \lra \CG_n(\barFF_l) \]
be a continuous representation unramified outside $S$ with $\nu \circ
\barr = \barmu$ and $\barr^{-1}\CG^0_n(\barFF_l)=G_F$. For $v \in S$ with $v\ndiv l$ let $\rho_v:G_{F_\tv} \ra \GL_n(\CO_{\barQQ_l})$ denote a lift of $\breve{\barr}|_{G_{F_\tv}}$.

\begin{prop}\label{ordlift}
Keep the notation and assumptions already stated in this section. Also make the following additional assumptions:
\begin{enumerate}
\renewcommand{\labelenumi}{(\alph{enumi})}
\item  Suppose that $\breve{\barr}|_{G_{F(\zeta_l)}}$ is irreducible. Also, writing $d$ for the maximal dimension of an irreducible constituent of the restriction of $\breve{\barr}$ to the closed subgroup of $G_{F^+}$ generated by all Sylow pro-$l$-subgroups, suppose that $l\geq 2(d+1)$.

\item Suppose  that for $u|l$ a place of $F$ 
the restriction $\breve{\barr}|_{G_{F_u}}$ admits a lift $\rho_u:G_{F_u} \ra \GL_n(\CO_{\barQQ_l})$ which is ordinary and crystalline with Hodge--Tate numbers $H_\tau$ for each $\tau:F_u \into \barQQ_l$. 
\end{enumerate}

Then there is a lift
\[ r:G_{F^+} \lra \CG_n(\CO_{\barQQ_l}) \]
of $\barr$ such that
\begin{enumerate}
\item $\nu \circ r = \mu$;
\item if $u|l$ is a place of $F$ then $\breve{r}|_{G_{F_u}}$ is ordinary and crystalline with Hodge--Tate numbers
$H_\tau$ for each $\tau:F_u \into \barQQ_l$;
\item if $v \in S$ and $v \ndiv l$ then $\breve{r}|_{G_{F_\tv}} \sim
  \rho_v$;\item $r$ is unramified outside $S$.
\end{enumerate} \end{prop}

\begin{proof}
Choose a place $v_q$ of $F$ above an odd rational prime $q$ such that $v_q$ is split over $F^+$ and $v_q$ does not divide any prime in $S$. Also choose integers $b_\tau$ for all $\tau:F \into \barQQ_l$ such that 
\begin{itemize}
\item $b_\tau+b_{\tau \circ c}=2n-1-w$ for all $\tau$,
\item and $|b_\tau - b_{\tau\circ
    c}|>|h-h'|$ for all $\tau$ and for all $h\in H_\tau, h'\in H_{\tau\circ c}$.
    \end{itemize}
Now choose a character $\psi:G_F \ra \barQQ_l^\times$ such that 
\begin{itemize}
\item $\psi|_{G_{F_\tv}}$ is unramified if $v \in S$ but $v\ndiv l$;
\item $\psi$ is crystalline at all primes above $l$ with $\HT_\tau(\psi)=\{b_\tau\}$ for all $\tau:F \into \barQQ_l$; 
\item $q|\#(\psi/\psi^c)(I_{F_{v_q}})$; and
\item $\psi \psi^c=\epsilon_l^{1-2n} \mu^{-1}|_{G_F}$. 
\end{itemize}
(To do this, apply Lemma \ref{l416} with the set $S$ of that Lemma equal to the primes of $F$ above our $S$, plus $v_q$ and $v_q^c$. For $v$ in this set take $\psi_v$ as follows:
\begin{itemize}
\item if $v \in S$ but $v\ndiv l$ then $\psi_\tv=1$ and $\psi_{\tv^c}=\epsilon_l^{1-2n}\mu^{-1}|_{G_{F_{\tv^c}}}$,
\item $\psi_{v_q}$ is a wildly ramified character and $\psi_{v_q^c}=(\psi_{v_q}^c)^{-1}$,
\item if $v|l$ then $\psi_v$ is crystalline with $\HT_\tau(\psi_v)=\{ b_\tau\}$ for all continuous $\tau:F_v \into \barQQ_l$.) 
\end{itemize}

In the notation of Section \ref{cgn}, we have a homomorphism $
(\barpsi, \barepsilon_l^{1-2n}
\barmu^{-1}\delta_{F/F^+}):G_{F^+}\to\cG_1(\Flbar)$, and we can
consider the representations $\barr \otimes (\barpsi,
\barepsilon_l^{1-2n}
\barmu^{-1}\delta_{F/F^+}):G_{F^+}\to\cG_n(\Flbar)$ and  $I_{\psibar}(\rbar):=I(\barr
\otimes (\barpsi, \barepsilon_l^{1-2n} \barmu^{-1}\delta_{F/F^+})):
G_{F^+} \ra \GSp_{2n}(\barFF_l)$. Note that $I_{\psibar}(\rbar)$ has multiplier $\barepsilon_l^{1-2n}$. By the third condition
on $\psi$ the representation $I_{\psibar}(\rbar)|_{G_{F^+(\zeta_l)}}$
is irreducible. (As it is the induction of an irreducible
representation from the index 2 subgroup $G_{F(\zeta_l)}$, it suffices to check that
the restriction to $G_{F(\zeta_l)}$ is not the sum of two isomorphic
representations, and this follows, as the two representations
differ when restricted to $I_{F_{v_q}}$.)  By Proposition \ref{ghtt}, $I_{\psibar}(\rbar)(G_{F^+(\zeta_l)})$ is adequate.  

Let $F_0/F^+$ be a totally imaginary quadratic extension linearly disjoint from $\barF^{\ker I_{\psibar}(\rbar)}(\zeta_l)$ over $F^+$. By Theorem \ref{prop:Dwork} there is a Galois totally
real field extension $F_1^+/F^+$ and a regular algebraic, cuspidal, polarized automorphic representation $(\pi_1,1)$ of
$\GL_{2n}(\A_{F_1^+})$ such that 
\begin{itemize}
\item $F_1^+$ is linearly disjoint from $\barF^{\ker  I_{\psibar}(\rbar)}F_0(\zeta_l)$ over $F^+$;
\item $\barr_{l,\imath}(\pi_1) \cong I_{\psibar}(\barr)|_{G_{F_1^+}}$;
\item and $\pi_1$ is $\imath$-ordinary.
\end{itemize}
Set $F_1=F_0F_1^+$. It is linearly disjoint from
$\barF^{\ker I_{\psibar}(\barr)}F_1^+(\zeta_l)$
over $F_1^+$. Set (again, in the notation of Section \ref{cgn})
\[ \barr_1= (I_{\psibar}(\barr )|_{G_{F_1^+}})^\wedge_{G_{F_1}}:G_{F_1^+}
\lra \CG_{2n}(\barFF_l). \]
Then $\breve{\barr}_1(G_{F_1(\zeta_l)})$ is adequate and $\zeta_l \not\in F_1$.

Let $T' \supset S$ denote a finite set of primes of $F^+$ including
all those above which $\barpsi$, $\pi_1$ or $F_1$ is ramified. Let
$F_2^+/F^+$ be a finite soluble Galois totally real extension,
linearly disjoint from $\barF_1^{\ker \breve{\barr}_1}(\zeta_l)$ over
$F^+$ such that all primes of $F_3^+=F_1^+F_2^+$ above $T'$ split in
$F_3=F_1F_2^+$. (We have introduced $F_2^+$ in order to be able to apply Theorem \ref{fdrord}. The primes of $F_1^+$ above $T'$ may not split in $F_1$.) Set
\[ \barr_3=\barr_1|_{G_{F_3^+}}:G_{F_3^+}
\lra \CG_{2n}(\barFF_l) \]
so that  $\barr_3^{-1} \CG_{2n}^0(\barFF_l)=G_{F_3}$. Then
$\breve{\barr}_3(G_{F_3}(\zeta_l))$
is adequate and $\zeta_l\not\in F_3$.
Let $T$ denote the set of places of $F_3^+$ lying over $T'$ and for each $u \in T$ choose a prime $\tu$ of $F_3$ above 
$u$ and let $\tT$ denote the set of $\tu$ for $u \in T$.

For $v \in S$ with $v \ndiv l$ let $\CC_v$ denote a component of $R^\Box_{\breve{\barr}|_{G_{F_\tv}}}
\otimes \barQQ_l$ containing $\rho_v$. Choose a finite extension $L$ of $\Q_l$ in $\barQQ_l$ with
integers $\CO$ and residue field $\F$ such that
\begin{itemize}
\item $L$ contains the image of each embedding $F_3 \into \barQQ_l$;
\item for $v \in S$ the component $\CC_v$ is defined over $L$;
\item $\barr$ and $\barpsi$ are defined over $\F$;
\item and $\mu$ is defined over $L$.
\end{itemize}
For $v \in S$ with $v \ndiv l$ let $\calD_v$ denote the deformation problem for $\breve{\barr}|_{G_{F_\tv}}$
corresponding to $\CC_v$. For $v \in S$ with $v|l$ let $\calD_v$ consist of all lifts of $\breve{\barr}|_{G_{F_\tv}}$
which factor through $R^\Box_{\CO,\breve{\barr}|_{G_{F_\tv}},\{
  H_\tau\},\crord}$.
Set
\[ \CS=(F/F^+, S, \tS, \CO, \barr, \mu, \{ \calD_v\}_{v \in S}). \]
For $u \in T$ with $u \ndiv l$ let $\calD_{3,u}$ consist of all lifts of $\breve{\barr}_3|_{G_{F_{3,\tu}}}$. For $u \in T$
with $u|l$ let $\calD_{3,u}$ consist of all lifts of $\breve{\barr}_3|_{G_{F_{3,\tu}}}$
which factor through $R^\Box_{\CO,\breve{\barr}_3|_{G_{F_{3,\tu}}},\{ H_{3,\tau}\},\ssord}$, where 
\[ H_{3,\tau} = \{ h + b_{\tau_1}:\,\, h \in H_{\tau_1}\} \cup \{ h+ b_{\tau_2}:\,\, h \in H_{\tau_2}\}, \]
and $\tau_1$ and $\tau_2$ denote the two embeddings of $F \into
\barQQ_l$ lying above $\tau|_{F^+}$.
Set
\[ \CS_3=(F_3/F_3^+, T, \tT, \CO, \barr_3, \epsilon_l^{1-2n}, \{ \calD_{3,u}\}_{u \in T}). \]
According to Theorem \ref{fdrord} the ring $R^\univ_{\CS_3}$ is a finitely generated $\CO$-module.
Hence by (all three parts of) Lemma \ref{cdr} the ring $R^\univ_\CS$ is also a finitely generated $\CO$-module. On the other 
hand by Proposition \ref{drb} $R^\univ_\CS$ has Krull dimension at least $1$ and so there is a continuous
ring homomorphism $R^\univ_\CS \ra \barQQ_l$. The push forward of the universal deformation of
$\barr$ by this homomorphism is our desired lift $r$.
\end{proof}

\subsection{Potential ordinary automorphy.}\label{poa} {$\mbox{}$} \newline

In this section we improve Theorem \ref{prop:Dwork} to show that a
suitable $\bmod$ $l$ representation is potentially ordinarily automorphic with prescribed ``weight and level''. The proof will combine Theorem \ref{prop:Dwork} and Proposition \ref{ordlift}. We will improve further on this result in Corollary \ref{padiag}.

\begin{prop}\label{paord} Suppose that we are in the following situation. 
\begin{enumerate}
\renewcommand{\labelenumi}{(\alph{enumi})}
\item Let $F/F_0$ be a finite, Galois extension of imaginary CM fields, and let $F^+$ and $F_0^+$ denote their maximal totally real subfields. Choose a complex conjugation $c \in G_{F^+}$.
\item Let $\CI$ be a finite set.
\item For each $i \in \CI$ let $n_i$ be a positive integer and $l_i$ be an odd rational prime with $\zeta_{l_i} \not\in F$.
Also choose $\imath_i:\barQQ_{l_i} \iso \C$ for each $i \in \CI$.
\item For each $i\in \CI$ let $\mu_i:G_{F^+} \ra \barQQ_{l_i}^\times$ be a continuous, totally odd, de Rham character.
Then there is a $w_i \in \Z$ such that for each $\tau:F^+ \into \barQQ_{l_i}$ we have $\HT_\tau(\mu_i)=\{ w_i \}$. 
\item For each $i \in \CI$ let $\barr_i: G_{F} \ra
  \GL_{n_i}(\barFF_{l_i})$ be an irreducible continuous representation
  such that $(\barr_i,\barmu_i)$ is a totally odd polarized mod $l$ representation.
Let $d_i$ denote the maximal dimension of an irreducible subrepresentation of the restriction of ${\barr}_i$ to the subgroup of $G_F$ generated by all Sylow pro-$l_i$-subgroups. Suppose that ${\barr}_i|_{G_{F(\zeta_{l_i})}}$ is irreducible and that $l_i \geq 2(d_i+1)$.
\item For each $i \in \CI$ and each $\tau:F \into \barQQ_{l_i}$ let
   $H_{i,\tau}$ be a set of $n_i$ distinct integers such that
   $H_{i,\tau \circ c} = \{ w_i-h: \,\,\, h\in H_{i,\tau}\}$.
 \item Let $S$ denote a finite $\Gal(F/F^+)$-invariant set of primes
   of $F$ including all those dividing $\prod_i l_i$ and
 all those at which some $\barr_i$ ramifies.
 \item for each $i \in \CI$ and $v \in S$ with $v\ndiv l_i$ let
   $\rho_{i,v}:G_{F_v} \ra \GL_{n_i}(\CO_{\barQQ_{l_i}})$
denote a lift of ${\barr}_i|_{G_{F_v}}$ such that $\rho_{i,cv}^c \cong \mu_i|_{G_{F_v}} \rho_{i,v}^\vee$. 
\item Let $\Favoid/F$ be a finite Galois extension.
\end{enumerate}

Then we can find a finite CM extension $F'/F$ and for each $i \in \CI$ a regular algebraic, cuspidal, polarized automorphic representation $(\pi_i,\chi_i)$ of $\GL_{n_i}(\A_{F'})$ such that
\begin{enumerate}
\item $F'/F_0$ is Galois,
\item $F'$ is linearly disjoint from $\Favoid$ over $F$,
\item $(\barr_{l_i,\imath_i}(\pi_i),\barr_{l_i,\imath_i}(\chi_i)\barepsilon_{l_i}^{1-n_i})\cong ({\barr}_i|_{G_{F'}}, \barmu_i|_{G_{(F')^+}})$;
\item $\pi_i$ is unramified above $l_i$ and outside $S$;
\item $\pi_i$ is $\imath_i$-ordinary;
\item if $\tau:F' \into \barQQ_{l_i}$ then $\HT_\tau(r_{l_i,\imath_i}(\pi_i))=H_{i,\tau|_F}$;
\item if $u\ndiv l_i$ is a prime of $F'$ lying above an element $v \in S$ then $r_{l_i,\imath_i}(\pi_i)|_{G_{F_u'}} \sim \rho_{i,v}|_{G_{F_u'}}$.
\end{enumerate}\end{prop}

\begin{proof} 
Note that $(\barr_i, \barmu_i|_{G_F})$ extends to a continuous
homomorphism $\tilde{\barr}_{i,\barmu_i}:G_{F^+} \ra \CG_{n_i}(\barFF_{l_i})$   with $\nu
\circ \tilde{\barr}_{i,\barmu_i} = \barmu_i$ (see section \ref{cgn}).

Choose a finite totally real extension $F_1^+/F^+$ so that
\begin{itemize}
\item $F_1^+/F_0^+$ is Galois;
\item $F_1^+$ is linearly disjoint from $\barF^{\cap_i \ker \barr_i}(\zeta_{\prod_il_i})\Favoid$ over $F^+$;
\item all places of $F_1=FF_1^+$ above $S$ are split over $F^+_1$;
\item and for all $i\in \CI$ and all places $u|l_i$ of $F_1$ 
the restriction ${\barr}_i|_{G_{F_{1,u}}}$ admits a lift $\rho_{i,u}:G_{F_{1,u}} \ra \GL_{n_i}(\CO_{\barQQ_{l_i}})$ which is ordinary and crystalline with Hodge--Tate numbers $H_{i,\tau|_F}$ for each $\tau:F_{1,u} \into \barQQ_l$.
\end{itemize}
(For all $v$ a prime of $F^+$ below an element of $S$ there is a finite Galois extension $E_v/F_v^+$ with the following property: The last two bullet points will be satisfied as long as, for all primes $w$ of $F_1^+$ above a prime $v$ of $F^+$ below an element of $S$, we have $(F_1^+)_w \supset E_v$. So we may replace the last two bullet points by this condition. Now the existence of $F_1^+$ follows from Lemma \ref{l412}.)
Replacing $F$ by $F_1$ (and $\Favoid$ by $F_1 \Favoid$) we may reduce the theorem to the special case that all elements of $S$ are split over $F^+$ and that for all $i\in \CI$ and all places $u|l_i$ of $F$ 
the restriction ${\barr}_i|_{G_{F_{u}}}$ admits a lift
$\rho_{i,u}:G_{F_{u}} \ra \GL_{n_i}(\CO_{\barQQ_{l_i}})$ which is
ordinary and crystalline with Hodge--Tate numbers $H_{i,\tau}$ for each
$\tau:F_{u} \into \barQQ_l$. (Note that if $F'/F_1$ is finite and
linearly disjoint from $\Favoid F_1$ over $F_1$ and if $F'/F_0$ is
Galois, then $F'/F$ is linearly disjoint from $\Favoid$ over $F$; thus
replacing $F$ by $F_1$ does not affect the condition that $F'$ is linearly disjoint from $\Favoid$ over $F$.) 

In this case, using Proposition \ref{ordlift}, we see that for each $i \in \CI$ there is a lift
\[ r_i:G_{F^+} \lra \CG_{n_i}(\CO_{\barQQ_{l_i}}) \]
of $\tilde{\barr}_{i,\barmu_i}$ such that
\begin{itemize}
\item $\nu \circ r_i = \mu_i$;
\item if $u|l_i$ is a place of $F$ then $\breve{r}_i|_{G_{F_u}}$ is ordinary and crystalline with Hodge--Tate numbers
$H_{i,\tau}$ for each $\tau:F_{u} \into \barQQ_{l_i}$;
\item if $v \in S$ and $v\ndiv l_i$ then $\breve{r}_i|_{G_{F_{v}}} \sim \rho_{i,v}$;
\item $r_i$ is unramified outside $S$.
\end{itemize} 
(Note that if we write $S=\tS \coprod c\tS$ then we only need check the penultimate assertion for $v\in \tS$ and it will follow also for $v \in c\tS$.)

Choose a place $v_q$ of $F$ which is split over $F^+$ and which lies above an odd rational prime $q$, which in turn does not lie under any prime in $S$. 
Also choose integers $b_{i,\tau}$ for all $i\in \CI$ and all $\tau:F \into \barQQ_l$ such that 
\begin{itemize}
\item $b_{i,\tau}+b_{i,\tau \circ c}=2n_i-1-w_i$ for all $\tau$,
\item and $|b_{i,\tau} - b_{i,\tau\circ
    c}|>|h-h'|$ for all $\tau$ and for all $h\in H_{i,\tau}, h'\in H_{i,\tau\circ c}$.
    \end{itemize}
Now choose a character $\psi_i:G_F \ra \barQQ_{l_i}^\times$ for $i \in \CI$ such that 
\begin{itemize}
\item $\psi_i$ is unramified at places in $S$ which do not divide $l_i$;
\item $\psi_i$ is crystalline at all places above $l_i$ and $\HT_\tau(\psi_i)=\{ b_{i,\tau}\}$ for all $\tau:F \into \barQQ_{l_i}$;
\item $q|\#(\psi_i/\psi_i^c)(I_{F_{v_q}})$; and
\item $\psi_i \psi_i^c=\epsilon_{l_i}^{1-2n_i} \mu_i^{-1}|_{G_F}$. 
\end{itemize}
(To do this apply Lemma \ref{l416} with the set $S$ of that Lemma
equal to the union of our $S$ and the set of primes of $F$ above $q$. For $v$ in this set take $\psi_v$ as follows:
\begin{itemize}
\item if $\{ v,v^c\} \subset S$ but $v\ndiv l_i$ then put $\psi_v=1$ and $\psi_{v^c}=\epsilon_l^{1-2n}\mu^{-1}|_{G_{F_{v^c}}}$ (or the other way round),
\item $\psi_{v_q}$ is a wildly ramified character and $\psi_{v_q^c}=(\psi_{v_q}^c)^{-1}$,
\item if $v|l_i$ then $\psi_v$ is crystalline with $\HT_\tau(\psi_v)=\{ b_{i,\tau}\}$ for all continuous $\tau:F_v \into \barQQ_l$.)
\end{itemize}

Consider 
\[ I_{\barpsi_i}(\tilde{\barr}_{i,\barmu_i}):=I(\tilde{\barr}_{i,\barmu_i} \otimes (\barpsi_i, \barepsilon_{l_i}^{1-2n_i} \barmu_i^{-1}\delta_{F/F^+})): G_{F^+} \ra \GSp_{2n_i}(\barFF_{l_i}), \]
which has multiplier $\barepsilon_{l_i}^{1-2n_i}$. As in the proof of Proposition \ref{ordlift} we see that 
$I_{\barpsi_i}(\tilde{\barr}_{i,\barmu_i})(G_{F^+(\zeta_{l_i})})$
is adequate.
Theorem \ref{prop:Dwork} tells us that there is a finite totally
real field extension $F_1^+/F^+$ 
such that 
\begin{itemize}
\item $F_1^+/F_0^+$ is Galois;
\item $F_1^+$ is linearly disjoint from 
$\barF^{\cap_i \ker I_{\barpsi_i}(\tilde{\barr}_{i,\barmu_i})}(\zeta_{\prod_i l_i})\Favoid$
over $F^+$;
\item each $((\Ind_{G_F}^{G_{F^+}} \barr_i \otimes \psi_i)|_{G_{F_1^+}}, \barepsilon_{l_i}^{1-2n_i})$ is ordinarily automorphic of weight $0$.
\end{itemize}
By Theorem \ref{ger} we conclude that each $((\Ind_{G_F}^{G_{F^+}} \breve{r}_i \otimes \psi_i)|_{G_{F_1^+}}, \epsilon_{l_i}^{1-2n_i})$ is ordinarily automorphic of level prime to $l_i$.

Let $F'=FF_1^+$. 
By Lemma \ref{htcmtr}, we see that each
$(( \breve{r}_i \otimes \psi_i)|_{G_{F'}}, \epsilon_{l_i}^{1-n_i})$ is automorphic of level prime to $l_i$. Hence each $(\breve{r}_i|_{G_{F'}}, \mu_i|_{G_{(F')^+}})$ is automorphic of level prime to $l_i$. As these representations are also ordinary, they are ordinarily automorphic of level prime to $l$.
The theorem follows (using local-global compatibility).

\end{proof}
\newpage

\section{The main theorems.}\label{alt2}

          \subsection{A preliminary automorphy lifting result.}\label{prelalt}{$\mbox{}$} \newline

The proof of the next proposition is our main innovation. The last two parts of assumption (\ref{part5}) are rather restrictive and mean that the proposition is not directly terribly useful. However in the next section we will see how we can combine this result with Theorem \ref{ger} to get a genuinely useful result.  Our main tool will be Harris' tensor product trick (see
        \cite{harris:manin} and \cite{blght}).
          
             \begin{prop} \label{propmlt} 
     Let $F$ be an imaginary CM field with maximal totally real subfield $F^+$
     and let $c$ denote the non-trivial element of $\Gal(F/F^+)$. Suppose that
     $l$ is an odd prime and let $n \in \Z_{\geq 1}$. Assume that $\zeta_l\notin
     F$. Let $(r,\mu)$ be a regular algebraic, irreducible, $n$-dimensional, polarized $l$-adic representation of $G_F$. 
Let $\barr$ denote the semi-simplification of the reduction
of $r$, and let $d$ denote the maximal dimension of an irreducible subrepresentation of the restriction of $\barr$ to the closed subgroup of $G_F$ generated by all Sylow pro-$l$-subgroups. 
Suppose that $(r,\mu)$ enjoys the following properties:
   \begin{enumerate}
\item $r|_{G_{F_v}}$ is potentially diagonalizable (and so in particular potentially crystalline) for all $v|l$.
\item  The restriction $ \barr|_{G_{F(\zeta_l)}}$ is irreducible and
  $l\geq 2(d+1)$. \item\label{part5}  $(\barr,\barmu)$ is automorphic of level prime to $l$ arising from a regular algebraic, cuspidal, polarized automorphic representation $(\pi,\chi)$  such that
\begin{itemize}
      \item $r_{l,\imath}(\pi)|_{G_{F_v}}$ is potentially diagonalizable for all $v|l$;
      \item for all $\tau:F \into \barQQ_l$ the set $\{ h+h': \,\, h \in \HT_\tau(r),\,\, h' \in \HT_\tau(r_{l,\imath}(\pi))\}$ has $n^2$ distinct elements;
      \item if $v \ndiv l$ then $r_{l,\imath}(\pi)|_{G_{F_v}} \sim r|_{G_{F_v}}$.
\end{itemize} \end{enumerate}

Then $(r,\mu)$ is potentially diagonalizably automorphic (of level potentially prime to $l$).
\end{prop}

\begin{proof}
Note that 
$\imath \WD (r_{l,\imath}(\pi)|_{G_{F_v}})^{F-\semis} = \rec(\pi_v |\,\,\,|_v^{(1-n)/2})$ for all $v \ndiv l$. Moreover as $\pi_v$ is generic we have $r_{l,\imath}(\pi)|_{G_{F_v}} \leadsto r|_{G_{F_v}}$ for all $v \ndiv l$.

Also note that by Proposition \ref{ghtt} $\barr(G_{F(\zeta_l)})$ is adequate and so (by the remark in the paragraph before the statement of Proposition \ref{ghtt}) we see that $l \ndiv n$. 

Using Lemma \ref{bc} we see that it is enough to prove the theorem after replacing $F$ by a soluble CM extension which is linearly disjoint from $\barF^{\ker \barr}(\zeta_l)$ over $F$. Thus (using Lemma \ref{l412}) we may suppose that 
\begin{itemize}
\item $F/F^+$ is unramified at all finite primes;
\item all primes dividing $l$ and all primes at which $\pi$ or $r$ ramify are split over $F^+$;
\item if $u$ is a place of $F$ above a rational prime which equals $l$ or above which $\pi$ ramifies, then
$\barr|_{G_{F_u}}$ is trivial;
\item if $u$ is a place of $F$ above $l$ then 
$r|_{G_{F_u}}$ and $r_{l,\imath}(\pi)|_{G_{F_u}}$ are diagonalizable,
and $\pi_u$ is unramified.
\end{itemize}
For each prime $v$ of $F^+$ which splits in $F$, choose once and for all a prime $\tv$ of $F$ above $v$.

For $u$ a prime of $F$ above $l$ suppose that
\[ r|_{G_{F_u}} \sim \psi^{(u)}_1 \oplus \dots \oplus \psi^{(u)}_n, \]
and
\[ r_{l,\imath}(\pi)|_{G_{F_u}} \sim \phi^{(u)}_1 \oplus \dots \oplus \phi^{(u)}_n, \]
for crystalline characters $\psi^{(u)}_i$ and $\phi_i^{(u)}:G_{F_{u}} \ra \CO_{\barQQ_l}^\times$. 
We can, and shall, assume that the characters $\psi^{(u)}_i$ and $\phi_i^{(u)}$ satisfy
$\psi_i^{(cu)}(\psi_i^{(u)})^c=\mu|_{G_{F_{cu}}}$
and $\phi_i^{(cu)}(\phi_i^{(u)})^c=(r_{l,\imath}(\chi)\epsilon_l^{1-n})|_{G_{F_{cu}}}$.
For $\tau:F_u \into \barQQ_l$ write $\HT_\tau(\psi_i^{(u)})=\{ h_{\tau,i}' \}$ and  $\HT_\tau(\phi_i^{(u)})=\{ h_{\tau,i}\}$. There are integers $w$ and $w'$ such that for each $\tau:F^+ \into \barQQ_l$ we have
$\HT_\tau(\mu)=\{ w'\}$ and $\HT_\tau(r_{l,\imath}(\chi))=\{ w+1-n\}$. Then
\[ h_{\tau,i}'+h_{\tau c, i}' =w' \]
and
\[ h_{\tau,i}+h_{\tau c, i}=w \]
for all $\tau$ and $i$.

Using Corollary \ref{cycCM} we may choose a CM extension $M/F$ such that
\begin{itemize}
\item $M/F$ is cyclic of  degree $n$;
\item $M$ is linearly disjoint from $\barF^{\ker \barr}(\zeta_l)$ over $F$;
\item and all primes of $F$ above $l$ or at which $\pi$ ramifies split completely in $M$.
\end{itemize}
Choose a prime $u_q$ of $F$ above a rational prime $q$ such that
\begin{itemize}
\item $q \neq l$ and $q$ splits completely in $M$;
\item $r$, $\mu$, $\pi$ and $\chi$ are unramified above $q$.
\end{itemize}
If $v|ql$ is a prime of $F$ we label the primes of $M$ above $v$ as $v_{M,1},\dots,v_{M,n}$ so that $(cv)_{M,i}=c(v_{M,i})$. Choose continuous characters
\[ \theta,\theta':G_M \lra \barQQ_l^\times \]
such that
\begin{itemize}
\item the reductions $\bartheta$ and $\bartheta'$ are equal;
\item $\theta\theta^c=r_{l,\imath}(\chi)\epsilon_l^{1-n}$ and $\theta'(\theta')^c=\mu $;
\item $\theta$ and $\theta'$ are de Rham;
\item if $\tau:M \into \barQQ_l$ lies above a place $v_{M,i}|l$ of $M$ then $\HT_\tau(\theta)=\{h_{\tau|_F,i}\}$ and $\HT_{\tau}(\theta')=\{h_{\tau|_F,i}'\}$;
\item $\theta$ and $\theta'$ are unramified at $u_{q,M,i}$ for $i>1$, but $q$ divides $\#\theta(I_{M_{u_{q,M,1}}})$ and
$\#\theta'(I_{M_{u_{q,M,1}}})$.
\end{itemize}
(First use the first part of Lemma \ref{l416} to choose (say) $\theta$ and then use the second part of Lemma \ref{l416} to choose $\theta'$.)
Note that if $u|l$ is a place of $F$ and if $K/F_u$ is a finite
extension over which $\theta$ and $\theta'$ become crystalline and
$\bartheta=\bartheta'$ become trivial, then 
\[ (\Ind_{G_M}^{G_F} \theta)|_{G_{K}} \sim \phi_1^{(u)}|_{G_K} \oplus \dots \oplus \phi_n^{(u)}|_{G_K} \]
and
\[ (\Ind_{G_M}^{G_F} \theta')|_{G_{K}} \sim \psi_1^{(u)}|_{G_K} \oplus
\dots \oplus \psi_n^{(u)}|_{G_K}. \] (To see this, note that both
sides are residually trivial by the choice of $K$, and both
sides are sums of crystalline characters with the same Hodge--Tate
numbers. The result then follows from points (5) and (6) of Section \ref{l=p}.)

Now let $F_1/F$ be a solvable CM extension such that
\begin{itemize}
\item $\theta|_{G_{F_1M}}$ and $\theta'|_{G_{F_1M}}$ are unramified away from $l$ and crystalline at all primes above $l$;
\item $\thetabar|_{G_{F_1M}}$  is
  trivial at all primes above $l$; 
\item $F_1$ is linearly disjoint
from $\barF^{\ker (\barr \otimes \Ind_{G_M}^{G_F}\bartheta)}M(\zeta_l)$ over $F$;
\item $MF_1/F_1$ is unramified at all finite places.
\end{itemize}
(Use Lemma \ref{l412}.)
Note that $MF_1/F_1$ is split completely above all places of $F$ at
which $\pi$ is ramified.

Put  
\[R:=(r\otimes (\Ind_{G_M}^{G_{F}} \theta))|_{G_{F_1}},\]
$$R':=(r_{l,\imath}(\pi)\otimes (\Ind_{G_M}^{G_{F}} \theta'))|_{G_{F_1}}.$$
Note that we have the following facts.
\begin{itemize}
\item $\barR \cong \barR'$.
\item $R^c \cong (r^\vee \otimes \Ind_{G_M}^{G_F} \theta^c \otimes \mu )|_{G_{F_1}} \cong R^\vee \otimes (\mu r_{l,\imath}(\chi) \epsilon_l^{1-n}\delta_{F_1/F_1^+})|_{G_{F_1}}$.
\item $(R')^c \cong (R')^\vee \otimes(\mu r_{l,\imath}(\chi) \epsilon_l^{1-n}\delta_{F_1/F_1^+})|_{G_{F_1}}$.

\item $\barR$ is irreducible and $\barR(G_{F_1(\zeta_l)})$ is
  adequate.

[As $\barr|_{G_{M(\zeta_l)}}$ is irreducible, we see that the
restriction to $G_{M(\zeta_l)}$ of any
constituent of $(\barr \otimes \Ind_{G_M}^{G_{F}}
\bartheta)|_{G_{F(\zeta_l)}}$ is a sum of $\barr|_{G_{M(\zeta_l)}}
\bartheta^\tau|_{G_{M(\zeta_l)}}$ as $\tau$ runs over some subset of $\Gal(M/F)$. Looking at ramification above $u_q$ we see that the
$\barr|_{G_{M(\zeta_l)}} \bartheta^\tau|_{G_{M(\zeta_l)}}$ are
permuted transitively by $\Gal(M/F)$ and hence
$(\barr \otimes \Ind_{G_M}^{G_{F}} \bartheta
)|_{G_{F(\zeta_l)}}$ is irreducible. Since $F_1$ is linearly disjoint
from $\barF^{\ker (\barr \otimes \Ind_{G_M}^{G_F}\bartheta)}M(\zeta_l)$ over $F$, we see
that $\barR|_{G_{F_1(\zeta_l)}}$ is irreducible. 
As $l \ndiv n$,
every Sylow pro-$l$ subgroup of $G_{F(\zeta_l)}$ is a subgroup
of $G_{M(\zeta_l)}$. By Proposition \ref{ghtt}, we see that $\barR(G_{F_1(\zeta_l)})$ is
adequate.]

\item $(R',  \mu r_{l,\imath}(\chi) \epsilon_l^{1-n}\delta_{F_1/F_1^+})$ is automorphic of level prime to $l$, say 
\[ (R', \mu r_{l,\imath}(\chi) \epsilon_l^{1-n}\delta_{F_1/F_1^+})\cong (r_{l,\imath}(\pi_1), r_{l,\imath}(\chi_1)\epsilon_l^{1-n^2}). \]
Moreover $\pi_1$ only ramifies at places of $F_1$ where $\BC_{F_1/F}(\pi)$ is ramified.
($\pi_1$ is constructed as the automorphic
induction of 
\[ \BC_{F_1M/F}(\pi) \otimes (\phi' |\,\,|^{n(n-1)/2}\circ \det) \]
to $F_1$, where $r_{l,\imath}(\phi')= \theta'|_{G_{F_1M}}$. Note that if $\sigma \in \Gal(F_1M/F_1)$ then $\barr_{l,\imath}(\pi)|_{G_{F_1M}} \bartheta'|_{G_{F_1M}}^\sigma \not\cong \barr_{l,\imath}(\pi)|_{G_{F_1M}} \bartheta'|_{G_{F_1M}}$, so that $(\BC_{F_1M/F}(\pi) \otimes (\phi' \circ \det))^\sigma \not\cong \BC_{F_1M/F}(\pi) \otimes (\phi' \circ \det)$ and $\pi_1$ is cuspidal.)
 \item For all places $u|l$ of $F_1$, 
\[ \begin{array}{rcl} R|_{G_{F_{1,u}}} & \sim & (\psi^{(u|_{F})}_1 \oplus \dots \oplus \psi^{(u|_F)}_n)|_{G_{F_{1,u}}} \otimes (\phi^{(u|_F)}_1 \oplus \dots \oplus \phi^{(u|_F)}_n)|_{G_{F_{1,u}}} \\ & \sim &
      R'|_{G_{F_{1,u}}}. \end{array} \]
 \item For all places $u\nmid l$ of $F_1$ we have $R'|_{G_{F_{1,u}}}\sim R_{G_{F_{1,u}}}$.
[Because we know that $r_{l,\imath}(\BC_{F_1/F}(\pi))|_{G_{F_{1,u}}}\sim r|_{G_{F_{1,u}}}$ and $(\Ind_{G_M}^{G_F}\theta')|_{G_{F_{1,u}}} \sim
(\Ind_{G_M}^{G_F}\theta)|_{G_{F_{1,u}}}$, the latter because they are both unramified.]
\end{itemize}

We now apply Theorem \ref{twkmlt}, with \begin{itemize}
\item $F$ the present $F_1$,
\item $l$ as in the present setting,
\item $n$ the present $n^2$,
\item $r$ the present $R$,
\item $\mu$ the present $ \mu r_{l,\imath}(\chi) \epsilon_l^{1-n}\delta_{F_1/F_1^+}$,
\item $(\pi,\chi)$ the present $(\pi_1,\chi_1)$.
\end{itemize}
We conclude that $(R,(\mu r_{l,\imath}(\chi) \epsilon_l^{1-n}\delta_{F_1/F_1^+})|_{G_{(F_1M)^+}})$ is automorphic of level prime to $l$. By Lemma \ref{htcmtr}, $(r|_{G_{F_1M}}\otimes \theta|_{G_{F_1M}}, (\mu r_{l,\imath}(\chi) \epsilon_l^{1-n}\delta_{F_1/F_1^+})|_{G_{(F_1M)^+}})$ is automorphic of level prime to
$l$. Using Lemma \ref{untwist} we deduce that $(r|_{G_{F_1M}}, \mu|_{G_{(F_1M)^+}} )$ is automorphic of level prime to
$l$. Finally by Lemma \ref{bc}, $r$ is automorphic of
level potentially prime to $l$, and hence potentially diagonalizably automorphic.
\end{proof}

          \subsection{Automorphy lifting: the potentially diagonalizable case.}\label{malt}{$\mbox{}$} \newline
     
        In this section we will prove our main automorphy lifting theorem. It
        generalizes Theorem \ref{ger}
        from the ordinary case to the potentially diagonalizable
        case. It is proved by combining Theorem \ref{ger} and Propositions \ref{ordlift} and \ref{propmlt}. 
               
        \begin{thm} \label{mainmlt} 
     Let $F$ be an imaginary CM field with maximal totally real subfield $F^+$ and let $c$ denote the non-trivial element of $\Gal(F/F^+)$. Suppose that that $l$ is an odd prime, and that $(r,\mu)$ is a regular algebraic, irreducible, $n$-dimensional, polarized representation of $G_F$. 
     Let $\barr$ denote the semi-simplification of the reduction
of $r$, and let $d$ denote the maximal dimension of an irreducible subrepresentation of the restriction of $\barr$ to the closed subgroup of $G_F$ generated by all Sylow pro-$l$-subgroups. 
Suppose that $(r,\mu)$ enjoys the following properties:
   \begin{enumerate}
\item\label{pdiag} $r|_{G_{F_v}}$ is potentially diagonalizable (and so in particular potentially crystalline) for all $v|l$.
\item  The restriction $ \barr|_{G_{F(\zeta_l)}}$ is irreducible, $l \geq 2(d+1)$, and $\zeta_l \not\in F$.
\item\label{auto}  $(\barr,\barmu)$ is either ordinarily automorphic or potentially diagonalizably automorphic.
\end{enumerate}

Then $(r,\mu)$ is potentially diagonalizably automorphic (of level potentially prime to $l$).  
\end{thm}

We remark that condition (\ref{pdiag}) of the theorem will be satisfied
if, in particular, $l$ is unramified in $F$ and $r$ is crystalline at
all primes above $l$,  and $\HT_\tau(r)$ is contained in an interval of the form $[a_\tau,a_\tau+l-2]$ for all $\tau$
(the ``Fontaine--Laffaille'' case).  
We also remark that the reason we can not immediately apply
Proposition \ref{propmlt} to deduce this theorem is the last two parts
of assumption \ref{part5} in Proposition \ref{propmlt} (i.e.\ roughly
speaking $r$ and $r_{l,\imath}(\pi)$ may have different level or $r
\otimes r_{l,\imath}(\pi)$ may have repeated Hodge--Tate weights).
To get round this problem we use Proposition \ref{ordlift} to
create two ordinary intermediate lifts of $\barr$, one $r_1$ with
similar behavior (`level') to $r$, and
one $r_2$ with similar behavior to $r_{l,\imath}(\pi)$. We also ensure that $r_1
\otimes r$ and $r_2 \otimes r_{l,\imath}(\pi)$ are Hodge--Tate regular. Theorem \ref{ger} tells us
that if $r_2$ is automorphic so is $r_1$. On the other hand
Proposition \ref{propmlt} allows us to show that $r_2$ is automorphic
and that if $r_1$ is automorphic then so is $r$.

   \begin{proof}
   Using Lemma \ref{bc} it is enough to prove the theorem after replacing $F$ by a soluble CM extension which is linearly disjoint from $\barF^{\ker \barr}(\zeta_l)$ over $F$. Thus we may suppose that 
\begin{itemize}
\item $F/F^+$ is unramified at all finite primes;
\item all primes dividing $l$ and all primes at which $\pi$ or $r$ ramify are split over $F^+$;
\item if $u$ is a place of $F$ above $l$ then $F_u$ contains a
  primitive $l^{th}$ root of unity, and
$\barr|_{G_{F_u}}$ and $\barr_{l,\imath}(\pi)|_{G_{F_u}}$ are trivial.
\end{itemize}
Let $S$ denote the set of primes of $F^+$ which divide $l$ or above which $r$ or $\pi$ ramifies. 
For each prime $v\in S$ choose once and for all a prime $\tv$ of $F$ above $v$. 

Note that $\mu(c)=-1$ for all complex conjugations $c$ and that
we may extend $\barr \cong \barr_{l,\imath}(\pi)$ to a homomorphism
\[ \tilde{\barr}_{\barmu}:G_{F^+} \lra \CG_n(\barFF_l) \]
with multiplier $\barmu$.

Choose an integer
$m$ strictly greater than $|h-h'|$ for all $h$ and $h'$, Hodge--Tate numbers for $r$ or $r_{l,\imath}(\pi)$. If
$\tau:F \into \C$ set
\[ H_\tau=\{ 0, m, \dots, (n-1)m \}. \]
Note that if $u|l$ then both $\barr|_{G_{F_u}}$ and $\barr_{l,\imath}(\pi)|_{G_{F_u}}$ have ordinary and 
crystalline lifts $1 \oplus \epsilon_l^{-m} \oplus \dots \oplus
\epsilon_l^{(1-n)m}$ with $\tau$-Hodge--Tate numbers $H_{\tau|_F}$ for
each $\tau:F_u \into \barQQ_l$. (It is here that we use the assumption
that $F_u$ contains a
  primitive $l^{th}$ root of unity, and
$\barr|_{G_{F_u}}$ and $\barr_{l,\imath}(\pi)|_{G_{F_u}}$ are
trivial.) Applying Proposition \ref{ordlift} we see that there is a continuous homomorphism
\[ r_1: G_{F^+} \lra \CG_n(\barQQ_l) \]
lifting $\tilde{\barr}_{\barmu}$ and such that
\begin{itemize}
\item $\nu \circ r_1= \epsilon_l^{(1-n)m} \omega_l^{(n-1)m} \tmu$ where $\tmu$ denotes the Teichmuller lift of $\barmu$;
\item if $u|l$ then $\breve{r}_1|_{G_{F_u}}$ is ordinary and crystalline with Hodge--Tate numbers $H_{\tau|_F}$ for each $\tau:F_u \into \barQQ_l$;
\item $r_1$ is unramified outside $S$;
\item and if $v \in S$ and $v \ndiv l$ then $r|_{G_{F_\tv}} \sim \breve{r}_1|_{G_{F_\tv}}$.
\end{itemize}

First we treat the case that $(\barr,\barmu)$ is ordinarily automorphic. In this case Theorem \ref{ger} tells us that $(\breve{r}_1,\epsilon_l^{(1-n)m}\omega_l^{(n-1)m} \tmu)$ is 
automorphic of level prime to $l$. Then Proposition \ref{propmlt} tells us that $(r,\mu)$ is potentially diagonalizably automorphic, and we have completed the proof of the theorem in this case.

Secondly we treat the case that $(\barr,\barmu)$ is potentially diagonalizably automorphic, say $(\barr,\barmu)=(r_{l,\imath}(\pi),r_{l,\imath}(\chi))$. In doing so we are free to make use of the first case, which we have already proved. Again applying Proposition \ref{ordlift} we find 
a continuous homomorphism
\[ r_2: G_{F^+} \lra \CG_n(\barQQ_l) \]
lifting $\tilde{\barr}_{\barmu}$ and such that
\begin{itemize}
\item $\nu \circ r_2= \epsilon_l^{(1-n)m} \omega_l^{(n-1)(m-1)} \tchi$ where $\tchi$ denotes the Teichmuller lift of $\barr_{l,\imath}(\chi)$;
\item if $u|l$ then $\breve{r}_2|_{G_{F_u}}$ is ordinary and crystalline with Hodge--Tate numbers $H_{\tau|_F}$ for each $\tau:F_u \into \barQQ_l$;
\item $r_2$ is unramified outside $S$;
\item and if $v \in S$ and $v \ndiv l$ then $r_{l,\imath}(\pi)|_{G_{F_\tv}} \sim \breve{r}_2|_{G_{F_\tv}}$.
\end{itemize}
By Proposition \ref{propmlt}, $(r_2,\epsilon_l^{(1-n)m} \omega_l^{(n-1)(m-1)} \tchi)$ is automorphic of level
potentially prime to $l$, say  $r_2=r_{l,\imath}(\pi_2)$. As $r_2$ is ordinary and $\pi_2$ has level potentially prime to 
$l$ we can conclude that $\pi_2$ is $\imath$-ordinary. The second case of our theorem now follows from the first case.
\end{proof}

\subsection{Lifting Galois representations II}\label{lgr2} {$\mbox{}$} \newline

We now use the same idea that we used to prove Theorem \ref{mainmlt} to prove a strengthening of Proposition \ref{ordlift}.

Let $n$ be a positive integer and $l$ an odd  prime. Suppose that $F$
is an imaginary CM field not containing $\zeta_l$ and with maximal totally real subfield $F^+$.
Let $S$ be a finite set of finite places of $F^+$ which split in $F$ and suppose that $S$ includes all places above $l$. For each $v \in S$ choose a prime $\tv$ of $F$ above $v$. 

Let $\mu:G_{F^+} \ra \barQQ_l^\times$ be an algebraic character unramified outside $S$ such that $\mu(c_v)=-1$ for all $v|\infty$.

Also let
\[ \barr: G_{F^+} \lra \CG_n(\barFF_l) \]
be a continuous representation unramified outside $S$ with $\nu \circ \barr = \barmu$ and $\barr^{-1}\CG^0_n(\barFF_l)=G_F$. For $v \in S$, let $\rho_v:G_{F_\tv} \ra \GL_n(\CO_{\barQQ_l})$ denote a lift of $\breve{\barr}|_{G_{F_\tv}}$.

\begin{thm}\label{diaglift} Keep the notation and assumptions already stated in this section. Also make the following assumptions:
\begin{itemize}
\item Suppose that $\breve{\barr}|_{G_{F(\zeta_l)}}$ is irreducible. Also, writing $d$ for the maximal dimension of an irreducible subrepresentation of the restriction of $\barr$ to the closed subgroup of $G_F$ generated by all Sylow pro-$l$-subgroups,  suppose that $l \geq 2(d+1)$.

\item If $v|l$ we suppose that $\rho_v$ is potentially diagonalizable and that, for all $\tau:F_\tv\into\Qlbar$, the multiset $\HT_\tau(\rho_v)$ consists of $n$ distinct integers.
\end{itemize}

Then there
is a lift
\[ r:G_{F^+} \lra \CG_n(\CO_{\barQQ_l}) \]
of $\barr$ such that
\begin{enumerate}
\item $\nu \circ r = \mu$;
\item if $v \in S$ then $\breve{r}|_{G_{F_\tv}} \sim \rho_v$;
\item $r$ is unramified outside $S$.
\end{enumerate} \end{thm}

\begin{proof}
  We may suppose that for $v \in S$ with $v \ndiv l$ the
  representation $\rho_v$ is robustly smooth (see Lemma \ref{gen}) and
  hence lies on a unique component $\CC_v$ of
  $R^\Box_{\breve{\barr}|_{G_{F_\tv}}} \otimes \barQQ_l$.  If $v|l$ is
  a place of $F^+$ then choose a finite extension $K_v/F_\tv$ over
  which $\rho_v$ becomes crystalline, and let $\CC_v$ denote the
  unique component of $R^\Box_{\breve{\barr}|_{G_{F_\tv}}, \{ \HT_\tau(\rho_v)
    \}, K_v-\cris} \otimes \barQQ_l$ on which $\rho_v$ lies.  
Let $\tmu$ denote the Teichmuller lift of $\barmu$. Choose a
  positive integer $m$ which is greater than one plus the difference
  of any two Hodge--Tate numbers of $\rho_v$ for every $v|l$.
 
Choose (using Lemma \ref{l412}) a finite, soluble, Galois, CM extension $F_1/F$ which is linearly disjoint from $\barF^{\ker \barr}(\zeta_l)$ over $F$ such that
\begin{itemize}
\item for all $u$ lying above $S$ we have $\barr(G_{F_{1,u}})=\{ 1\}$;
\item for all $u|l$ we have $\zeta_l \in F_{1,u}$;
\item $\mu|_{G_{F_1^+}}$ is crystalline above $l$;
\item if $u|\tv|l$ with $v \in S$ then $\rho_v|_{G_{F_{1,u}}}$ is crystalline and $\rho_v|_{G_{F_{1,u}}} \sim \psi_{1}^{(u)} \oplus \dots \oplus \psi_{n}^{(u)}$ with each $\psi_{i}^{(u)}$ a crystalline character.
\end{itemize}
If $u|\tv|l$ with $v \in S$, then for $i=1,\ldots,n$, we define
$\psi_{i}^{(cu)} : G_{F_{1,cu}}\ra \Qlbar^\times$ by $(\psi_{i}^{(cu)})^c\psi_{i}^{(u)}=\mu|_{G_{F_1,u}}$.

By Proposition \ref{paord} there is a finite, Galois, CM extension $F_2/F_1$ linearly disjoint from  $F_1\barF^{\ker \barr}(\zeta_l)$ over $F_1$ and a regular algebraic, cuspidal, polarized automorphic representation $(\pi_2, \chi_2)$ of $\GL_n(\A_{F_2})$ such that
\begin{itemize}
\item $\barr_{l,\imath}(\pi_2) \cong \breve{\barr}|_{G_{F_2}}$;
\item $r_{l,\imath}(\chi_2)=\tmu|_{G_{F_2^+}} \omega_l^{(n-1)m}\epsilon_l^{(1-n)(m-1)}$;
\item $\pi_2$ is $\imath$-ordinary and unramified above $l$;
\item if $\tau:F_2 \into \barQQ_l$, then $\HT_\tau(r_{l,\imath}(\pi_2))=\{0,m,2m,\dots,(n-1)m\}$;
\item $\pi_2$ is unramified outside $S$;
\item and if $v \ndiv l$ is in $S$ and if $u$ is  a prime of $F_2$ above  $\tv$ then $r_{l,\imath}(\pi_2)|_{G_{F_{2,u}}} \sim \rho_v|_{G_{F_{2,u}}}$.
\end{itemize}
In particular if $u|l$ is a place of $F_2$ then
\[ r_{l,\imath}(\pi_2)|_{G_{F_{2,u}}} \sim 1 \oplus \epsilon_l^{-m} \oplus \dots \oplus \epsilon_l^{(1-n)m}. \]

Choose (using Corollary \ref{cycCM}) a CM extension $M/F_2$ such that
\begin{itemize}
\item $M/F_2$ is cyclic of  degree $n$;
\item $M$ is linearly disjoint from $\barF^{\ker \barr}(\zeta_l)$ over $F$;
\item and all primes of $F_2$ above $l$ split completely in $M$.
\end{itemize}
Choose a prime $u_q$ of $F_2$ above a rational prime $q$ such that
\begin{itemize}
\item $q \neq l$ and $q$ splits completely in $M$;
\item $\barr$ is unramified above $q$.
\end{itemize}
If $v|ql$ is a prime of $F_2$ we label the primes of $M$ above $v$ as $v_{M,1},\dots,v_{M,n}$ so that $(cv)_{M,i}=c(v_{M,i})$. Choose continuous characters
\[ \theta,\theta':G_M \lra \barQQ_l^\times \]
such that
\begin{itemize}
\item the reductions $\bartheta$ and $\bartheta'$ are equal;
\item $\theta\theta^c=r_{l,\imath}(\chi_2)\epsilon_l^{1-n}$ and $\theta'(\theta')^c=\mu$;
\item $\theta$ and $\theta'$ are de Rham;
\item if $\tau:M \into \barQQ_l$ lies above a place $v_{M,i}|l$ of $M$ then $\HT_\tau(\theta)=\{(i-1)m\}$ and $\HT_{\tau}(\theta')=\HT_{\tau|_{F_1}}(\psi_{i}^{(v_{M,i}|_{F_1})})$;
\item $\theta$ and $\theta'$ are unramified at $u_{q,M,i}$ for $i>1$, but $q$ divides $\#\theta(I_{M_{u_{q,M,1}}})$ and
$\#\theta'(I_{M_{u_{q,M,1}}})$.
\end{itemize}
(First choose (say) $\theta$ using the first part of Lemma \ref{l416}, then choose $\theta'$ using the second part of Lemma \ref{l416}.)

Note the following:
\begin{itemize}
\item If $u|l$ is a place of $F_2$ and if $K/F_{2,u}$ is a finite extension over which $\theta$ and $\theta'$ become crystalline and $\bartheta=\bartheta'$ become trivial, then 
\[ (\Ind_{G_M}^{G_{F_2}} \theta)|_{G_{K}} \sim 1 \oplus \epsilon_l^{-m} \oplus \dots\oplus \epsilon_l^{(1-n)m} \]
and
\[ (\Ind_{G_M}^{G_{F_2}} \theta')|_{G_{K}} \sim \psi_{1}^{(u|_{F_1})}|_{G_K} \oplus \dots \oplus \psi_{n}^{(u|_{F_1})}|_{G_K}. \]
[To see this, note that both
sides are residually trivial by the choice of $K$, and both
sides are sums of crystalline characters with the same Hodge--Tate
numbers. The result then follows from points (5) and (6) of Section \ref{l=p}.]

\item $(\Ind_{G_M}^{G_{F_2}} \theta)^c \cong (\Ind_{G_M}^{G_{F_2}} \theta)^\vee \otimes r_{l,\imath}(\chi_2) \epsilon_l^{1-n}$ and $(\Ind_{G_M}^{G_{F_2}} \theta')^c \cong (\Ind_{G_M}^{G_{F_2}} \theta')^\vee \otimes \mu$.

\item The representation 
\[ \breve{\barr}|_{G_{F_2(\zeta_l)}} \otimes (\Ind_{G_M}^{G_{F_2}} \bartheta )|_{G_{F_2(\zeta_l)}} \cong \barr_{l,\imath}(\pi_2)|_{G_{F_2(\zeta_l)}} \otimes (\Ind_{G_M}^{G_{F_2}} \bartheta' )|_{G_{F_2(\zeta_l)}} \]
is irreducible, and hence by Proposition \ref{ghtt}
\[  (\breve{\barr}|_{G_{F_2}} \otimes (\Ind_{G_M}^{G_{F_2}} \bartheta ))(G_{F_2(\zeta_l)}) \]
is adequate.

 [As $\breve{\barr}|_{G_{M(\zeta_l)}}$ is irreducible, we see that the
  restriction to $G_{M(\zeta_l)}$ of any
 constituent of $(\breve{\barr}|_{G_{F_2}} \otimes
 \Ind_{G_M}^{G_{F_2}} \bartheta)|_{G_{F_2(\zeta_l)}}$ is a sum of
 $\breve{\barr}|_{G_{M(\zeta_l)}} \bartheta^\tau|_{G_{M(\zeta_l)}}$ as
 $\tau$ runs over some subset of $\Gal(M/F_2)$. Looking at ramification above $u_q$ we see that the
$\breve{\barr}|_{G_{M(\zeta_l)}} \bartheta^\tau|_{G_{M(\zeta_l)}}$ are permuted transitively by $\Gal(M/F_2)$ and hence $(\breve{\barr}|_{G_{F_2}} \otimes \Ind_{G_M}^{G_{F_2}} \bartheta) |_{G_{F_2(\zeta_l)}}$ is irreducible.]

\end{itemize}

Let $F_3/F_2$ be a finite, soluble, Galois, CM extension linearly disjoint from $\barF_2^{\ker \Ind_{G_M}^{G_{F_2}} \bartheta} \barF^{\ker \barr}(\zeta_l)$ over $F_2$ such that
\begin{itemize}
\item $\theta|_{G_{F_3M}}$ and $\theta'|_{G_{F_3M}}$ are crystalline above $l$ and unramified away from $l$;
\item $MF_3/F_3$ is unramified everywhere.\end{itemize}
(Use Lemma \ref{l412}.)

Then there is a regular algebraic, cuspidal, polarized automorphic representation $(\pi_3,\chi_3)$ of $\GL_{n^2}(\A_{F_3})$ such
that
\begin{itemize}
\item $r_{l,\imath}(\pi_3) \cong (r_{l,\imath}(\pi_2) \otimes \Ind_{G_{M}}^{G_{F_2}} \theta')|_{G_{F_3}}$;
\item $r_{l,\imath}(\chi_3)=\mu r_{l,\imath}(\chi_2) \epsilon_l^{n(n-1)}\delta_{F_3/F_3^+}$;
\item $\pi_3$ is unramified above $l$ and outside $S$.
\end{itemize}
[The representation $\pi_3$ is the automorphic induction of
$(\pi_2)_{MF_3}\otimes (\phi' |\,\,|^{n(n-1)/2}\circ \det)$
to $F_3$, where $r_{l,\imath}(\phi')= \theta'|_{G_{F_3M}}$. The first
two properties are clear.
The third property follows by the choice of $F_3$ and the fact that
$\pi_2$ is unramified above $l$ and outside $S$.]

Let $\tS$ denote the set of $\tv$ as $v$ runs over $S$, let $S_3$
denote the primes of $F_3^+$ above $S$ and $\tS_3$ the primes of $F_3$
above $\tS$. If $v\in S_3$, let $\tv$ denote the element of $\tS_3$
lying above it. For $v \in S_3$ with $v \ndiv l$ (resp.\ $v|l$) let $\CC_{3,v}$ denote the unique component of $R^\Box_{\barr_{l,\imath}(\pi_3)|_{G_{F_{3,\tv}}}} \otimes \barQQ_l$ (resp.\ $R^\Box_{\barr_{l,\imath}(\pi_3)|_{G_{F_{3,\tv}}}, \{\HT_\tau(r_{l,\imath}(\pi_3)|_{G_{F_{3,\tv}}})\}, \cris} \otimes \barQQ_l$) containing $r_{l,\imath}(\pi_3)|_{G_{F_{3,\tv}}}$. Choose a finite extension $L/\Q_l$ in $\barQQ_l$ such that
\begin{itemize}
\item $L$ contains the image of each embedding $F_3 \into \barQQ_l$;
\item $L$ contains the image of $\mu$ and of $\theta$;
\item $r_{l,\imath}(\pi_3)$ is defined over $L$;
\item each of the components $\CC_v$ for $v \in S$ and $\CC_{3,v}$ for $v \in S_3$ is defined over $L$.
\end{itemize}
Set
\[ s = \Ind_{G_{M^+},G_{M}}^{G_{F_2^+},G_{F_2}, r_{l,\imath}(\chi_2)\epsilon_l^{1-n}} (\theta, r_{l,\imath}(\chi_2) \epsilon_l^{1-n}): G_{F_2^+} \lra \CG_n(\CO_L) \]
in the notation of section 1.1 of this paper and section 2.1 of \cite{cht}. Thus $\nu \circ s = r_{l,\imath}(\chi_2)\epsilon_l^{1-n}$.
For $v \in S$ (resp.\ $v \in S_3$) let $\calD_v$ (resp.\ $\calD_{3,v}$) denote the deformation problem for $\barr|_{G_{F_\tv}}$ (resp.\ $\barr_{l,\imath}(\pi_3)|_{G_{F_{3,\tv}}}$) over $\CO_L$ 
corresponding to $\CC_v$ (resp.\ $\CC_{3,v}$). Also let 
\[ \CS=(F/F^+,S,\tS, \CO_L,\barr, \mu, \{ \calD_v\}) \]
and
\[ \CS_3=(F_3/F_3^+,S_3,\tS_3, \CO_L, \barr_{l,\imath}(\pi_3), \mu r_{l,\imath}(\chi_2) \epsilon_l^{1-n} \delta_{F_3^+/F_3}, \{ \calD_{3,v}\}). \]

We next check that if $u \in S_3$ 
then $\breve{r}_\CS^\univ|_{G_{F_{3,\tu}}} \otimes
(\Ind_{G_M}^{G_{F_2}} \theta)|_{G_{F_{3,\tu}}} \in \calD_{3,u}$. To this end, let
$v=u|_{F^+}$ and let $\rho^\Box_{v,\CC_v}$ denote the universal lift
of $\barr|_{G_{F_\tv}}$ to $R_{\CO_L,\barr|_{G_{F_{\tv}}},\CC_v}^\Box$.
It suffices to show that $\rho^\Box_{v,\CC_v}|_{G_{F_{3,\tu}}} \otimes
(\Ind_{G_M}^{G_{F_2}} \theta)|_{G_{F_{3,\tu}}} \in \calD_{3,u}$. For
this, it suffices to show if $\rho : G_{F_\tv} \ra
\GL_n(\CO_{\Qlbar})$ is a lift of $\barr|_{G_{F_\tv}}$ lying on
$\CC_v$, then $\rho|_{G_{F_{3,\tu}}} \otimes (\Ind_{G_M}^{G_{F_2}}
\theta)|_{G_{F_{3,\tu}}}$ lies on $\CC_{3,u}$. 
If $u|l$, then we have $\rho|_{G_{F_{3,\tu}}} \sim
(\Ind_{G_M}^{G_{F_2}}\theta')|_{G_{F_{3,\tu}}}$ and $ (\Ind_{G_M}^{G_{F_2}}
\theta)|_{G_{F_{3,\tu}}} \sim r_{l,\imath}(\pi_2)|_{G_{F_{3,\tu}}}$   and hence
\[  \rho|_{G_{F_{3,\tu}}} \otimes (\Ind_{G_M}^{G_{F_2}}
\theta)|_{G_{F_{3,\tu}}} \sim (r_{l,\imath}(\pi_2)\otimes \Ind_{G_M}^{G_{F_2}}\theta')|_{G_{F_{3,\tu}}}\cong
r_{l,\imath}(\pi_3)|_{G_{F_{3,\tu}}}.\]
If $u\ndiv l$, note that $\rho|_{G_{F_{3,\tu}}} \sim
r_{l,\imath}(\pi_2)|_{G_{F_{3,\tu}}}$ (since $\rho_v$ is
robustly smooth, we have
$\rho_v|_{G_{F_{3,\tu}}} \leadsto \rho|_{G_{F_{3,\tu}}}$ and
$\rho_v|_{G_{F_{3,\tu}}} \leadsto r_{l,\imath}(\pi_2)|_{G_{F_{3,\tu}}}$).
By the choice of $F_3$ we have
$(\Ind_{G_M}^{G_{F_2}}\theta)|_{G_{F_{3,\tu}}}\sim
(\Ind_{G_M}^{G_{F_2}}\theta')|_{G_{F_{3,\tu}}}$. Hence
\[  \rho|_{G_{F_{3,\tu}}} \otimes (\Ind_{G_M}^{G_{F_2}}
\theta)|_{G_{F_{3,\tu}}} \sim (r_{l,\imath}(\pi_2)\otimes \Ind_{G_M}^{G_{F_2}}\theta')|_{G_{F_{3,\tu}}}\cong
r_{l,\imath}(\pi_3)|_{G_{F_{3,\tu}}}\]
and we are done.

We deduce that there is a natural map
\[ R_{\CS_3}^\univ \lra R_\CS^\univ \]
induced by $r_\CS^\univ|_{G_{F_3^+}} \otimes s|_{G_{F_3^+}}$. It follows from
Lemma \ref{cdr} that this map makes $R_\CS^\univ$ a finite $R_{\CS_3}^\univ$-module. By Theorem \ref{fdmin} $R_{\CS_3}^\univ$ is a finite $\CO_L$-module, and hence $R_\CS^\univ$ is a finite $\CO_L$-module. On the other hand by Proposition
\ref{drb} $R_\CS^\univ$ has Krull dimension at least $1$. Hence $\Spec R_\CS^\univ$ has a $\barQQ_l$-point. This point gives rise to the desired lifting of $\barr$.
\end{proof}

\subsection{Change of weight and level.} \label{sec:cwl} {$\mbox{}$} \newline

In this section we combine Theorems \ref{mainmlt} and \ref{diaglift} to obtain a general ``change of weight and level'' result in the potentially diagonalizable case.

\begin{thm}\label{cwl}
  Let $F$ be an imaginary CM field with maximal totally real subfield
  $F^+$. Let $n\in\Z_{\ge 1}$ be an integer, and let $l>2(n+1)$ be an odd
  prime, such that $\zeta_l \not\in F$ and all primes of $F^+$ above
  $l$ split in $F$. Let $S$ be a finite set of finite places of $F^+$,
  including all places above $l$, such that each place in $S$ splits
  in $F$. For each place $v\in S$ choose a place $\tv$ of $F$ lying
  over $v$.
  Let $\mu$ be an algebraic character of $G_{F^+}$ and let $\barr:G_F \ra \GL_n(\barFF_l)$ be a continuous representation such that 
  \begin{itemize}
  \item $(\barr,\barmu)$ is a polarized mod $l$ representation unramified outside $S$, which either we suppose is ordinarily automorphic or we suppose is potentially diagonalizably automorphic;
  \item $\barr|_{G_{F(\zeta_l)}}$ is irreducible.
  \end{itemize}

For $v\in S$ let $\rho_v:G_{F_\tv}\to \GL_n(\CO_{\Qlbar})$ be a lift of
$\rbar_{l,\imath}(\pi)|_{G_{F_\tv}}$. If $v|l$, assume further that $\rho_v$ is
potentially diagonalizable, and that for all $\tau:F_\tv\into\Qlbar$,
$\HT_\tau(\rho_v)$ consists of $n$ distinct integers.

 Then there is a regular algebraic, cuspidal, polarized automorphic representation $(\pi,\chi)$ of $\GL_n(\A_F)$ such that
\begin{enumerate}
\item $\barr_{l,\imath}(\pi)\cong \barr$;
\item $r_{l,\imath}(\chi)\epsilon_l^{1-n}=\mu$;
\item $\pi$ has level potentially prime to $l$;
\item $\pi$ is unramified outside $S$;
\item for $v \in S$ we have $\rho_v \sim r_{l,\imath}(\pi)|_{G_{F_\tv}}$.
\end{enumerate}
\end{thm}

\begin{proof} By Theorem \ref{diaglift} there is a  continuous homomorphism
 \[r:G_{F^+}\to\Gn(\CO_{\Qlbar})\] 
 such that 
\begin{itemize}
\item $\breve{\barr} \cong \barr$;
\item $r$ is unramified outside $S$;
\item $\nu\circ r=\mu$;
\item if $v|l$ then $\breve{r}|_{G_{F_\tv}}$ is potentially diagonalizable;
\item if $v\in S$ then $\breve{r}|_{G_{F_\tv}}\sim\rho_v$.
\end{itemize}
By Theorem \ref{mainmlt} $(\breve{r},\mu)$ is automorphic of level potentially prime
to $l$ and our present theorem follows (using local-global
compatibility to establish point (4)).
\end{proof}

\subsection{Potential automorphy II}\label{pa1} {$\mbox{}$} \newline

We can now turn to our main potential automorphy theorem for (finite
collections of) single
$l$-adic representations. We will treat the case of compatible systems
in the next section.

\begin{thm}\label{mtpm} Suppose that we are in the following situation. 
\begin{enumerate}
\renewcommand{\labelenumi}{(\alph{enumi})}
\item Let $F/F_0$ be a finite, Galois extension of imaginary CM fields; and let $F^+$ and $F_0^+$ denote their maximal totally real subfields.
\item Let $\CI$ be a finite set.
\item For each $i \in \CI$ let $n_i$ and $d_i$ be positive integers and $l_i$ be an odd rational prime such that $l_i \geq 2(d_i+1)$ and $\zeta_{l_i} \not\in F$.
Also choose $\imath_i:\barQQ_{l_i} \iso \C$ for each $i \in \CI$.
\item For each $i\in \CI$ let $(r_i,\mu_i)$ be a {\bf totally odd, regular algebraic,} $n_i$-dimensional, {\bf polarized} $l_i$-adic representation of $G_F$ such that
$d_i$ is the maximum dimension of an irreducible constituent of the restriction of $\barr_i$ to the closed subgroup of $G_F$ generated by all Sylow pro-$l_i$-subgroups.
\item Let $\Favoid/F$ be a finite Galois extension.
\end{enumerate}

Suppose moreover that the following conditions are satisfied for every $i \in \CI$.
\begin{enumerate}
\item {\bf (Potential diagonalizability at primes above $l_i$)}
  $r_i$ is potentially diagonalizable (and hence potentially crystalline) at each prime $v$ of $F^+$ above $l_i$.
  \item {\bf (Irreducibility)} $\bar{r}_i|_{G_{F(\zeta_{l_i})}}$ is irreducible.
  \end{enumerate}

Then we can find a finite CM extension $F'/F$ and for each $i \in \CI$ a regular algebraic, cuspidal, polarized automorphic representation $(\pi_i,\chi_i)$ of $\GL_{n_i}(\A_{F'})$ such that
\begin{enumerate}
\renewcommand{\labelenumi}{(\roman{enumi})}
\item $F'/F_0$ is Galois,
\item $F'$ is linearly disjoint from $\Favoid$ over $F$,
\item $\pi_i$ is unramified above $l_i$, and
\item $(r_{l_i,\imath_i}(\pi_i),r_{l_i,\imath_i}(\chi_i)\epsilon_{l_i}^{1-n_i}) \cong (r_i|_{G_{F'}}, \mu_i|_{G_{(F')^+}})$.
\end{enumerate}\end{thm}

We remark that by Lemma \ref{fllift} the hypothesis of potential diagonalizability will hold if $l_i$ is unramified in $F^+$, and $r_i$
is crystalline at all primes $v|l_i$, and $\HT_\tau(r_i)$ is contained in an interval of the form $[a_\tau,a_\tau+l-2]$ for all $\tau$
(the ``Fontaine--Laffaille'' case).

\begin{proof} By Proposition \ref{paord} there
is a finite CM extension $F'/F$ and regular algebraic, cuspidal, polarized automorphic representations $(\pi_i',\chi_i')$ of $\GL_{n_i}(\A_{F'})$
such that 
\begin{itemize}
\item $F'/F_0$ is Galois;
\item $F'$ is linearly disjoint from $\barF^{\bigcap_i \ker \barr_i} \Favoid (\zeta_{\prod_i l_i})$ over $F$;
\item $\barr_{l_i,\imath_i}(\pi_i')\cong \barr_i|_{G_{F'}}$;
\item $r_{l_i,\imath_i}(\chi_i')\epsilon_{l_i}^{1-n_i} = \mu_i|_{G_{(F')^+}}$;
\item $\pi_i'$ is unramified above $l$;
\item $\pi_i'$ is $\imath_i$-ordinary.
\end{itemize}
Then the current theorem follows from Theorem \ref{mainmlt}. 
\end{proof}

We can immediately deduce a version over totally real
fields. For instance we have the following.
\begin{cor} \label{mtpmtotreal}Suppose $F^+$ is a totally real field and $n\in \Z_{\geq 1}$.  Suppose that $l\geq 2(n+1)$ is a rational prime.

Suppose also that $(r,\mu)$ is a {\bf totally odd, regular algebraic}, $n$-dimensional, {\bf polarized} $l$-adic representation of $G_{F+}$.
Let $\barr$ denote the semi-simplification of the reduction of $r$ and
suppose that the following conditions hold:
\begin{enumerate}
\item {\bf (Potential diagonalizability and regularity at primes above $l$)}
  $r$ is potentially diagonalizable (and hence potentially crystalline) at each prime $v$ of $F^+$ above $l$.
  \item {\bf (Irreducibility)}  $\bar{r}|_{G_{F^+(\zeta_l)}}$ is irreducible.
\end{enumerate}
Then there is a Galois totally real extension $F^{+,\prime}/F^+$ 
such 
that $(r|_{G_{F^{+,\prime}}},\mu|_{G_{F^{+,\prime}}})$ is automorphic of level
prime to $l$.
\end{cor}
\begin{proof}
  Choose $F/F^+$ a totally imaginary quadratic extension in which all the
  places lying over $l$ split completely, and which is linearly disjoint
  from $(\barF^+)^{\ker\ad\rbar}(\zeta_l)$ over $F^+$. The representation $r|_{G_F}$
  satisfies the hypotheses of Theorem \ref{mtpm}, so that there is a
  finite Galois CM extension $F'/F$ such that
  $(r|_{G_{F'}},\mu|_{G_{F^{\prime,+}}})$ is automorphic of level
  prime to $l$. By Lemma \ref{bc},
  $(r|_{G_{F^{\prime,+}}},\mu|_{G_{F^{\prime,+}}})$ is also
  automorphic of level prime to $l$, as required.
 \end{proof}
 
  Combining Theorem \ref{mtpm} with Theorem \ref{diaglift} we get a potential automorphy theorem for mod $l$ Galois representations which strengthens Proposition \ref{paord}.
  
  \begin{cor}\label{padiag} Suppose that we are in the following situation. 
\begin{enumerate}
\renewcommand{\labelenumi}{(\alph{enumi})}
\item Let $F/F_0$ be a finite, Galois extension of imaginary CM fields, and let $F^+$ and $F_0^+$ denote their maximal totally real subfields. Choose a complex conjugation $c \in G_{F^+}$.
\item Let $\CI$ be a finite set.
\item For each $i \in \CI$ let $n_i$ and $d_i$ be positive 
  integers and $l_i\geq 2(d_i+1)$ be an odd rational prime such that $\zeta_{l_i} \not\in F$.
Also choose $\imath_i:\barQQ_{l_i} \iso \C$ for each $i \in \CI$.
\item For each $i\in \CI$ let $\mu_i:G_{F^+} \ra \barQQ_{l_i}^\times$ be a continuous, totally odd, de Rham character.
\item For each $i \in \CI$ let $\barr_i: G_{F} \ra
  \GL_{n_i}(\barFF_{l_i})$ be an irreducible continuous representation so that $d_i$ is the maximal dimension of an irreducible constituent of the restriction of $\barr_i$ to the closed subgroup of $G_F$ generated by all Sylow pro-$l_i$-subgroups. Suppose also 
  that $(\barr_i,\barmu_i)$ is a polarized mod $l_i$ representation and that $\barr_i|_{G_{F(\zeta_{l_i})}}$ is irreducible.
\item Let $S$ denote a finite $\Gal(F/F^+)$-invariant set of primes
   of $F$ including all those dividing $\prod_i l_i$ and
 all those at which some $\barr_i$ ramifies.
 \item for each $i \in \CI$ and $v \in S$ let  $\rho_{i,v}:G_{F_v} \ra \GL_{n_i}(\CO_{\barQQ_{l_i}})$ denote a lift of ${\barr}_i|_{G_{F_v}}$ such that $\rho_{i,cv}^c \cong \mu_i|_{G_{F_v}} \rho_{i,v}^\vee$. If $v|l_i$ further assume that $\rho_{i,v}$ is potentially diagonalizable and that for each $\tau:F_v \into \barQQ_{l_i}$
 the set $\HT_\tau(\rho_{v,i})$ has $n$ distinct elements.
\item Let $\Favoid/F$ be a finite Galois extension.
\end{enumerate}

Then we can find a finite CM extension $F'/F$ and for each $i \in \CI$ a regular algebraic, cuspidal, polarized automorphic representation $(\pi_i,\chi_i)$ of $\GL_{n_i}(\A_{F'})$ such that
\begin{enumerate}
\item $F'/F_0$ is Galois,
\item $F'$ is linearly disjoint from $\Favoid$ over $F$,
\item $\barr_{l_i,\imath_i}(\pi_i)\cong {\barr}_i|_{G_{F'}}$;
\item $r_{l_i,\imath_i}(\chi_i) \epsilon_{l_i}^{1-n_i} = \mu_i|_{G_{(F')^+}}$;
\item $\pi_i$ has level potentially prime to $l_i$;
\item if $u$ is a prime of $F'$ not lying above a prime in $S$ then $\pi_{i,u}$ is unramified;
\item if $u$ is a prime of $F'$ lying above an element $v \in S$ then $r_{l_i,\imath_i}(\pi_i)|_{G_{F_u'}} \sim \rho_v|_{G_{F_u'}}$.
\end{enumerate}
  \end{cor}
  \begin{proof}
Note that $(\barr_i, \barmu_i)$ corresponds to a continuous
homomorphism $\tilde{\barr}_{i,\barmu_i}:G_{F^+} \ra \CG_{n_i}(\barFF_{l_i})$   with $\nu
\circ \tilde{\barr}_i = \barmu_i$ (see Section \ref{cgn}). As in the
proof of Proposition \ref{paord}, we may reduce to the case where all
elements of $S$ are split over $F^+$. Then by Theorem \ref{diaglift},
we see that for
each $i \in \cI$ there exists a lift 
\[ r_i : G_{F} \to \GL_{n_i}(\CO_{\Qlbar}) \]
of $\barr_i$ such that
\begin{itemize}
\item $r_i^c \cong r_i^\vee \mu_i|_{G_F}$;
\item if $v\in S$, then $r_i|_{G_{F_v}}\sim \rho_{i,v}$;
\item $r_i$ is unramified outside $S$.
\end{itemize}
The result now follows from Theorem \ref{mtpm}.
  \end{proof}

 \newpage

\section{Compatible systems.}\label{sec:main result}

\subsection{Compatible systems: definitions.}\label{cs}{$\mbox{}$} \newline

Let $F$ denote a number field. 
By a {\em rank $n$ weakly compatible system of $l$-adic representations} $\CR$
{\em of} $G_F$ {\em defined over} $M$ we shall mean a $5$-tuple 
\[ (M,S,\{ Q_v(X) \}, \{r_\lambda \}, \{H_\tau\} ) \]
where
\begin{enumerate}
\item $M$ is a number field;
\item $S$ is a finite set of primes of $F$;
\item for each  prime $v\not\in S$ of $F$, $Q_v(X)$ is a monic degree $n$
polynomial in $M[X]$;
\item for each prime $\lambda$ of $M$ (with residue characteristic $l$ say) 
\[r_\lambda:G_F \lra \GL_n(\barM_\lambda) \]
is a continuous, semi-simple, representation such that 
\begin{itemize}
\item if $v \not\in S$ and $v \ndiv l$ is a prime of $F$ then $r_\lambda$
is unramified at $v$ and $r_\lambda(\Frob_v)$ has characteristic
polynomial $Q_v(X)$,
\item while if $v|l$ then $r_\lambda|_{G_{F_v}}$ is de Rham  and in the case $v \not\in S$ crystalline;
\end{itemize}
\item for $\tau:F \into \barM$, $H_\tau$ is a multiset of $n$ integers such that 
for any $\barM \into \barM_\lambda$ over $M$ we have 
$\HT_\tau(r_\lambda)=H_\tau$.
\end{enumerate}
We will refer to a rank $1$ weakly compatible system of representations as a weakly compatible system of characters.

We make the following subsidiary definitions:

We define the usual linear algebra and group theoretic operations on weakly compatible systems 
by applying the corresponding operation to each $r_\lambda$. For instance
\[ \CR^\vee = (M,S, \{ X^nQ_v(0)^{-1}Q_v(X^{-1}) \}, \{ r_\lambda^\vee\}, \{ -H_\tau \}), \]
where $-H_\tau=\{ -h: h \in H_\tau\}$.

We will call $\CR$ {\em regular} if for each $\tau:F \into \barM$ every element of $H_\tau$ has multiplicity $1$. 

 We will call $\CR$ {\em extremely regular} if it is regular and  for some $\tau:F \into \barM$ the multiset $H_\tau$ has the following property: if $H$ and $H'$ are subsets of $H_\tau$ of the same cardinality and if $\sum_{h \in H} h = \sum_{h \in H'} h$ then $H=H'$.

If $F$ is totally real and if $n=1$ then we will call $\CR$ {\em
  totally odd} (resp.\ {\em totally even}) if for some place $\lambda$ of $M$ we have $r_\lambda(c_v)=-1$ (resp.\ $1$) for all infinite places $v$ of $F$. In this case this will also be true for all places $\lambda$ of $M$.

If $F$ is CM and if $\CM=(M,S_{F^+},\{X-\alpha_v\}, \{ \mu_\lambda \}, \{ w \})$ is a weakly compatible system of characters of $G_{F^+}$ then we will call $(\CR,\CM)$ a {\em polarized} (resp.\ {\em totally odd, polarized}) weakly compatible system if for all primes $\lambda$ of $M$ the pair $(r_\lambda,\mu_\lambda)$ is a polarized (resp.\ totally odd polarized) $l$-adic representation. (Here $S_{F^+}$ denotes the set of places of $F^+$ lying below an element of $S$.) We will call $\CR$ {\em polarizable} (resp.\ {\em totally odd, polarizable}) if there exists a $\CM$ such that $(\CR,\CM)$ is a polarized (resp.\ totally odd polarized) weakly compatible system.

We will call $\CR$ {\em irreducible} if there is a set $\CL$ of rational primes of Dirichlet density $1$ such that for $\lambda|l \in \CL$ the representation $r_\lambda$ is irreducible. 
We will call $\CR$ {\em strictly compatible}  if for each finite place $v$ of $F$ there is a Weil--Deligne representation $\WD_v(\CR)$ of $W_{F_v}$ over $\barM$ such that for each place $\lambda$ of $M$ not dividing the residue characteristic of $v$ and every $M$-linear embedding $\varsigma:\barM \into \barM_\lambda$ the push forward $\varsigma\WD_v(\CR)\cong \WD(r_\lambda|_{G_{F_v}})^\Fsemis$. 

We will call $\CR$ {\em pure} of weight $w$ if
\begin{itemize}
\item for each $v \not\in S$, each root $\alpha$ of $Q_v(X)$ in $\barM$ and each $\imath:\barM \into \C$
we have 
\[ | \imath \alpha |^2 = q_v^w; \]
\item and for each $\tau:F \into \barM$ and each complex conjugation $c$ in $\Gal(\barM/\Q)$ we have
\[ H_{c \tau} = \{ w-h: \,\,\, h \in H_\tau\}. \]
\end{itemize}

We will call $\CR$ {\em strictly pure} of weight $w$ if
\begin{itemize}
\item $\CR$ is strictly compatible and for each prime $v$ of $F$ the Weil--Deligne representation $\WD_v(\CR)$ is pure of weight $w$;
\item and for each $\tau:F \into \barM$ and each complex conjugation $c$ in $\Gal(\barM/\Q)$ we have
\[ H_{c \tau} = \{ w-h: \,\,\, h \in H_\tau\}. \]
\end{itemize}

 If $\imath:M \into \C$ we define the partial L-function 
\[ L^S(\imath \CR,s) = \prod_{v \not\in S} (q_v^{ns}/\imath Q_v(q_v^s)). \] This may or may not converge. If $\CR$ is pure of weight $w$ then it will converge to an analytic
function in $\Re s > 1+w/2$. If $\lambda|l$ and every place of $F$ above $l$ lies in $S$, then $L^S(\imath \CR,s)$ depends only on $r_\lambda$ so, if $\widetilde{\imath}:\barM_\lambda \iso \C$ extends $\imath$, we will sometimes write $L^S(\widetilde{\imath} r_\lambda,s)$ instead of $L^S(\imath \CR, s)$. This makes sense even for $r_\lambda$ not part of a weakly compatible system, provided that $S$ contains all primes above $l$ and all primes at which $r_\lambda$ ramifies.

If $\CR$ is strictly compatible then we can define the $L$-function
\[ L(\imath \CR,s) = \prod_{v\ndiv \infty} L(\imath \WD_v(\CR), s) \]
which differs from $L^S(\imath \CR,s)$ only by the addition of finitely many Euler factors. 

Suppose that $\CR$ is
strictly compatible, pure of weight $w$ and regular. Also let
$\psi=\prod_v \psi_v: \A_F/F \ra \C^\times$ be the non-trivial
additive character such that if $v$ is real then $\psi_v(x)=e^{2 \pi i
  x}$;  if $v$ is complex then $\psi_v(x)=e^{2\pi i
  (x+c_vx)}$; while if $v$ is $p$-adic then
$\psi_v(x)=\psi_p(\tr_{F_v/\Qp}(x))$ where $\psi_p|_{\bb{Z}_p}=1$ and
$\psi_p(1/p)=e^{-2\pi i/p}$. We wish to define the completed $L$-function and $\epsilon$-factors of $\CR$. (See \cite{MR546607}.)
\begin{itemize}
\item Write $\Gamma_\R(s)=\pi^{-s/2}\Gamma(s/2)$ and
$\Gamma_\C(s)=2(2\pi)^{-s}\Gamma(s)=\Gamma_\R(s)\Gamma_\R(s+1)$. 

\item Suppose that $v$ is a complex infinite place of $F$  and that $\tau, \tau' \in \Hom(F, \barM)$ are the two distinct elements such that $\imath \circ \tau$ and $\imath \circ \tau'$ extend to continuous isomorphisms $F_v \iso \C$. Then we set
\[ \begin{array}{rcr} L_v(\imath\CR,s) &= & \Gamma_\C(s-w/2)^{\dim \CR} \prod_{h \in H_\tau, \,\, h< w/2} (\Gamma_\C(s-h)/\Gamma_\C(s-w/2)) \\ && \prod_{h \in H_{ \tau'}, \,\, h< w/2} (\Gamma_\C(s-h)/\Gamma_\C(s-w/2)) \end{array} \]
and
\[ \epsilon_v(\imath \CR, \psi_v,s) = i^{\sum_{h \in H_\tau}|h-w/2|+\sum_{h \in H_{\tau'}}|h-w/2|}. \]

\item Suppose that $v$ is a real infinite place of $F$  and that $\tau:F \into \barM$ is such that $\imath \circ \tau$ extends to a continuous isomorphism $F_v \iso \R \subset \C$. As $\{\det r_\lambda \}$ is a weakly compatible system of characters, $\det r_\lambda(c_v)=\pm 1$ is independent of $\lambda$. We set $\det \CR(c_v)=\det r_\lambda(c_v) \in \{ \pm 1\}$. 

If $\dim \CR$ is even set $d_{\pm}=\dim \CR/2$, while, if $\dim \CR$ is odd (in which case $w/2 \in H_\tau$ so that $w$ is even) set
\[ d_{\pm} = (\dim \CR \pm (-1)^{w/2} (\det \CR)(c_v))/2. \]
(The reader might like to think of $d_{\pm }$ as $\dim r_\lambda^{c_v=\pm (-1)^{w/2}}$. Because we didn't make a compatibility assumption between the $r_\lambda$ at the infinite place $v$ this doesn't make sense directly. However we make use of our regularity assumption to give this alternate definition, which does make sense and suffices for our purposes.)

Now define
\[ L_v(\imath \CR,s) = \Gamma_\R(s-w/2)^{d_+} \Gamma_\R(s+1-w/2)^{d_-} \prod_{h \in H_\tau,\,\, h<w/2} (\Gamma_\C(s-h)/\Gamma_\C(s-w/2)), \]
and
\[ \epsilon_v(\imath \CR, \psi_v,s)=i^{d_-+ \sum_{h \in H_\tau}|h-w/2|}. \]

\item Finally we define the completed $L$-function
\[ \Lambda(\imath \CR,s)=L(\imath \CR,s) \prod_{v|\infty} L_v(\imath\CR,s)  \]
 and the epsilon factor 
 \[ \epsilon(\imath \CR,s)=\left( \prod_{v \ndiv \infty} \epsilon(\imath \WD_v(\CR),\psi_v,s) \right) \left(\prod_{v|\infty} \epsilon_v(\imath \CR,\psi_v,s) \right) . \]
 The latter does not depend on the choice of $\psi$.
\end{itemize}

 We will call $\CR$ {\em automorphic} if there is a regular algebraic, cuspidal  automorphic
representation $\pi$ of $\GL_n(\A_F)$ and an embedding $\imath:M \into \C$, such that if $v \not\in S$ then $\pi_v$ is unramified and 
$\rec(\pi_v |\det|_v^{(1-n)/2})(\Frob_v)$ has characteristic
polynomial $\imath (Q_v(X))$. Note that if $\CR$ is polarizable then
so is $\pi$. It follows from Theorem 3.13 of \cite{MR1044819} that,
when $\CR$ is automorphic,
for any embedding $\imath':M \into \C$ there is a regular, algebraic, cuspidal  automorphic
representation $\pi'$ of $\GL_n(\A_F)$, such that if $v \not\in S$ then $\pi_v'$ is unramified and 
$\rec(\pi_v' |\det|_v^{(1-n)/2})(\Frob_v)$ has characteristic polynomial $\imath' (Q_v(X))$. [It would be more natural not to include the assumption that $\pi$ is regular. However for the purposes of this paper our definition will be more convenient.]

We will call an $n$-dimensional polarized weakly compatible system $(\CR,\CM)$ {\em automorphic} if there is a regular algebraic, cuspidal, polarized automorphic representation $(\pi,\chi)$ of $\GL_n(\A_F)$ and an embedding $\imath:M \into \C$ such that
\begin{itemize}
\item if $v \not\in S$ then $\pi_v$ is unramified and $\rec(\pi_v |\det|_v^{(1-n)/2})(\Frob_v)$ has characteristic polynomial equal to the image under $\imath$ of the polynomial $Q_v(X)$ for $\CR$, 
\item and if $v \not\in S_{F^+}$ then $\chi_v$ is unramified and $\rec(\chi_v |\,\,\,|_v^{1-n})(\Frob_v)$ has characteristic polynomial equal to the image under $\imath$ of the polynomial $Q_v(X)$ for $\CM$.
\end{itemize}

If $F'/F$ is a finite extension we define $\CR|_{G_{F'}}$ to be the weakly compatible system of representations of $G_{F'}$:
\[ (M,S^{(F')},\{ Q_v^{(F')}(X) \}, \{r_\lambda|_{G_{F'}} \}, \{H_\tau^{(F')}\} ), \]
where 
\begin{itemize}
\item $S^{(F')}$ is the set of primes of $F'$ above $S$; 
\item $H^{(F')}_\tau=H_{\tau|_F}$;
\item $Q_v^{(F')}(X)$ is the monic polynomial in $M[X]$ of degree $n$ with roots the $\alpha^{[k(v):k(v|_F)]}$ as $\alpha$ runs over roots of $Q_{v|_F}(X)$.
\end{itemize}

We remark that if $(\pi,\chi)$ is a regular algebraic, cuspidal, polarized automorphic
representation of $\GL_n(\A_F)$ then $\{ r_{l,\imath}(\chi)\}$ is a
strictly pure compatible system of some necessarily even weight
(because $F^+$ is totally real), which we will write $2(w+1-n)$. Moreover  $\{ r_{l,\imath}(\pi) \}$ is a strictly pure compatible system of weight $w$. (See Theorem \ref{grfaf}.)

\subsection{Rational compatible systems.}\label{rcs}{$\mbox{}$}\newline

 In this section we present a formulation of the results of Larsen from \cite{larsen}. 

We will consider a weakly compatible system $\CR=(\Q,S,\{Q_v(X)\},\{r_l\},\{ H_\tau\})$, where for each $l$ we have
\[ r_l:G_F \lra \GL_n(\Q_l). \]
We are going to make a number of definitions which depend on $\CR$, but for simplicity we won't put $\CR$ in the notation. We hope this will cause no confusion.

We will write $V_l$ for the $\Q_l$-vector space underlying $r_l$. We will also write $G_l$ for the Zariski closure of the image of $r_l$ in $\GL_n/\Q_l$, and $G_l^0$ for the connected component of the identity in $G_l$. Thus $G_l^0$ is a reductive group over  $\Q_l$. Let $\Gamma_l$ denote the image of $G_F$ in $G_l(\Q_l)$ and set $\Gamma_l^0=\Gamma_l \cap G_l^0(\Q_l)$. There is a finite Galois extension $F^0/F$ such that $\Gal(F^0/F) \iso \Gamma_l/\Gamma_l^0$ for all $l$ (Proposition 6.14 of \cite{lp}). Since $r_l$ is Hodge-Tate, the Lie algebra of $\Gamma_l$ is algebraic (as defined in Section II.7 of \cite{borellag}) and the group $\Gamma_l$ is open in $G_l(\Q_l)$ (\cite{bogomolov}).

Let $Z_l$ (resp.\ $G_l^\der$) denote the centre (resp.\ derived subgroup) of $G_l^0$.
Let $G_l^\ad=G_l^0/Z_l$ and let $C_l=G_l^0/G_l^\der$. Also let $G_l^\SC$ denote the simply connected cover of $G_l^\ad$.  We have surjective maps with finite kernels
\[ G_l^\SC \lra G_l^\der \lra G_l^\ad. \]
Because the dimension of $G_l^\ad$ is bounded only depending on $n$,
we see that there is a positive constant $A(n)$ depending only on $n$ (and not on $\CR$ or $l$) such that
\[ \# \ker(G_l^\SC \ra G_l^\ad) | A(n). \]
The natural map $Z_l \ra C_l$ is surjective with finite kernel of
order dividing $A(n)$. We will write $\Gamma_l^Z$ for $\Gamma_l^0\cap Z_l(\Q_l)$ and $\Gamma_l^C$ for the image of $\Gamma_l^0$ in $C_l(\Q_l)$. As the cokernel of the map $G_l^0(\Q_l) \ra C_l(\Q_l)$ is finite (because it is a quotient of $C_l(\Q_l)/Z_l(\Q_l) \subset H^1(G_{\Q_l}, \ker(Z_l \ra C_l))$), we see that $\Gamma_l^C$ is open in $C_l(\Q_l)$.

Set $H_l=G_l^\SC \times Z_l$. Then there is a natural surjection of algebraic groups $H_l \onto G_l^0$, with a finite, central kernel with order dividing $A(n)$ (as it equals the kernel of $G_l^\SC \ra G_l^\ad$). Using Galois cohomology we see that the cokernel of the map $H_l(\Q_l) \ra G_l^0(\Q_l)$ is a finite abelian group of order dividing $A(n)^3$ and exponent dividing $A(n)$. (The cokernel embeds in $H^1(G_{\Q_l},\ker(H_l \ra G_l^0))$, which by the local Euler characteristic formula has order dividing $(\#\ker(H_l \ra G_l^0))^3$.) We will write
$\Gamma_l^{00}=\Gamma_l^0 \cap \Im(H_l(\Q_l) \ra G_l^0(\Q_l))$ and $\Gamma_l^H$ for the pre-image of $\Gamma_l^{00}$ in $H_l(\Q_l)$. Thus the kernels of 
\[ \Gamma_l^H \lra \Gamma_l^0 \]
and
\[ \Gamma_l^Z \lra \Gamma_l^C \]
are both finite abelian groups of order dividing $A(n)$, while the cokernels are both 
finite abelian groups of order dividing $A(n)^3$ and exponent dividing $A(n)$. It will often be convenient to work with $\Gamma_l^H$ in place of $\Gamma_l^0$, because it is easier to control. 

As $G_l$ acts by conjugation on $G_l^0$ it also acts on $Z_l$, $G_l^\der$, $G_l^\ad$, $C_l$, $G_l^\SC$ and $H_l$. Thus $\Gamma_l$ does as well. Moreover, $\Gamma_l^{00}$ is a normal subgroup of $\Gamma_l$ and the conjugation action of $\Gamma_l$ on $\Gamma_l^{00}$ lifts to an action on $\Gamma_l^H$. 
Let $T_l$ denote a maximal torus in $G_l^0$ which we assume to be
chosen unramified whenever $G_l^0$ is unramified. Let
$T_l^\ad=T_l/Z_l$, let $T_l^\der=\ker(T_l \ra C_l)^0$, let $T_l^\SC$
equal the connected preimage of $T_l^\ad$ in $G_l^\SC$, and let $T_l^H=T_l^\SC \times Z_l$. We have natural embeddings 
\[ X^*(T_l^\ad) \subset X^*(T_l^\der) \subset X^*(T_l^\SC) \subset (1/A(n)) X^*(T_l^\ad) \]
and
\[ A(n) X^*(T^H_l) \subset X^*(T_l) \subset X^*(T^H_l) \]
and
\[ A(n) X^*(Z_l) \subset X^*(C_l) \subset X^*(Z_l). \]
Let $\Delta \subset X^*(T_l^\ad) \subset X^*(T_l)$ denote a
  basis for the root system of $G_l^0$. 
  
 The dimensions of $V_l$ and $G_l^\SC$ are bounded only in terms of $n$. Hence, if $(1/A(n)) \sum_{\delta \in \Delta} m_\delta \delta$ is a weight of $G_l^\SC$ on $V_l$, then the $|m_\delta|$ can be bounded by a constant $B(n)$ depending only on $n$ (and not on $\CR$ or $l$). 
  
  If $\mu \in X^*(Z_l)$ is a
  $Z_l$-weight of $V_l$ then we can find $m_{\mu, \delta} \in \Z$ for
  $\delta \in \Delta$ such that $(1/A(n))\sum_{\delta \in \Delta}
  m_{\mu,\delta}\delta$ is a weight of $T_l^\der$ on $V_l^\mu$.  Thus 
  \[ \left((1/A(n))\sum_{\delta \in \Delta} m_{\mu,\delta}\delta, \,\, \mu\right) \in X^*(T_l^H) \]
  is a weight of $V_l$ and $|m_{\mu,\delta}|\leq B(n)$ for all $\delta \in \Delta$.
  
  Let $\tS_{F^0,l}$ denote the restriction of scalars from $\CO_{F^0,l}$ to $\Z_l$ of $\G_m$ and let $S_{F^0,l}=\tS_{F^0,l} \times \Q_l$.
  Note that  $\Hom(F^0,\barQQ_l)$ gives a natural basis of $X^*(S_{F^0,l} )$. There is a homomorphism 
  \[ \theta_l: S_{F^0,l} \lra C_l \]
  such that $\theta_l$ agrees with $(r_l  \bmod G_l^\der(\Q_l))\circ \Art_{F^0}$ on an open subgroup of $\CO_{F^0,l}^\times$. (See Sections
  III.1.2 and III.2.1 of \cite{MR0263823}. To aid comparison with \cite{MR0263823} we remark that $i$ of Section II.1.1 of \cite{MR0263823} is the inverse of our $\Art_{F^0}$. Thus $\theta_l$ is constructed from $r_l \mod G_l^\der(\Q_l)$ in just the same way as $r$ is constructed from $\rho$ in Section II.1.1 of \cite{MR0263823}.) 
  
  \begin{lem}\label{charbd} Keep the notation and assumptions established earlier in this section. \begin{enumerate}
  \item $\theta_l:S_{F^0,l} \ra C_l$ is surjective.
  \item If $l \not\in S$ then $\theta_l=(r_l  \bmod G_l^\der(\Q_l)) \circ \Art_{F^0}$ on all of $\CO_{F^0,l}^\times$.
  \item There is a constant $C(\CR)$ depending only on $\CR$ (and not on $l$) such that, if $\mu \in X^*(Z_l)$ is a weight of $Z_l$ on $V_l$, then 
  \[ (A(n) \mu )\circ \theta_l = \sum_{\sigma \in \Hom(F^0,\barQQ_l)} {m_{\mu,\sigma}} \sigma \]
  with $|m_{\mu,\sigma}|<C(\CR)$.
  \item There is a constant $D(\CR)$ depending only on $\CR$ (and not on $l$) such that 
  \[ \# (X^*(S_{F^0,l})/\theta_l^* X^*(C_l))^\tor \leq D(\CR). \]
  \end{enumerate} \end{lem}
  
  \begin{proof}
  For the first part let $U$ denote the open subgroup of $\CO_{F^0,l}^\times$ where 
 $\theta_l$ and $(r_l \bmod G_l^\der(\Q_l)) \circ \Art_{F^0}$ agree. As $C_l(\Q_l)$ has an open subgroup which is pro-$l$, we see that 
 \[ ((r_l \bmod G_l^\der(\Q_l)) \circ \Art_{F^0})(U) =\theta_l(U) \]
 is open, and hence of finite index, in 
 \[ ((r_l \bmod G_l^\der(\Q_l)) \circ \Art_{F^0})(\A_{F^0}^\times) = (r_l \bmod G_l^\der)(G_{F^0}).\]
As the image of $G_{F^0}$ in $C_l(\Q_l)$ is Zariski dense and $C_l$ is connected we deduce that $\theta_l(U)$ is Zariski dense in $C_l/\Q_l$ and the result follows.

 For the second part, if $l\not \in S$, then $r_l:G_F \to G_l(\Q_l)$ is crystalline at
each prime above $l$, and hence so is $r_l\mod G^{\der}_l(\Q_l):G_F
\to C_l(\Q_l)$. Proposition 6.3 of \cite{cco}, then implies that for
each $\mu\in X^*(C_l)$, the characters $\mu\circ \theta_l$
and $\mu \circ (r_l  \bmod G_l^\der(\Q_l))\circ \Art_{F^0}$ agree on
all of $\CO_{F^0_l}^\times$. Since this holds for all $\mu\in
X^*(C_l)$, it follows immediately that $\theta_l$
and $(r_l  \bmod G_l^\der(\Q_l))\circ \Art_{F^0}$ agree on all of $\CO_{F^0_l}^\times$.
    
  For the third part note that $-m_{\mu,\sigma}$ is the $\sigma$-Hodge--Tate number of $(A(n) \mu) \circ (r_l \bmod G_l^0(\Q_l))$.   If $v|l$ is a prime of $F^0$ then there is an element $\nu_{\HT,v} \in X_*(T_l)$ such that for any
  algebraic representation $\rho$ of $G_l^0$ defined over $\Q_l$, the
  Hodge--Tate numbers (with respect to any continuous $\sigma : F^0_v\into \barQQ_l$) of $\rho \circ r_l |_{G_{F^0_v}}$ are the $\langle \mu,
  \nu_{\HT,v} \rangle$, where $\mu$ runs over the weights of $\rho$ in
  $X^*(T_l)$. (See section 1.2 of \cite{wintenberger}.) Thus what we have to show is that if  $\mu \in X^*(Z_l)$ is a weight of $Z_l$ on $V_l$, then
  \[ | \langle A(n) \mu, \nu_{\HT,v} \rangle | \]
  is bounded independently of $l$, $\mu$ and $v$. (Here we think of $A(n) \mu \in A(n)X^*(Z_l) \subset X^*(C_l) \subset X^*(T_l)$.) 
  However
  \[ ((1/A(n))\sum_{\delta \in \Delta} m_{\mu,\delta}\delta, \mu) \in X^*(T_l) \subset X^*(T_l^H) \]
  is a weight of $T_l$ on $V_l$ and so
  \[ | \langle (\sum_{\delta \in \Delta} m_{\mu,\delta}\delta, A(n)\mu), \nu_{\HT,v} \rangle | \]
  is bounded independently of $l$, $\mu$ and $v$ (because the Hodge-Tate numbers of $r_l$ are independent of $l$). As the $m_{\mu,\delta}$ are also bounded independently of $l$ and $\mu$, we see that it suffices to show that each
  \[ | \langle \delta, \nu_{\HT,v} \rangle | \]
  is bounded independently of $l$, $\delta \in \Delta$ and $v$. But $\Lie G_l^0 \subset \Hom_{\Q_l}(V_l,V_l)$ so that each $\delta \in\Delta$ is the difference of two weights of $T_l$ on $V_l$. It follows that 
 \[ | \langle \delta, \nu_{\HT,v} \rangle | \]
  is bounded independently of $l$, $\delta \in \Delta$ and $v$, as desired. (Again because the Hodge-Tate numbers of $r_l$ are independent of $l$.)
  
  The fourth part follows from the third.
  \end{proof}
  
 \begin{prop}\label{lar} Keep the notation and assumptions established earlier in this section. There is a Dirichlet density $1$ set $\CL$ of rational primes with the following properties.
\begin{enumerate}
\item If $l \in \CL$ then $G_l^0$, and hence also $Z_l$, $C_l$, $G_l^\SC$ and $H_l$, are unramified. Write $\tZ_l$ (resp.\ $\tC_l$) for the torus over $\Z_l$ with generic fibre
$Z_l$ (resp.\ $C_l$).  
\item If $l \in \CL$ then there is a semi-simple group scheme
  $\tG_l^\SC/\Z_l$ with generic fibre $G_l^\SC$ 
such that $\Gamma_l^H=\tG_l^\SC(\Z_l) \times \Gamma_l^Z$. Write
  $\tH_l=\tG_l^\SC \times \tZ_l$.
\item The index $[\tZ_l(\Z_l):\Gamma_l^Z]$ is bounded independently of $l\in \CL$.
\item If $l \in \CL$ then the conjugation action of $\Gamma_l$ on $H_l$ extends (uniquely) to an action on $\tH_l$. This makes $V_l$ into an $\tH_l \rtimes \Gamma_l$-module.
\item If $l \in \CL$ then $V_l$ contains an $\tH_l \rtimes \Gamma_l$ invariant $\Z_l$-lattice. 
\item Suppose that $l \in \CL$. There is a finite unramified extension $M_\lambda/\Q_l$ (of degree bounded independently of $l \in \cL$) such that all $G_l^\SC$-irreducible subquotients of $V_l \otimes \barQQ_l$ can be defined over $M_\lambda$. Choose such a field $M_\lambda$ and let 
\[ V_l \otimes M_\lambda = \bigoplus_i V_{\lambda,i}\]
 be the decomposition of $V_l \otimes M_\lambda$ into maximal $G_l^\SC$-isotypical submodules. Suppose also that $\Lambda \subset V_l \otimes M_\lambda$ is a $\tH_l$-invariant $\CO_{M_\lambda}$-lattice. Then
\[ \Lambda = \bigoplus_i (\Lambda \cap V_{\lambda,i}). \]
Moreover all irreducible $\tG_l^\SC(\Z_l)$-subquotients of $\Lambda \cap V_{\lambda,i}$ are absolutely irreducible and isomorphic, say to $\barrho_i$. Moreover $\dim_{k(\lambda)} \barrho_i$ equals the dimension (over $M_\lambda$) of an irreducible constituent of $V_{\lambda,i}$. If $\barrho_i \cong \barrho_j$ then $i=j$. 
\end{enumerate} \end{prop}

\begin{proof}
Proposition 8.9 of \cite{lp} tells us that we can find a Dirichlet density $1$ set $\CL$ of rational primes such that for $l \in \CL$ the group $G_l^0$ is unramified. As the dimension of $T_l$ is bounded independently of $l$ we see that the order of any finite order element of $\Aut(X^*(T_l))$ is independent of $l$. Hence for $l \in \cL$ the group splits over an unramified extension $M_\lambda/\Q_l$ of degree bounded independently of $l$.

Theorem 3.17 of \cite{larsen} tells us that we can replace $\CL$ by a possibly smaller set of Dirichlet density $1$, which we will also denote $\CL$, such that for $l \in \CL$ there is a semi-simple group scheme $\tG_l^\SC/\Z_l$ with generic fibre $G_l^\SC$ such that 
\[ \Gamma_l^H \cap G_l^\SC(\Q_l) = \tG_l^\SC(\Z_l).\]
As $\tG_l^\SC(\Z_l)$ is a maximal compact subgroup of $G_l^\SC(\Q_l)$ we see that $\tG_l^\SC(\Z_l)$ is also the projection of $\Gamma_l^H$ to $G_l^\SC(\Q_l)$ and so $\Gamma_l^H=\tG_l^\SC(\Z_l) \times \Gamma_l^Z$ for some compact subgroup $\Gamma_l^Z \subset Z_l(\Q_l)$. As $\tZ_l(\Z_l)$ is the unique maximal compact subgroup of $Z_l$ we see that $\Gamma_l^Z \subset \tZ_l(\Z_l)$. As the image of $\Gamma_l^Z$ in $\Gamma_l^C$ has finite index we see that $\Gamma_l^Z$ is open in $\tZ_l(\Z_l)$, and we have verified the first two parts of the Proposition.

Removing finitely many primes from $\CL$ we may suppose that for all $l \in \CL$ the representation $r_l$ is crystalline at $l$ and $l$ is unramified in $F^0$. Then for $l \in \CL$ we have
\[ \Gamma_l^C \supset \theta_l(\tS_{F^0,l}(\Z_l)). \]
The final part of Lemma \ref{charbd} together with Lemma \ref{nrtorus} shows that the index of $\Gamma_l^C$ in $\tC_l(\Z_l)$ is bounded by $D(\CR)$ (independently of $l \in \CL$). However the index of $\Gamma_l^Z$ in $\tZ_l(\Z_l)$ is bounded by the product
\[ \# \ker(Z_l \ra C_l) \times [\Gamma_l^C:\Gamma_l^Z] \times [\tC_l(\Z_l):\Gamma_l^C] \leq A(n)^4 D(\CR), \]
and so the third part of the Proposition follows.

Suppose that $\gamma \in \Gamma_l$, then we have seen that $\gamma$ acts on $G_l^\SC$. We will consider the reductive group scheme 
\[ {}^\gamma \tG^\SC_l=\Spec \left( (\gamma^*)^{-1} \CO_{\tG^\SC_l}(\tG^\SC_l)\right), \]
where $(\gamma^*)^{-1} \CO_{\tG^\SC_l}(\tG^\SC_l) \subset \CO_{G^\SC_l}(G^\SC_l)$ is the pre-image of $\CO_{\tG^\SC_l}(\tG^\SC_l) \subset \CO_{G^\SC_l}(G^\SC_l)$ under $\gamma^*$. By Proposition 5.1.40 of \cite{bruhattits2} we see that $\tG^\SC_l$ and ${}^\gamma \tG^\SC_l$ must be the group schemes (with connected fibres) attached to special points $x$ and $x'$ in the building of $G^\SC_l$ in \cite{bruhattits2}. (The group schemes $\gG_x^{0\natural}$ and $\gG_{x'}^{0\natural}$ in the notation of \cite{bruhattits2}.) 
The sets $\{ x \}$ and $\{ x'\}$ are facets for the building of $G_l^\SC$ (as $G_l^\SC$ is semi-simple). As 
\[ ({}^\gamma \tG_l)(\Z_l)={}^\gamma(\tG_l(\Z_l)) = \tG_l(\Z_l) \]
we deduce from the proof of Proposition 5.2.8 of \cite{bruhattits2} that $\{x\}=\{x'\}$, and hence ${}^\gamma\tG_l^\SC=\tG_l^\SC$. Thus the action of $\Gamma_l$ on $G_l^\SC$ extends to one on $\tG_l^\SC$. (We thank Jiu-Kang Yu for valuable help with this argument.) As $\tZ_l$ is the unique torus with generic fibre $Z_l$ we also see that the action of $\Gamma_l$ on $Z_l$ extends to an action on $\tZ_l$. Thus the action of $\Gamma_l$ on $H_l$ extends to one on $\tH_l$, and the fourth part of the Proposition follows. 

As in the fourth paragraph of Section 1.12 of \cite{larsen} we can find a $\Z_l$-lattice $\Lambda \subset V_l$ such that $\Lambda \otimes \Z_l^\nr$ is $\tH_l(\Z_l^\nr)$-invariant. Replacing $\Lambda$ by the sum of its $\Gamma_l$-translates, which is again a lattice because $\Gamma_l$ is compact, we see that we may suppose $\Lambda$ to also be $\Gamma_l$-invariant. Again as in the fourth paragraph of Section 1.12 of \cite{larsen} we see that the natural map $H_l \ra \GL(V_l)$ extends to a map $\tH_l \ra \GL(\Lambda)$. This establishes the fifth part of the Proposition.

Let $\{ \mu_i\}$ be the set of highest weights (with respect to
$\Delta$) of irreducible $G_l^\SC$-submodules of $V_l \otimes
\barQQ_l$. Let $\rho_{\mu_i}$ be the corresponding irreducible
representations of $G_l^\SC$. As in the third paragraph of section 1.12 of \cite{larsen} we see that each $\rho_{\mu_i}$ can be defined over $M_\lambda$. The first assertion of part six follows. Moreover $\rho_{\mu_i}$ extends to a representation $\trho_{\mu_i}$ of $\tG^\SC_l$ over $\CO_{M_\lambda}$. (By the argument of paragraph four of Section 1.12 of \cite{larsen}.) Let $\barrho_{\mu_i}$ denote the reduction of $\trho_{\mu_i}$ modulo $\lambda$, and let $\barrho_{\mu_i}^+$ denote the unique absolutely irreducible subquotient of $\barrho_{\mu_i}$ which contains $\mu_i$. (See paragraph five of section 1.12 of \cite{larsen}.) 

If $\delta \in \Delta$ let $\nu_\delta\in X^*(T_l^\SC)$ denote the fundamental weight corresponding to $\delta$, so that $\{ \nu_\delta\}$ is the basis of $X^*(T_l^\SC)$ dual to the basis of $X_*(T_l^\SC)$ consisting of the coroots corresponding to $\delta \in \Delta$. If $\mu_i=\sum_{\delta \in \Delta} m_{i,\delta} \nu_\delta$ then the $m_{i,\delta}$ are bounded independently of $l$. (As $\dim G_l^\SC$ is bounded independently of $l$, there are only finitely many possibilities for the change of basis matrix between $\Delta$ and  $\{ \nu_\delta: \,\,\delta \in \Delta\}$.) Thus after removing a finite number of elements from $\CL$ we may assume that for all $l \in \CL$ and all $i$ we have  $0 \leq m_{i,\delta} < l$. As in the first paragraph of section 1.13 of \cite{larsen} we see that $\barrho_{\mu_i}^+$ is an absolutely irreducible representation of $\tG^\SC_l(\F_l)$ and that $\barrho_{\mu_i}^+ \cong \barrho_{\mu_j}^+$ implies $i=j$. Again removing a finite number of primes from $\CL$ we can further assure, as in the second paragraph of Section 1.13 of \cite{larsen}, that $\barrho_{\mu_i}^+ = \barrho_{\mu_i}$ (for all $l \in \CL$ and all $i$). 

Now suppose that $\Lambda$ is as in part 6 of the Proposition. (We no longer
use $\Lambda$ to denote the lattice constructed in the proof of the
fifth part.) We see that $\barrho_{\mu_i}^+$ is the only irreducible subquotient of  $\Lambda \cap V_{\lambda,i}$. Suppose that 
\[ \Lambda \neq \bigoplus_i (\Lambda \cap V_{\lambda,i}). \]
Choose a $\tG_l^\SC$-invariant lattice $\Lambda'$ with
\[ \Lambda \supset \Lambda' \supset \bigoplus_i (\Lambda \cap V_{\lambda,i}) \]
with $\Lambda' \left/ \left( \bigoplus_i (\Lambda \cap V_{\lambda,i}) \right) \right.$ simple and non-trivial. Then 
\[ \Lambda' \left/ \left( \bigoplus_i (\Lambda \cap V_{\lambda,i}) \right) \right. \]
must be equivalent to $\barrho_{\mu_j}^+$ for some $j$. Suppose $i \neq j$. Then we have a commutative diagram
\[ \begin{array}{ccc} \Lambda' & \onto & \Lambda' \left/ \left( \bigoplus_i (\Lambda \cap V_{\lambda,i}) \right) \right. \\ \da && \da \\
V_{\lambda,i} & \onto & V_{\lambda,i}/(\Lambda \cap V_{\lambda,i}), \end{array} \]
and we see that the right hand vertical arrow must be zero. Thus the image of $\Lambda'$ under projection to $V_{\lambda,i}$ is $\Lambda \cap V_{\lambda,i}$. We conclude that 
\[ \Lambda' = \left( \bigoplus_{i \neq j} \Lambda \cap V_{\lambda,i} \right) +  \bigcap_{i \neq j} \ker(\Lambda' \ra V_{\Lambda,i}). \]
However 
\[ \bigcap_{i \neq j} \ker(\Lambda' \ra V_{\Lambda,i}) = \Lambda \cap V_{\lambda,j} \]
and so
\[ \Lambda'= \bigoplus_i (\Lambda \cap V_{\lambda,i}), \]
a contradiction. This completes the proof of the Proposition.
\end{proof}

 We remark that we won't need part (3) of the Proposition, but we thought it was worth recording it anyway. We also remark that using part (3) of the Proposition one can prove a version of part (6) in which $H_l$ replaces $G_l^\SC$ and $\Gamma_l^H$ replaces $\tG_l^\SC(\Z_l)$. However, as we won't need it, we chose not to present the details.

\subsection{Compatible systems: lemmas.}\label{csl}{$\mbox{}$} \newline

We now return to more general compatible systems.

\begin{lem}\label{cc} Suppose that $\CR$ is a weakly compatible system of $l$-adic representations of $G_F$ of dimension $n$. Let $G_\lambda$ denote the Zariski closure of $r_\lambda(G_F)$ in $\GL_n/\barM_\lambda$ and let $G_\lambda^0$ denote its connected component. Then:
\begin{enumerate}
\item There is a finite Galois extension $F^1/F$ such that for all $\lambda$ the map $G_F \ra G_\lambda(\barM_\lambda)$ induces an isomorphism
$\Gal(F^1/F) \iso G_\lambda/G_\lambda^0$.

\item If $\CR$ is regular and $H$ is an open subgroup of $G_F$ then any irreducible 
$H$-subrepresentation of $r_\lambda$ has multiplicity one.

\item If $\CR$ is regular, then after replacing $M$ by a finite extension we may suppose that for any open subgroup $H \subset G_F$ and any $\lambda$ and any $H$-subrepresentation $s$ of $r_\lambda$, the representation $s$ is defined over $\CO_{M,\lambda}$. 
\end{enumerate}\end{lem}

\begin{proof} The proof of the first part is the same as the proof of Proposition 6.14
  of \cite{lp}. (One must replace $\CG$ by $G_F$, and $\rho_l$ by $r_\lambda$, and $G_l^0$ by $G_\lambda^0$; and $\Q_l$ by $\barM_\lambda$, and $\Q$ by $M$ in section 6 of \cite{lp}. These arguments are due to Serre.) 
  
As $r_\lambda$ is semi-simple we see that $G_\lambda^0$ is reductive. As $\tr r_\lambda$ is continuous on $G_F$ and $M_\lambda$ is closed in $\barM_\lambda$, we see that $\tr r_\lambda$ is valued in $M_\lambda$.

Now assume that $\CR$ is regular. By Theorem 1 of \cite{sen} there is an element of $(\Lie G_\lambda^0) \otimes \widehat{\barM_\lambda}$ with $n$ distinct eigenvalues. Thus if $T_\lambda$ is a maximal torus in $G_\lambda^0$ it has $n$ distinct weights under $r_\lambda$. In particular all irreducible $G_\lambda^0$-subrepresentations of $r_\lambda$ have multiplicity one.
If $H$ is any open subgroup of $G_F$ then its Zariski closure in
$G_\lambda$ contains $G_\lambda^0$. Thus any irreducible
$H$-subrepresentation of $r_\lambda$ has multiplicity one, and the
second part of the Lemma is proved. (We note that the second part
follows immediatetly from a consideration of the Hodge-Tate
weights but we shall need the facts deduced in this paragraph below.)

 Moreover the set of elements of $G_\lambda^0$ with $n$ distinct eigenvalues  under $r_\lambda$ is a non-empty Zariski open subset. As the images of Frobenius elements at primes which split completely in $F^1/F$ are Zariski
dense in $G_\lambda^0$ we conclude that infinitely many $Q_v(X)$, for $v$ splitting completely in $F^1/F$, have
$n$ distinct roots. Replace $M$ by the splitting field over $M$ of the
product $Q_v(X)Q_{v'}(X)$ for two such $v$, $v'$ with distinct residue characteristic. Then for all $\lambda$ the image $r_\lambda(G_{F^1})$ contains an element 
with $n$ distinct $M_\lambda$-rational eigenvalues.

Let $H$ be any open subgroup of $G_F$. We must show that any $H$-subrepresentation of $r_\lambda$ is defined over $\CO_{M,\lambda}$. As $H$ is compact it suffices to show that it is defined over $M_\lambda$. As the Zariski closure of $H$ in $G_\lambda$ contains $G_\lambda^0$, it equals the Zariski closure of $HG_{F^1}$. Thus if $s$ is an $H$-subrepresentation of $r_\lambda$ it is also a $HG_{F^1}$-subrepresentation. So we are reduced to the case $H \supset G_{F^1}$. In this case the result follows from Lemma \ref{repth}.
\end{proof}

\begin{prop}\label{big} Suppose that $\CR$ is a regular, weakly
  compatible system of $l$-adic representations of $G_F$ defined over
  $M$. 
If $s$ is a subrepresentation of $r_\lambda$ then we will write $\bars$ for the semi-simplification of the reduction of $s$. Also write $l$ for the rational prime below $\lambda$. Then there is a set of rational primes $\CL$ of 
Dirichlet density $1$ (depending only on $\CR$), such that if $s$ is any irreducible subrepresentation of $r_\lambda$ for any $\lambda$ dividing any element of $\CL$ then $\bars|_{G_{F(\zeta_l)}}$ is irreducible.
 \end{prop}
 
 \begin{proof}
     If need be replace $M$ by a finite extension so that 
     \begin{itemize}
     \item $M/\Q$ is Galois, 
     \item and for every  $\lambda$ and every open subgroup $H \subset G_F$ any $H$-subrepresentation $s$ of $r_\lambda$ is defined over $\CO_{M,\lambda}$ (Lemma \ref{cc}). 
     \end{itemize}

For a rational prime $l$ define 
\[ r_l =\epsilon_l \oplus  \bigoplus_{\lambda|l} r_\lambda : G_F \lra \GL_{1+n[M:\Q]}(\Q_l). \]
Let
\[ H_{\Q}= \{-1\} \amalg \coprod_{\tau:F\into \barM} \coprod_{i=1}^{[M:F]}H_\tau, \]
and, for $v$ a prime of $F$ not in $S$, let 
\[ Q_{\Q,v}(X)=(X-q_v^{-1}) \prod_{\sigma \in \Gal(M/\Q)} {}^\sigma Q_v(X). \]

We are going to apply the results of section \ref{rcs} to
\[ \CR_\Q=(\Q,S,\{Q_{\Q,v}\}, \{ r_l\}, \{ H_\Q\}), \]
and we will use the notation established there without comment. If
need be, we replace $M$ by a finite extension so that $M$ contains the image of every embedding $F^0 \into \barM$. Also choose a set $\CL$ of rational primes of Dirichlet density one as in Proposition \ref{lar}. Removing a finite number of primes from $\CL$ we may further assume that
\begin{enumerate}
\item\label{Fnr} if $l\in \CL$ then $l$ is unramified in $F^0/\Q$;
\item\label{size} if $l \in \CL$ then $l > 4A(1+[M:\Q]n) C(\CR_\Q)+ 2$;
\item\label{cryscond} $l \not\in S_\Q$, where $S_\Q$ denotes the set of rational primes which lie below an element of $S$.
\end{enumerate}
Let $\Lambda_l$ be a $\tH_l \rtimes \Gamma_l$-invariant lattice in $V_l$ (Proposition \ref{lar}).

The character 
$\epsilon_l$ of $\Gamma_l$ extends to an algebraic character
\[ \epsilon_l: G_l \onto \G_m. \]
(It is surjective because $\epsilon_l(G_F)$ is infinite.) 
We will write $Z_l^1$, $C_l^1$ and $H_l^1$ for the kernel of $\epsilon_l$ on $Z_l$, $C_l$ and $H_l$. The character $\epsilon_l$ extends to a character of $\tC_l$. We will write $\tC_l^1$ (resp.\ $\tZ_l^1$) for the unique torus over $\Z_l$ with generic fibre $C_l^1$ (resp.\ $Z_l^1$), and set $\tH_l^1=\tG_l^\SC \times \tZ_l^1$. Set $\Gamma_l^{Z,1}=\Gamma_l^Z \cap Z^1_l(\Q_l)$ and $\Gamma_l^{C,1} = \Gamma_l^C \cap C^1_l(\Q_l)$, so that $\Gamma_l^{Z,1}$ is the pre-image of $\Gamma_l^{C,1}$ under
 $\Gamma_l^Z \ra \Gamma_l^C$. Also set $\Gamma_l^{H,1}=\tG_l^\SC(\Z_l) \times \Gamma_l^{Z,1}$. The conjugation action of $\Gamma_l$ on $\tZ_l$, $\tC_l$ and $\tH_l$ preserves $\tZ_l^1$, $\tC_l^1$ and $\tH_l^1$.
 
We will next prove the following claim:

{\em Suppose that $l \in \CL$, that $\lambda|l$ is a prime of $M$, and that $W_1$ and $W_2$ are two irreducible $M_\lambda[\Gamma_l^0]$-sub-modules of $V_l \otimes M_\lambda$. Then (by Zariski density) $W_i$ is $G_l^0$-invariant. Write $s_i$ for the representation of $H_l \rtimes \Gamma_l^0$ on $W_i$, and further assume that the semi-simplified reductions $\bars_i$ of $s_i$ are isomorphic as $\Gamma_l^{H,1}$-modules. Then $s_1 \cong s_2 \otimes \epsilon_l^a$ (as representations of $\Gamma_l^0$) for some $a \in \Z$.} 

\begin{proof} 
Set $\ts_i$ denote the action of $\tH_l$ on the intersection of $\Lambda_l \otimes \CO_{M_\lambda}$ with the space underlying $s_i$ and let $\bars_i$ denote its reduction modulo $\lambda$. From Proposition \ref{lar} we see that $\bars_i|_{\tG_l^\SC(\Z_l)}$ is absolutely irreducible and that $s_1|_{G_l^\SC} \cong s_2|_{G_l^\SC}$. 

Let $\mu_i:Z_l \ra \G_m$ denote the action of $Z_l$ on $s_i$, and let $\tmu_i$ denote the extension of $\mu_i$ to $\tZ_l$ and $\barmu_i$ the reduction of $\tmu_i$ modulo $\lambda$. Then
\[ (A(1+[M:\Q]n) \mu_i) \circ \theta_l = \sum m_{i,\sigma} \sigma \]
with $|m_{i,\sigma}|<C(\CR_\Q)$. We have
\[ \barmu_1|_{\Gamma_l^{Z,1}}=\barmu_2|_{\Gamma_l^{Z,1}}. \]
Hence 
\[ (A(1+[M:\Q]n)^2 \barmu_1)|_{\Gamma_l^{C,1}}=(A(1+[M:\Q]n)^2 \barmu_2)|_{\Gamma_l^{C,1}}, \]
and for some $(1-l)/2 < b < (l-1)/2$ we have
\[ (A(1+[M:\Q]n)^2 \barmu_1)|_{\Gamma_l^{C}}=(A(1+[M:\Q]n)^2 \barmu_2+b\epsilon_l)|_{\Gamma_l^{C}}. \]
By assumption \ref{cryscond} on $\CL$ we know that $\mu_i^{A(1+[M:\Q]n)}$ is a crystalline character of $G_{F^0}$. Thus 
\[ \prod_{\sigma \in \Hom(F^0,\Q_l^\nr)} \barsigma^{A(1+n[M:\Q])(m_{1,\sigma}-m_{2,\sigma})} = \prod_{\sigma \in \Hom(F^0,\Q_l^\nr)} \barsigma^b\]
on $(\CO_{F^0}/l)^\times$. If $v|l$ is a prime of $F^0$, choose an embedding $\sigma_v:F^0 \into \Q_l^\nr$ above $v$. Then the embeddings of $F^0_v$ into $\Q_l^\nr$ are $\Frob_l^i \circ \sigma_v$ for $i=0,\dots,f_v-1$ with $f_v=[k(v):\F_l]$. Then
\[ \sum_{i=0}^{f_v-1} A(1+[M:\Q]n)(m_{1,\Frob_l^i \circ \sigma_v}-m_{2,\Frob_l^i \circ \sigma_v})l^i \equiv b (l^{f_v}-1)/(l-1) \bmod (l^{f_v}-1). \]
As $(1-l)/2<A(1+[M:\Q]n)(m_{1,\sigma}-m_{2,\sigma})<(l-1)/2$ for all $\sigma$ (by assumption \ref{size} on $\CL$), both sides lie in the range $((1-l^{f_v})/2, (l^{f_v}-1)/2)$ and so
\[ \sum_{i=0}^{f_v-1} A(1+[M:\Q]n)(m_{1,\Frob_l^i \circ \sigma_v}-m_{2,\Frob_l^i \circ \sigma_v})l^i = b (l^{f_v}-1)/(l-1). \]
Then we see that $A(1+[M:\Q]n)(m_{1,\Frob_l^0 \circ \sigma_v}-m_{2,\Frob_l^0 \circ \sigma_v})\equiv b \bmod l$ and again using the bounds on both sides we conclude that $A(1+[M:\Q]n)(m_{1,\Frob_l^0 \circ \sigma_v}-m_{2,\Frob_l^0 \circ \sigma_v})= b$. Subtracting these terms, dividing by $l$ and arguing recursively we see that
\[ A(1+[M:\Q]n)(m_{1,\Frob_l^i \circ \sigma_v}-m_{2,\Frob_l^i \circ \sigma_v})= b \]
for all $i$. Thus $A(1+[M:\Q]n)|b$ and 
\[ \mu_1^{A(1+[M:\Q]n)} \circ \theta_l = (\mu_2^{A(1+[M:\Q]n)}  \otimes \epsilon_l^{b/A(1+[M:\Q]n)}) \circ \theta_l \]
as characters of $S_{F^0,l}$.  Thus $\mu_1=\mu_2$ as characters of
$Z_l^1$ and so, for some $a \in \Z$, we have $\mu_1=\mu_2
\epsilon_l^a$  as characters of $Z_l$. We deduce that $s_1=s_2
\epsilon_l^a$ as representations of $H_l$ and hence of $G_l^0$, and
hence again of $\Gamma_l^0$. The claim follows.
\end{proof}

We now return to the proof of the Proposition. Let $\lambda|l \in \CL$ be a prime of $M$. Let $s$ be an irreducible subrepresentation of $r_\lambda$. Then $s$ occurs on a $M_\lambda[\Gamma_l]$-submodule $W \subset V_l \otimes M_\lambda$. Let $W_0$ denote an irreducible $\Gamma_l^0$-submodule of $W$ and let $s_0$ denote the representation of $\Gamma_l^0$ on $W_0$. Note that $W_0$ is also $H_l$-invariant and $G_l^\SC$-irreducible.  
Let $\Gamma_l'$ denote the set of $\gamma$ in $\Gamma_l$ with
$s_0^\gamma \cong s_0$, or what comes to the same thing (by
regularity) $\gamma W_0=W_0 \subset W$ (see Lemma \ref{cc}). Then $s_0$ extends to a
representation of $\Gamma_l'$ and 
\[ s \cong
\Ind_{\Gamma_l'}^{\Gamma_l} s_0.\]
Write $\bars$ (resp.\ $\bars_0$) for the semi-simplified reduction of $s$ (resp.\ $s_0$).

 Let $W_1$ denote the $W_0$-isotypical component of $V_l \otimes M_\lambda$ for the action of $G_l^\SC$. By Proposition \ref{lar} we see that $(W_1 \cap (\Lambda_l \otimes \cO_{M_\lambda}))/\lambda (W_1 \cap (\Lambda_l \otimes \cO_{M_\lambda}))$ is isotypical for the action of $\tG_l^\SC(\Z_l)$ corresponding to an absolutely irreducible representation of dimension equal to $\dim_{M_\lambda} W_0$. Thus $(W_0 \cap (\Lambda_l \otimes \cO_{M_\lambda}))/\lambda (W_0 \cap (\Lambda_l \otimes \cO_{M_\lambda}))$ must be absolutely irreducible as a representation of $\tG_l^\SC(\Z_l)$. 
 
As $F^0/\Q$ is unramified above $l$ (assumption \ref{Fnr} on $\CL$) we see that $F(\zeta_l)$ is linearly disjoint from $F^0$ over $F$. Thus
\[ \bars|_{\ker \barepsilon_l|_{\Gamma_l}} \cong (\Ind_{\ker \barepsilon_l|_{\Gamma_l'}}^{\ker \barepsilon_l|_{\Gamma_l}} \bars_0)^\semis. \]
Hence it suffices to show that for $\gamma \in \Gamma_l-\Gamma_l'$ we have
\[ \bars_0^\gamma|_{\Gamma_l^{H,1}} \not\cong \bars_0|_{\Gamma_l^{H,1}}. \]
If they were equivalent then by the claim we would have 
\[ s_0^\gamma \cong s_0 \otimes \epsilon_l^a \]
as $\Gamma_l^0$-modules (for some $a\in\Z$). As $\gamma$ has finite order in $\Gamma_l/\Gamma_l^0$, but $\epsilon_l$ has infinite order,  we see that $a=0$, and so $\gamma \in \Gamma_l'$, a contradiction. The Proposition follows.
 \end{proof}

\subsection{Potential automorphy for weakly compatible systems.}\label{pa2}{$\mbox{}$} \newline

In this section we prove a potential automorphy theorem for weakly compatible systems of $l$-adic
representations of the absolute Galois group of a CM field. 

\begin{thm}\label{mtcs} Suppose that $F/F_0$ is a finite Galois extension of CM (resp.\ totally real)
  fields and that $\Favoid/F$ is a finite Galois extension. Suppose also for $i=1,\dots,r$ that  $(\CR_i,\CM_i)$ are totally odd, polarized weakly compatible systems of $l$-adic representations of $G_F$, with each $\CR_i$ regular and irreducible. Then there is a finite, CM (resp.\ totally real), extension $F'/F$ linearly disjoint from $\Favoid$ over $F$ and with $F'/F_0$ Galois, such that each $(\CR_i|_{G_{F'}},\CM_i|_{G_{(F')^+}})$ is automorphic. \end{thm}

\begin{proof}
The totally real case follows easily from the imaginary CM case by
Lemma \ref{bc}. (The only
thing to check is that we can find an imaginary
CM extension $F'/F$ such that each $\CR_i$ remains irreducible upon
restriction to $G_{F'}$. It will be enough to choose $F'$ linearly
disjoint from the fields $F^1_i/F$ obtained by applying part
(1) Lemma \ref{cc} to each $\CR_i$. Such a choice is possibly by
Lemma~\ref{l412}.) Thus we treat only the imaginary CM case.

 Replacing each
of the fields $M_i$ associated to $\CR_i$ by their compositum, we may
assume that $M_i=M$ is independent of $i$. Let $\CL_{i,1}$ be the Dirichlet density $1$ set of rational primes provided by applying Proposition \ref{big} to $\CR_i$, and let $\CL_{i,2}$ be the Dirichlet density $1$ set of rational primes above which $\CR$ is irreducible. Let $\CL=\bigcap_i \CL_{i,1} \cap \CL_{i,2}$. Then $\CL$ also has Dirchlet density $1$ (Lemma \ref{dd}). For $\lambda|l \in \CL$ we see that each $\barr_{i,\lambda}|_{G_{F(\zeta_l)}}$ is irreducible.

Removing finitely many primes from $\CL$ we may further suppose that 
\begin{itemize}
\item $l \in \CL$ implies $l\geq 2(\dim \CR_i+1)$ for each $i$,
\item $l\in \CL$ implies $l$ is unramified in $F$ and $l$ lies below none of the elements of the sets $S_i$ of bad primes for $\CR_i$,
\item if $\lambda|l \in \CL$ is a place of $M$, then all the Hodge--Tate numbers of $r_{i,\lambda}$ lie in a range of the form $[a,a+l-2]$.
\end{itemize}
We deduce, by Lemma \ref{locallift}, each $r_{i,\lambda}$ is potentially diagonalizable. 
Our theorem now follows by applying Theorem \ref{mtpm} to the $r_{i,\lambda}$ for any $\lambda|l \in \CL$.
\end{proof}

We state a simple special case separately.

\begin{cor}\label{mtcscor}Suppose that $F$ is a CM (resp.\ totally real)
  field and that  $(\CR,\CM)$ is a {\bf totally odd, polarized weakly compatible system} of $l$-adic representations of $G_F$, with $\CR$ {\bf regular and irreducible}. Then there is a finite, CM (resp.\ totally real), Galois extension $F'/F$, such that $(\CR|_{G_{F'}},\CM|_{G_{(F')^+}})$ is automorphic. \end{cor}

\begin{cor} Keep the assumptions of the last Corollary.
\begin{enumerate}
\item If $\imath:M \into \C$, then $L^S(\imath \CR,s)$ converges
  (uniformly absolutely on compact subsets) on some right half plane
  and has meromorphic continuation to the whole complex plane.
\item The compatible system $\CR$ is strictly pure. Moreover
\[ \Lambda(\imath \CR,s)=\epsilon(\imath \CR,s) \Lambda(\imath \CR^\vee, 1-s). \]
\item If $F$ is totally real, $n$ is odd, and $v|\infty$ then $\tr r_\lambda(c_v)=\pm 1$ and is independent of $\lambda$.
\end{enumerate} \end{cor}

\begin{proof}
The strict purity follows from Theorem \ref{mtcs}, Theorem \ref{grfaf} and the usual Brauer's theorem argument as in the last paragraph of the proof of Theorem \ref{incompsyst} below. 
The convergence and meromorphic continuation and functional equation of the L-function
follow from the theorem and a Brauer's theorem argument as in Theorem 4.2 of \cite{hsbt}. The last part generalizes an observation of F.\ Calegari \cite{frank}. The theorem reduces the question to the 
automorphic case where it is the main result of \cite{tsign}.
\end{proof}

As one example of the above theorem we state the following result.

\begin{cor}\label{product} Suppose that $\CK$ is a finite set of positive integers with the property that the $2^{\# \CK}$
partial sums of elements of $\CK$ are all distinct. For each $k \in \CK$ let $f_k$ be an elliptic modular newform of weight $k+1$ without complex multiplication and let $\pi_k$ be the corresponding
automorphic representation of $\GL_2(\A)$. Then there is a totally real Galois extension $F/\Q$ and
a regular algebraic, polarizable, cuspidal automorphic representation $\Pi$ of $\GL_{2^{\# \CK}}(\A_F)$ such that for all but finitely
many primes $v$ of $F$ we have
\[ \rec(\Pi_v|\det|_v^{(1-2^{\# \CK})/2}) = \left. \left(\bigotimes_{k \in \CK} \rec(\pi_{k,v|_\Q} |\det |_{v|_\Q}^{-1/2})\right) \right|_{W_{F_v}}. \]
In particular the `multiple product' $L$-function $L(\times_{k \in \CK} \pi_k,s)$ has meromorphic continuation to the whole complex plane. \end{cor}

\begin{proof} Let $M$ denote the compositum of the fields of coefficients of the $f_k$'s. Let $\lambda$
be any prime of $M$ and let $r_{k,\lambda}: G_\Q \ra \GL_2(\barM_\lambda)$ be the $\lambda$-adic representation associated to $f_k$. Because $f_k$ is not CM we know that $r_{k,\lambda}$ has Zariski dense image. We will apply Theorem \ref{mtcs} to the weakly compatible system 
\[ \bigotimes_\CK r_{k,\lambda}. \]
The only assumption that is perhaps not clear, is that this system is irreducible. So it only remains to check this property. The argument is a variant of Goursat's Lemma.

Now let $H$ denote the Zariski closure of $(\prod_\CK r_{k,\lambda})(G_\Q)$ in
$\GL_2(\barM_\lambda)^\CK$ and let $\barH$ denote its image in $\PGL_2(\barM_\lambda)^\CK$. 
Note that the projection of $\barH$ to each factor is surjective. As $\PGL_2$ is a simple algebraic group and all its automorphisms are inner, $\barH$ must be of the form $\PGL_2(\barM_\lambda)^\CI$ for
some set $\CI$. Moreover we can decompose $\CK=\coprod_{i \in \CI}
\CK_i$ and the mapping $\barH \ra \PGL_2(\barM_\lambda)^\CK$ is
conjugate to the mapping which sends the $i^{th}$ factor of 
$\PGL_2(\barM_\lambda)^{\CI}$ diagonally into $\prod_{i\in \CK_i} \PGL_2(\barM_\lambda)$. If for some
$i$ one had $\# \CK_i >1$ then we would have $r_{k,\lambda} = r_{k',\lambda} \otimes \chi$ for some
$k \neq k'$ in $\CK_i$ and some character $\chi$. We can conclude that $\chi$ is de Rham and then
looking at Hodge--Tate numbers gives a contradiction. Thus we must have $\barH = \PGL_2(\barM_\lambda)^\CK$ and $H \supset \SL_2(\barM_\lambda)^\CK$. 
As the tensor product representation of $\SL_2(\barM_\lambda)^\CK$ on 
$\barM_\lambda^{2^{\# \CK}}$ is irreducible we conclude that $\bigotimes_\CK r_{k,\lambda}$ is also irreducible, as desired. 
\end{proof}

We now turn to a proposition which will be useful in the next section. Its proof is essentially the same as the proof of Theorem \ref{mtcs}, once we have established the following Lemma.

\begin{lem}\label{dual} Suppose that $F$ is an imaginary CM 
  field, that $(\CR,\CM)$ is a polarized weakly compatible system of $l$-adic representations of $G_F$ defined over $M$, and that $\CR$ is pure and extremely regular. If $F'/F$ is a finite extension and if $s$ is a
subrepresentation of $r_\lambda|_{G_{F'}}$ for some prime $\lambda$
of $M$, then there is a CM field $F''$ with $F \subset F'' \subset F'$
such that the space of $s$ is invariant by $G_{F''}$. 
Moreover $(s,\mu_\lambda)$ is a polarized $l$-adic representation of $G_{F''}$. It is 
totally odd if $(\CR,\CM)$ is. 
\end{lem}

\begin{proof}
Let $F_1$ denote the normal closure of  $F'/F^+$. Let $\tau:F \into
\barM$ be an embedding with the property that if $H$ and $H'$ are
different subsets of $H_\tau$ of the same cardinality then $\sum_{h
  \in H} h \neq \sum_{h \in H'} h$. Choose an embedding $\tau_1:F_1
\into \barM$ extending $\tau$. Note that $\CR$ is pure of some weight
$w$.

If $s_1$ and $s_2$ are 
two $G_{F_1}$-submodules of some $r_{\lambda}$ of the same dimension we see that $s_1=s_2$ if and only if $\HT_{\tau_1}(s_1)=\HT_{\tau_1}(s_2)$ if and only if $\HT_{\tau_1}( \det s_1)=\HT_{\tau_1}(\det s_2)$. (The first equivalence by regularity and the second by extreme regularity. Note that in particular any irreducible submodule of $r_{\lambda}|_{G_{F_1}}$ has multiplicity $1$.)

For $\sigma \in \Gal(F_1/F^+)$ write $\HT_{\tau_1}((\det
s)^\sigma)=\HT_{\tau_1 \circ \sigma^{-1}}(\det s)=\{ h_\sigma \}$.  As
$\det s$ is de Rham and pure of weight $w\dim s$, we deduce that $h_\sigma + h_{\sigma c}=w \dim s $ for all $\sigma \in
\Gal(F_1/F^+)$ and all complex conjugations $c \in \Gal(F_1/F^+)$.
Thus if $c, c' \in \Gal(F_1/F^+)$ are complex conjugations then
$h_{\sigma c c'} = w \dim s- h_{\sigma c} = h_\sigma$ and so $s^{\sigma c c'} = s^\sigma$.

Let
$H \subset \Gal(F_1/F^+)$ be the normal subgroup generated by all
elements $cc'$ with $c, c' \in \Gal(F_1/F^+)$ complex conjugations.
The maximal CM subfield of $F_1$ is the maximal subfield on which all complex conjugations agree, i.e.\ $F_1^H$. Hence $F''=F_1^H \cap F'=(F_1)^{H\Gal(F_1/F')}$ is the maximal CM sub-field of $F'$.
Moreover if $\sigma \in H$ then $s^\sigma=s$ and so $s$ extends to a
representation of $G_{F''}$. 

If $c \in \Gal(F_1/F^+)$ is a
complex conjugation then 
\[ \HT_{\tau_1}(\det (\mu_\lambda
(s^\vee)^c))=\HT_{\tau_1 c}(\det (\mu_\lambda s^\vee))=\{ w \dim s -
h_c \}=\{h_1\}.\]
 As $\mu_\lambda (s^\vee)^c$ is also a constituent of
$r_\lambda$, we see that $s^c \cong \mu_\lambda s^\vee$ as
representations of $G_{F''}$. 
Let $v$ be an infinite place of $F$ and
$\langle\,\,\, ,\,\,\,\rangle_v$ a pairing on $\barM_\lambda^n$ as in
the definition of polarization for 
$(r_\lambda,\mu_\lambda)$. Then $\langle\,\,\, ,\,\,\,\rangle_v$ restricts to a perfect pairing on the space of $s$, as otherwise there would be a second irreducible constituent $s' \neq s$ of $r|_{G_{F''}}$ with $s' \cong \mu_\lambda (s^\vee)^c \cong s$, a contradiction. The Lemma follows.
\end{proof}

(We remark that we have not made use of the whole weakly compatible system, only of a single $l$-adic representation with the desired properties.)

\begin{prop}\label{redlem} Suppose that $F$ is a CM field; that $(\CR,\CM)$ is a totally odd, polarized weakly compatible system of $l$-adic representations. Suppose moreover that $\CR$ is  pure and extremely regular.
  Write
  $r_\lambda=r_{\lambda,1} \oplus \dots \oplus r_{\lambda,j_\lambda}$
  with each $r_{\lambda,\alpha}$ irreducible. Then there is a set of rational primes $\CL$ of Dirichlet density
  $1$ such that if $\lambda$ is a prime of $M$ lying above $l \in
  \CL$, then there is a finite, CM, Galois extension
  $F'/F$ such that each $(r_{\lambda,\alpha}|_{G_{F'}},\mu_\lambda|_{G_{(F')^+}})$ is irreducible and automorphic.
\end{prop}

\begin{proof}
By Lemma \ref{dual} we see that for all $\lambda$ and $\alpha$ the
pair $(r_{\lambda,\alpha},\mu_\lambda)$ is a totally odd polarized $l$-adic representation.

Let $\CL$ be the Dirichlet density $1$ set of rational primes obtained by applying Proposition \ref{big} to $\CR$. Then for $\lambda|l \in \CL$ we see that
$\barr_{\lambda,\alpha}|_{G_{F(\zeta_l)}}$ is irreducible. 
Removing finitely many primes from $\CL$ we may further suppose that 
\begin{itemize}
\item $l \in \CL$ implies $l\geq 2(\dim \CR+1)$,
\item $l\in \CL$ implies $l$ is unramified in $F$ and $l$ lies below no element of the set $S$ of bad primes for $\CR$,
\item if $\lambda|l \in \CL$ then all the Hodge--Tate numbers of $r_\lambda$ lie in a range of the form $[a,a+l-2]$.
\end{itemize}
We deduce, by Lemma \ref{locallift}, that $r_{\lambda,\alpha}$ is potentially diagonalizable. 
Our theorem now follows by applying Theorem \ref{mtpm} to $\{ r_{\lambda,\alpha}\}$ with $\Favoid$ equal to the compositum of the $\barF^{\ker \barr_{\lambda,\alpha}}$ for $\alpha=1,\dots,j_\lambda$. Then $\barr_{\lambda,\alpha}|_{G_{F'}}$ will be irreducible for $\alpha=1,\dots,j_\lambda$, and so $r_{\lambda,\alpha}|_{G_{F'}}$ will also be irreducible.
\end{proof}

\subsection{Irreducibilty results.}\label{ir}{$\mbox{}$} \newline

We will first recall some basic group theory. If $F$ is a number field and $l$ is a rational prime we will let $\REP_{F,l}$ denote the category of semi-simple, continuous representations of $G_F$ on finite dimensional $\barQQ_l$-vector spaces which ramify at only finitely many primes. If $U$, $V$ and $W$ are objects of $\REP_{F,l}$ with $U \oplus W \cong V \oplus W$ then $U \cong V$ (because they have the same traces). We will let $\GrG_{F,l}$ denote the Grothendieck group of $\REP_{F,l}$. If $V$ is an object of $\REP_{F,l}$ we will denote by $[V]$ its class in $\GrG_{F,l}$. We have the following functorialities.
\begin{enumerate}
\item The rule $[U][V]=[U \otimes V]$ makes $\GrG_{F,l}$ a commutative ring with $1$.

\item If $\sigma \in G_F$ then there is a ring homomorphism $\tr_\sigma:\GrG_{F,l} \ra \barQQ_l$ defined by $\tr_\sigma[V]=\tr \sigma|_V$. If $A \in \GrG_{F,l}$ then the function $\sigma \mapsto \tr_\sigma A$ is a continuous class function $G_F \ra \barQQ_l$. If $A,B \in \GrG_{F,l}$ and $\tr_\sigma A = \tr_\sigma B$ for all $\sigma \in G_F$ (or even for a dense set of $\sigma$) then $A=B$.

\item We will write $\dim$ for $\tr_1$. Then in fact $\dim:\GrG_{F,l} \ra \Z$ and $\dim [V] = \dim_{\barQQ_l} V$. 
\item\label{idim} There is a perfect symmetric $\Z$-valued pairing $(\,\,\, , \,\,\,)_{F,l}$ on $\GrG_{F,l}$ defined by
\[ ([U],[V])_{F,l} = \dim_{\barQQ_l} \Hom_{G_F}(U,V). \]
If $A=\sum_in_i[V_i]$ with the $V_i$ irreducible and distinct then $(A,A)_{F,l}=\sum_i n_i^2$. In particular if $A \in \GrG_{F,l}$ and $\dim A \geq 0$ and $(A,A)_{F,l}=1$ then
$A=[V]$ for some irreducible object $V$ of $\REP_{F,l}$. 

\item\label{iprod} Suppose that $G_1$ and $G_2$ are algebraic groups over $\barQQ_l$ and that $\theta:G_F \ra G_1(\barQQ_l) \times G_2(\barQQ_l)$ is a continuous homomorphism with Zariski dense image. Suppose also that $\rho_i$ and $\rho_i'$ are semi-simple algebraic representations of $G_i$ over $\barQQ_l$. Then
\[ \begin{array}{rcl} ([(\rho_1 \otimes \rho_2) \circ \theta], [(\rho_1' \otimes \rho_2') \circ \theta])_{F,l} &=& \dim \Hom_{G_1 \times G_2}(\rho_1 \otimes \rho_2,\rho_1'\otimes \rho_2')\\ &=&(\dim \Hom_{G_1} (\rho_1,\rho_1')) (\dim \Hom_{G_2}(\rho_2,\rho_2'))\\ &=& ([\rho_1 \circ \theta],[\rho_1' \circ \theta])_{F,l} ([\rho_2 \circ \theta],[\rho_2' \circ \theta])_{F,l}. \end{array}\] 

\item If $\sigma \in G_\Q$ then there is a ring isomorphism $\conju_\sigma$ from $\GrG_{F,l}$ to $\GrG_{\sigma^{-1}F,l}$ such that $\conju_\sigma[V]$ equals the class of the representation of $G_{\sigma^{-1}F}$ on $V$, under which $\tau$ acts by $\sigma \tau \sigma^{-1}$. It preserves dimension, and takes $(\,\,\,,\,\,\,)_{F,l}$ to $(\,\,\,,\,\,\,)_{\sigma^{-1}F,l}$. We have $\tr_\tau \conju_\sigma A =  \tr_{\sigma \tau \sigma^{-1}} A$. Also if $\sigma \in G_F$ then $\conju_\sigma$ is the identity on $\GrG_{F,l}$.

\item If $F'/F$ is a finite extension then the formula $\res_{F'/F} [V]=[V|_{G_{F'}}]$ defines a ring homomorphism $\res_{F'/F}:\GrG_{F,l} \ra \GrG_{F',l}$. Note that if $\sigma \in G_{F'}$ then $\tr_\sigma \res_{F'/F} A = \tr_\sigma A$ (so in particular $\dim \res_{F'/F} A = \dim A$). If $\sigma \in G_\Q$ then $\conju_\sigma \circ \res_{F'/F} = \res_{\sigma^{-1}F'/\sigma^{-1}F} \circ \conju_\sigma$. 

\item\label{ind} If $F'/F$ is a finite extension there is a $\Z$-linear map $\ind_{F'/F}:\GrG_{F',l} \ra \GrG_{F,l}$ defined by $\ind_{F'/F} [V]=[\Ind_{G_{F'}}^{G_F} V]$. Note the following.
\begin{enumerate}
\item $\tr_\sigma \ind_{F'/F} A = \sum_{\tau \in G_F/G_{F'}: \,\, \tau \sigma \tau^{-1} \in G_{F'}} \tr_{\tau \sigma \tau^{-1}} A$.
\item $\dim \ind_{F'/F} A = [F':F] \dim A$. 
\item $\ind_{F'/F}(A (\res_{F'/F} B)) = (\ind_{F'/F} A) B$.
\item $(\ind_{F'/F}A,B)_{F,l} = (A, \res_{F'/F} B)_{F',l}$. (By Frobenius reciprocity.)
\item If $F''/F$ is another finite extension then
\[ \res_{F''/F} \circ \ind_{F'/F} = \sum_{[\sigma]\in G_{F'}\backslash G_F / G_{F''}} \ind_{(\sigma^{-1}F').F''/F''} \circ \conju_\sigma \circ \res_{F'.(\sigma F'')/F'}. \]
(By Mackey's formula.)
\item\label{brauer} If $F'/F$ is a finite Galois extension then there
  is a finite collection of intermediate fields $F'/F'_i/F$ with $F'/F'_i$
  soluble together with characters $\psi_i: \Gal(F'/F_i') \ra \C^\times$ and integers $n_i$ such that 
\[ 1 = \sum_i n_i \ind_{F_i'/F}[\psi_i] \]
in the Grothendieck group of finite dimensional representation of the finite group $\Gal(F'/F)$ over $\C$.
(This is just Brauer's theorem for $\Gal(F'/F)$.) If $\imath:\barQQ_l \iso \C$ then applying $\imath^{-1}$ and multiplying by any $A\in \GrG_{F,l}$ we conclude that 
\[ A = \sum_i n_i \ind_{F_i'/F}( [\imath^{-1}\psi_i] \res_{F_i'/F} A). \]
Writing for any $i,j$
\[ G_F = \coprod_k G_{F_i'} \sigma_{ijk} G_{F_j'} \]
we see further that if 
\[  A = \sum_i n_i \ind_{F_i'/F}( [\imath^{-1}\psi_i]  B_i) \]
then 
\[ \begin{array}{l} (A,A)_{F,l}= \sum_{i,j,k} n_i n_j \\ ( (\conju_{\sigma_{ijk}} \circ \res_{F_i'.(\sigma_{ijk} F_j')/F_i'}) ([\imath^{-1}\psi_i] B_i), \res_{(\sigma_{ijk}^{-1}F_i').F_j'/F_j'}([\imath^{-1}\psi_j] B_j))_{(\sigma_{ijk}^{-1}F_i').F_j',l}.\end{array} \]
\end{enumerate} 

\item If $S$ is a finite set of primes of $F$ including all those above $l$ we will say that $A\in\GrG_{F,l}$ is unramified outside $S$ if we can write $A=\sum_i n_i [V_i]$ with each $V_i$ unramified outside $S$. 
In this case, we can define, for each $\imath:\barQQ_l \iso \C,$
\[ L^S(\imath A,s) = \prod_i L^S(\imath V_i,s)^{n_i} \]
at least as a formal Euler product, which will converge in some right
half plane if, for each $i$ the Weil--Deligne representation $\WD(V_i|_{G_{F,v}})$ is pure of weight $w_i$ for all but finitely many primes $v \not\in S$ of $F$.
This definition is independent of the choices and we have
\[ L^S(\imath(A+B),s)=L^S(\imath A,s)L^S(\imath B,s) \]
and
\[ L^S(\imath \ind_{F'/F} A,s) = L^{S'}(\imath A,s), \]
where $S'$ denotes the set of primes of $F'$ above $S$.
\end{enumerate}

Our first result is not really an irreducibility result, but it uses similar methods so we include it here.  It is a generalization of results of Dieulefait \cite{dieulefait2} in dimension $2$. The key ingredient is Theorem \ref{mtpm}.

\begin{thm}\label{incompsyst} Suppose that $F$ is a CM field, that $n$ is a positive integer and  that $l \geq 2(n+1)$ is a rational prime such that $\zeta_l \not\in F$.
Suppose that $(r,\mu)$ is an $n$-dimensional, {\bf totally odd, regular algebraic, polarized $l$-adic representation} of $G_F$. 
Suppose moreover that the following conditions are satisfied.
\begin{enumerate}
\item {\bf (Potential diagonalizability)}
  $r$ is potentially diagonalizable (and hence potentially crystalline) at each prime $v$ of $F$ above $l$.
\item {\bf (Irreducibility)} $\bar{r}|_{G_{F(\zeta_l)}}$ is irreducible.
\end{enumerate}

Then $r$ is part of a strictly pure compatible system of $l$-adic representations of $G_F$.
\end{thm}

\begin{proof}
Let $G$ denote the Zariski closure of the image of $r$, let $G^0$ denote the connected component of $G$ and let $F^0=\barF^{r^{-1} G^0(\barQQ_l)}$.
By Theorem \ref{mtpm} (or Corollary \ref{mtpmtotreal}) we can find a finite Galois CM
extension $F'/F$, which is totally real if $F$ is totally real and
which is linearly disjoint from $F^0\barF^{\ker
  \barr}$ over $F$, an isomorphism $\imath:\barQQ_l \iso \C$, and a
cuspidal, regular algebraic, polarized automorphic representation $(\pi,\chi)$ of $\GL_n(\A_{F'})$ such that 
\[ (r_{l,\imath}(\pi), r_{l,\imath}(\chi)\epsilon_l^{1-n})\cong (r|_{G_{F'}},\mu|_{G_{(F')^+}}). \]
For the rest of this proof we will fix $F'$. Note that $F'$ is automatically Galois over any intermediate field between $F'$ and $F$.
Suppose that $F' \supset F'' \supset F$ with $F'/F''$ soluble, then by
Lemma \ref{bc} there is a cuspidal, regular algebraic, polarized automorphic representation $(\pi^{(F'')},\chi^{(F'')})$ of $\GL_n(\A_{F''})$ such that 
\[ (r_{l,\imath}(\pi^{(F'')}),r_{l,\imath}(\chi)\epsilon_l^{1-n})\cong (r|_{G_{F''}},\mu|_{G_{(F'')^+}}). \]

Let $l'$ be a rational prime and let $\imath':\barQQ_{l'} \iso \C$. Note that if $\sigma \in G_F$ and if $F' \supset F'' \supset F''' \supset F$ with $F'/F'''$ soluble (in which case $F'/F''$ is also soluble) then $r_{l',\imath'}(\pi^{(F''')})|_{G_{F''}} \cong r_{l',\imath'}(\pi^{(F'')})$ and $r_{l',\imath'}(\pi^{(F'')})^\sigma \cong r_{l',\imath'}(\pi^{(\sigma^{-1}F'')})$. Let $G'$ (resp.\ $G^{(F'')} \supset G'$) denote the Zariski closure of $r_{l',\imath'}(\pi)$ (resp.\ $r_{l',\imath'}(\pi^{(F'')})$) and let $(G')^0$
(resp.\ $(G^{(F'')})^0$) denote the connected component of $G'$ (resp.\ $G^{(F'')}$). It follows from Lemma \ref{cc} that $G^{(F'')}/(G^{(F'')})^0$ and the map $G_{F''} \onto
G^{(F'')}/(G^{(F'')})^0$ is independent of $l'$. In the case $(l',\imath')=(l,\imath)$ we see, by the choice of $F'$, that $\Gal(F^0F''/F'') \iso G^{(F'')}/(G^{(F'')})^0$. Thus this is true for all $(l',\imath')$. We deduce that (for any $(l',\imath')$) we have $G^{(F'')}=G'$ and that the natural map
\[ G_{F''} \lra G'(\barQQ_{l'}) \times \Gal(F'/F'') \]
has Zariski dense image (where we consider $\Gal(F'/F'')$ as a finite algebraic group). If we decompose $r_{l',\imath'}(\pi)$ into irreducibles as
\[ r_{l',\imath'}(\pi)=  r_{l',\imath'}(\pi)_1 \oplus \dots \oplus  r_{l',\imath'}(\pi)_t \]
then this induces a unique decomposition
\[ r_{l',\imath'}(\pi^{(F'')})=  r_{l',\imath'}(\pi^{(F'')})_1 \oplus \dots \oplus  r_{l',\imath'}(\pi^{(F'')})_t \]
with $r_{l',\imath'}(\pi^{(F'')})_\alpha|_{G_{F'}} \cong r_{l',\imath'}(\pi)_\alpha$. (Note that by regularity the
$r_{l',\imath'}(\pi)_\alpha$ are pairwise non-isomorphic, as they will have different Hodge--Tate numbers.) We deduce that if $\sigma \in G_F$ and if $F' \supset F'' \supset F''' \supset F$ with $F'/F'''$ soluble (so that $F'/F''$ is also soluble) then $r_{l',\imath'}(\pi^{(F''')})_\alpha|_{G_{F''}} \cong r_{l',\imath'}(\pi^{(F'')})_\alpha$ and $r_{l',\imath'}(\pi^{(F'')})_\alpha^\sigma \cong r_{l',\imath'}(\pi^{(\sigma^{-1}F'')})_\alpha$. Moreover if $\rho_1$ and $\rho_2$ are representations of $\Gal(F'/F'')$ over $\barQQ_{l'}$ then
\[ ([r_{l'\imath'}(\pi^{(F'')})_\alpha][\rho_1],[r_{l'\imath'}(\pi^{(F'')})_\alpha][\rho_2])_{F'',l'}=([\rho_1],[\rho_2])_{F'',l'}. \]

Choose intermediate fields $F_i'$, characters $\psi_i$, integers
$n_i$, and elements $\sigma_{ijk}\in G_F$, as in item (\ref{brauer})
above. 
Write $F_{ijk}$ for $(\sigma_{ijk}^{-1}F_i').F_j'$. 
Then
\[ [r] = \sum_i n_i \ind_{F_i'/F} [r_{l,\imath}(\pi^{(F'_i)} \otimes (\psi_i \circ \Art_{F_i'} \circ \det))] \]
in $\GrG_{F,l}$. This motivates us to set 
\[ A_{l',\imath',\alpha} = \sum_i n_i \ind_{F_i'/F}( [r_{l',\imath'}(\pi^{(F_i')})_\alpha ][(\imath')^{-1} \circ \psi_i]) \in \GrG_{F,l'}.\]
Note that 
\[ \dim A_{l',\imath',\alpha}=\sum_i n_i [F_i':F] \dim r_{l',\imath'}(\pi)_\alpha = \dim r_{l',\imath'}(\pi)_\alpha. \]
Also note that 
\[ \begin{array}{rl} & (A_{l',\imath',\alpha},A_{l',\imath',\alpha})_{F,l'} \\ = & \sum_{i,j,k} n_i n_j  \\ &( [r_{l',\imath'}(\pi^{(F_{ijk})})_\alpha][ (\imath')^{-1} \psi_i^{\sigma_{ijk}}|_{G_{F_{ijk}}}] ,[r_{l',\imath'}(\pi^{(F_{ijk})})_\alpha] [(\imath')^{-1} \psi_j|_{G_{F_{ijk}}}])_{F_{ijk},l'} \\ =& \sum_{i,j,k} n_i n_j  ( [ (\imath')^{-1} \psi_i^{\sigma_{ijk}}|_{G_{F_{ijk}}}] ,[(\imath')^{-1} \psi_j|_{G_{F_{ijk}}}])_{F_{ijk},l'} \\ = & (1,1)_{F,l'} \\ =& 1
. \end{array} \]
(The first and third equality follow from item (\ref{brauer}) above, while the second equality follows from items (\ref{idim}) and (\ref{iprod}) above.)
Thus $A_{l',\imath',\alpha}=[r_{l',\imath',\alpha}]$ for some irreducible continuous representation $r_{l',\imath', \alpha}$ of $G_F$ on a $\barQQ_{l'}$-vector space of dimension $\dim r_{l',\imath'}(\pi)_\alpha$. 
Set
\[ r_{l',\imath'}= r_{l',\imath',1} \oplus \dots \oplus r_{l',\imath',t}. \]
We see that $r_{l,\imath}\cong r$ and that
\[ \tr r_{l',\imath'}(\sigma) = \sum_i n_i \sum_{\tau \in G_F/G_{F_i'}, \,\, \tau \sigma \tau^{-1} \in G_{F_i'}} 
((\imath')^{-1} \psi_i(\tau \sigma \tau^{-1})) \tr r_{l',\imath'}(\pi^{(F_i')})(\tau \sigma \tau^{-1}). \]

Let $v'$ denote a prime of $F'$ and set $v=v'|_F$. By Theorem \ref{grfaf}, if $v \ndiv l'$ then 
\[ \imath'\WD(r_{l',\imath'}|_{G_{F'_{v'}}})^{F-\semis} \cong \rec(\pi_{v'} \otimes |\det|_{v'}^{(1-n)/2})\]
 is pure.  Hence $\imath'\WD(r_{l',\imath'}|_{G_{F_{v}}})$ is also pure.
(See Lemma \ref{tylem}.) Moreover if $\sigma \in W_{F_v}$ then
\[ \begin{array}{l}  \hspace{4mm} \tr \imath'\WD(r_{l',\imath'}|_{G_{F_v}})(\sigma) \\ = \imath' \tr r_{l',\imath'}(\sigma)\\ = \sum_i n_i \sum_{\tau \in G_F/G_{F_i'}, \,\, \tau \sigma \tau^{-1} \in G_{F_i'}} 
 \psi_i(\tau \sigma \tau^{-1})\imath' \tr r_{l',\imath'}(\pi^{(F_i')})(\tau \sigma \tau^{-1})\\ = \sum_i n_i \sum_{\tau \in G_F/G_{F_i'}, \,\, \tau \sigma \tau^{-1} \in G_{F_i'}} 
 \psi_i(\tau \sigma
 \tau^{-1})\tr\rec(\pi^{(F_i')}_{(\tau v')|_{F_i'}}\otimes |\det|_{(\tau v')|_{F_i'}}^{(1-n)/2})(\tau \sigma \tau^{-1}). \end{array} \]
 (Again by Theorem \ref{grfaf}.)
 If $l''$ is another prime with $v\ndiv l''$ and if $\imath'':\barQQ_{l''} \iso \C$ then we conclude that
 \[ \imath'\WD(r_{l',\imath'}|_{G_{F_v}})^\semis \cong \imath''\WD(r_{l'',\imath''}|_{G_{F_v}})^\semis. \]
 As both are pure we conclude from Lemma \ref{tylem} that
 \[ \imath'\WD(r_{l',\imath'}|_{G_{F_v}})^{F-\semis} \cong \imath''\WD(r_{l'',\imath''}|_{G_{F_v}})^{F-\semis}. \]
 Thus the $r_{l',\imath'}$ form a strictly pure compatible system. 
\end{proof}

Recall that if $(\pi,\chi)$ is a cuspidal, regular algebraic, polarized automorphic representation of $\GL_n(\A_F)$ then  $(\{ r_{l,\imath}(\pi)\},\{r_{l,\imath}(\chi) \epsilon_l^{1-n}\})$ is a strictly pure polarized compatible system  of weight $w$. (See Theorem \ref{grfaf} and the discussion at the end of section \ref{cs}.) Then $|\chi|=||\,\,||_{F^+}^{n-1-w}$. If $\pi$ has central character $\chi_\pi$ then we see that $|\chi_\pi|=||\,\,||_F^{n(n-1-w)/2}$, and so $\pi \otimes ||\det||_F^{(w+1-n)/2}$ has a unitary central character and so is unitary. If $(\pi',\chi')$ is a cuspidal, regular algebraic, polarized automorphic representation of $\GL_{n'}(\A_F)$ and if $\{r_{l,\imath}(\pi')\}$ has weight $w'$ and if $S$ is a finite set of finite places of $F$ then
\[ \begin{array}{rl} & L^S(\pi \times (\pi')^\vee, s+(w-w'+n'-n)/2 ) \\ = & L^S((\pi ||\det||_F^{(w+1-n)/2}) \times (\pi' ||\det||_F^{(w'+1-n')/2})^\vee,s) \end{array} \]
is meromorphic and is holomorphic and non-zero at $s=1$ unless 
\[ \pi \cong \pi' ||\det||_F^{(w'-w+n-n')/2} \]
in which case it has a simple pole at $s=1$ (see \cite{shah} and \cite{MR623137}). 

\begin{thm} \label{irred} Suppose that $F$ is a CM field and that $\pi$ is a regular algebraic, polarizable, cuspidal
automorphic representation of $\GL_n(\A_F)$. If $\pi$ has extremely regular weight, then there is a set of rational primes $\CL$ of Dirichlet density $1$ such that if $l \in \CL$ and $\imath:\barQQ_l \iso \C$ then $r_{l,\imath}(\pi)$ is irreducible. \end{thm}

\begin{proof}
Let $\CL$ be the set of rational primes of Dirichlet density $1$
provided by Proposition \ref{redlem} applied to the compatible system
$\CR:=\{ r_{l,\imath}(\pi)\}$. 
Suppose $l\in \CL$ and $\imath:\barQQ_l \iso \C$. 
Let 
\[ r_{l,\imath}(\pi)= r_{l,\imath}(\pi)_1 \oplus \dots \oplus r_{l,\imath}(\pi)_j \]
be a decomposition into irreducibles. Let $F'/F$ and $\pi_\alpha$ for $\alpha=1,\dots,j$ be as in Proposition \ref{redlem} for $\{ r_{l',\imath'}(\pi)\}$ and $(l,\imath)$. Let $S$ denote the finite set of primes of $F$ which divide $l$ or above which $\pi$ ramifies or above which $F'$ ramifies.  Then
\[ \ord_{s=1} L^S(\imath (\CR \otimes \CR^\vee), s) = \ord_{s=1} L^S(\pi \times \pi^\vee,s) =-1.\] 
We will show that $\ord_{s=1} L^S(\imath (\CR \otimes \CR^\vee), s)$ also equals $-j$, and the theorem will follow.

Suppose that we are given an intermediate field $F' \supset F'' \supset F$ with $F'/F''$ soluble. By Lemma \ref{bc} there is a regular algebraic, polarizable, cuspidal
automorphic representation $\pi^{(F'')}_\alpha$ of
$\GL_{n_\alpha}(\A_{F''})$ such that
$r_{l,\imath}(\pi_\alpha^{(F'')})\cong
r_{l,\imath}(\pi)_\alpha|_{G_{F''}}$. Moreover
$r_{l,\imath}(\pi)_\alpha|_{G_{F''}}$ is irreducible. Let $n_\alpha
=\dim r_{l,\imath}(\pi)_{\alpha}$. If $\psi:\Gal(F'/F'') \ra \barQQ_l^\times$ is a character then 
\[ \begin{array}{rl} & \ord_{s=1} L^S(\imath r_{l,\imath}(\pi)_\alpha|_{G_{F''}} \otimes r_{l,\imath}(\pi)_\beta^\vee|_{G_{F''}} \otimes \psi, s) \\ = & \ord_{s=1} L^S(\pi_\alpha^{(F'')} \times (\pi_\beta^{(F'')})^\vee \times (\imath \circ \psi \circ \Art_{F''}), s+(n_\beta-n_\alpha)/2)\\ = &-\delta_{\alpha,\beta} \delta_{\psi,1}\\ =& -([r_{l,\imath}(\pi)_\alpha|_{G_{F''}}][\psi],[r_{l,\imath}(\pi)_\beta|_{G_{F''}}])_{F'',l},\end{array} \]
where $\delta_{\alpha,\beta}=1$ if $\alpha=\beta$ and equals $0$
otherwise, and where $\delta_{\psi,1}=1$ if $\psi=1$ and equals $0$
otherwise. [Note that the
$\pi_\gamma^{(F'')}||\det||_{F''}^{(1-n_\gamma)/2}$ have different
weights so that if $\pi_\gamma^{(F'')}||\det||_{F''}^{(1-n_\gamma)/2}
\cong \pi_{\gamma'}^{(F'')}||\det||^{(1-n_{\gamma'})/2} (\imath \circ
\psi \circ \Art_{F''} \circ \det)$ then $\gamma=\gamma'$. Moreover, we
would also have $r_{l,\imath}(\pi)_\gamma|_{G_{F''}} \cong r_{l,\imath}(\pi)_\gamma|_{G_{F''}}  \otimes \psi$, and so, as $r_{l,\imath}(\pi)_{\gamma}|_{G_{F'}}$ is irreducible, $\psi=1$. Similarly the $r_{l,\imath}(\pi)_\gamma|_{G_{F''}}$ have different Hodge--Tate numbers, so if $r_{l,\imath}(\pi)_\gamma|_{G_{F''}} \cong r_{l,\imath}(\pi)_{\gamma'}|_{G_{F''}} \otimes \psi$ then $\gamma=\gamma'$. Moreover as $r_{l,\imath}(\pi)_\gamma|_{G_{F'}}$ is irreducible we see that we also have $\psi=1$.] Thus
\[ \begin{array}{rl} & \ord_{s=1} L^S(\imath (r_{l,\imath}(\pi)|_{G_{F''}} \otimes r_{l,\imath}(\pi)|_{G_{F''}}^\vee \otimes \psi, s) \\ = & - (\res_{F''/F} [r_{l,\imath}(\pi)][\psi], \res_{F''/F} [r_{l,\imath}(\pi)])_{F'',l}. \end{array} \]

Now let $F_i'$, $n_i$ and $\psi_i$ be as in item (\ref{brauer}) of the list at the start of this section. Then
\[ L^S(\imath (\CR \otimes \CR^\vee), s)= \prod_i L^S(\imath (r_{l,\imath}(\pi)|_{G_{F'_i}} \otimes r_{l,\imath}(\pi)|_{G_{F'_i}}^\vee \otimes \psi_i, s)^{n_i} \]
and
\[ \begin{array}{rl} &\ord_{s=1} L^S(\imath (\CR \otimes \CR^\vee), s)\\ = &-\sum_{i} n_i ([\psi_i]\res_{F_i'/F}[r_{l,\imath}(\pi)],\res_{F_i'/F}[r_{l,\imath}(\pi)])_{F_i',l} \\ = &-\sum_{i} n_i (\ind_{F_i'/F}([\psi_i]\res_{F_i'/F}[r_{l,\imath}(\pi)]),[r_{l,\imath}(\pi)])_{F,l} \\ = &- ([r_{l,\imath}(\pi)],[r_{l,\imath}(\pi)])_{F,l} \\ = &-j, \end{array} \]
as desired.
\end{proof}

\begin{thm} Suppose that $F$ is a CM field and that
  $\CR$ is a pure, extremely regular, totally odd, polarizable weakly compatible system of $l$-adic representations of $G_F$.
Then we can write $\CR=\CR_1 \oplus \dots \oplus \CR_s$ where each $\CR_i$ is an irreducible, strictly pure, totally odd, polarizable compatible system of $l$-adic representations of $G_F$. \end{thm}

\begin{proof}
Choose a set $\CL$ of rational primes of Dirichlet density $1$ which
simultaneously works for Propositions \ref{big} and
\ref{redlem}. (See Lemma \ref{dd}.) Choose $\lambda|l \in \CL$ such that $l$
is unramified in $F$ and $l \geq 2(n+1)$ and $r_\lambda$ is
crystalline with Hodge--Tate numbers all in an interval of the form
$[a,a+l-2]$. Decompose $r_\lambda$ into irreducible subrepresentations
\[ r_\lambda = r_{\lambda,1} \oplus \dots \oplus r_{\lambda,j_\lambda}. \]
By Theorem \ref{incompsyst} each $r_{\lambda,\alpha}$ is part of a strictly pure compatible system $\CR_\alpha$. Let $F'/F$ and $\pi_\alpha$ for $\alpha=1,\dots,j_\alpha$ be as in Proposition \ref{redlem} for $\CR$ and $\lambda$. Then $\CR_\alpha|_{G_{F'}}$ is the compatible system associated to $\pi_\alpha$. Moreover $\pi_\alpha$ is extremely regular. By Theorem \ref{irred} there is a set $\CL_\alpha$ of rational primes of Dirichlet density $1$ such that if $\lambda'|l' \in \CL_\alpha$ then $r_{\alpha,\lambda'}|_{G_{F'}}$ is irreducible. Thus $\CR_\alpha$ is irreducible.
\end{proof}
\newpage

\appendix
\section{}

The results recorded in this appendix make no reference to results proved elsewhere in this paper.

\subsection{Some algebra}{$\mbox{}$} \newline

The results of this section are probably well known to experts. However as we couldn't find references we include proofs. We start with some commutative algebra.

Suppose that $L$ is a finite extension of $\Q_l$ and let $|\,\,\,|_L$
denote the $l$-adic norm on $L$ normalized by
$|l|_L=l^{-[L:\Q_l]}$. We will denote the $l$-adic completion of the polynomial ring $\CO_L[s_1,\dots,s_r]$ by $\CO_L\langle s_1,\dots,s_r\rangle$, and we will let $L\langle s_1,\dots,s_r\rangle=\CO_L\langle s_1,\dots,s_r\rangle[1/l]$ denote the Tate algebra. The Gauss norm $|f|_L$ of an element $f \in L\langle s_1,\dots,s_r\rangle$ is defined to be the maximum of the $|\,\,\,|_L$ norm of any coefficient of a monomial in the $s_i$'s. The Gauss norm is multiplicative (\cite[5.1.2]{bgr}) and $L\langle s_1,\dots,s_r\rangle$ is a UFD (\cite[5.2.6]{bgr}). From these two facts one can deduce without difficulty that $\CO_L\langle s_1,\dots,s_r\rangle$ is also a UFD. (The units are the units in $L\langle s_1,\dots,s_r\rangle$ with Gauss norm $1$; and the irreducibles are the irreducibles in $L\langle s_1,\dots,s_r\rangle$ with Gauss norm $1$ together with a uniformizer in $\CO_L$ times any unit in $\CO_L\langle s_1,\dots,s_r \rangle$.) The Tate algebra $L\langle s_1,\dots,s_r\rangle$ is noetherian (\cite[5.2.6]{bgr}); all its ideals are closed (\cite[5.2.7]{bgr}); and a prime ideal $\wp$ of $L\langle s_1,\dots,s_r \rangle$ is maximal if and only if $L\langle s_1,\dots,s_r\rangle/\wp$ is a finite extension of $L$ (\cite[6.1.2]{bgr}). 

By an affinoid algebra we shall mean a quotient of a Tate algebra by some ideal. Maximal ideals are dense in the spectrum of an affinoid algebra (\cite[6.1.1]{bgr}). The completed tensor product of two affinoid algebras is again an affinoid algebra (\cite[6.1.1]{bgr}). We will call an affinoid algebra $A$ {\em geometrically connected} if $\Spec A \otimes_L L'$ is connected for all finite extensions $L'/L$. It suffices to check this for a cofinal collection of finite extensions $L'/L$ (because, for any field $L'$, any $L'$-algebra $B$ and any finite extension $L''/L$, if $\Spec B \otimes_{L'} L''$ is connected then so is $\Spec B$). 

\begin{lem}\label{prod} Suppose that $A_1$ and $A_2$ are two geometrically connected affinoid algebras over $L$, then $A_1 \hatotimes_L A_2$ is also geometrically connected. \end{lem}

\begin{proof}
Suppose that 
\[ \Spec A_1 \hatotimes_L A_2 \otimes_L L' = U^{(1)} \coprod U^{(2)} \]
is a decomposition into two non-empty open subsets, for some finite
extension $L'/L$. We will derive a
contradiction. Each $U^{(j)}$ contains a maximal ideal and after
replacing $L'$ by a finite extension we may suppose that each $U^{(j)}$ contains a maximal ideal $\gm^{(j)}$ with residue field $L'$. Let $\gm_j$ denote the contraction of $\gm^{(j)}$ into $A_j\otimes_L L'$. Then $\gm_j$ is a prime ideal and $(A_j \otimes_L L')/\gm_j =L'$. Thus $\gm_j$ is maximal. We have an isomorphism
\[ (A_1 \hatotimes_L A_2 \otimes_L L')/\gm_j \cong A_{j'} \otimes_L L' \]
where $j'\neq j$. Thus the fibre $\Spec (A_1 \hatotimes_L A_2 \otimes_L L')/\gm_j$ is connected and so 
\[ \Spec (A_1 \hatotimes_L A_2 \otimes_L L')/\gm_j \subset U^{(j)}. \]
However 
\[ (A_1 \hatotimes_L A_2 \otimes_L L')/( \gm_1 , \gm_2) \cong (A_1 \otimes L')/\gm_1 \otimes_{L'} (A_2 \otimes L')/\gm_2 \cong L', \]
and so we see $( \gm_1,\gm_2) $ is a maximal ideal of $A_1 \hatotimes_L A_2 \otimes_L L'$ lying in $U^{(1)} \cap U^{(2)}$, a contradiction.
\end{proof}

\begin{lem}\label{ID1} Suppose that $f(t) \in \CO_L[t]$ and $f(0) \in \CO_L^\times$. Then 
\[ A=\CO_L\langle s,t\rangle/(sf(t)-1) \]
is an integral domain complete in the $l$-adic topology. 
\end{lem}

\begin{proof}
As $\CO_L\langle s,t\rangle$ is a UFD, it suffices to prove that $sf(t)-1$ is irreducible. Suppose $sf(t)-1=gh$ in $\CO_L\langle s,t\rangle$. Reducing modulo $\lambda$ and using the fact that $sf(t)-1$ is irreducible in $(\CO_L/\lambda)[s,t]$, we see that one of $g$ and $h$ reduces to a constant and hence is itself a unit. The Lemma follows.
\end{proof}

\begin{cor}\label{corID} If $\alpha \in \CO_L-\{0\}$ then $\CO_L\langle s,t\rangle/(st-\alpha)$ is an integral domain. \end{cor}

\begin{proof}
Consider the map
\[ \begin{array}{rcl} \theta: \CO_L\langle s,t\rangle/(st-\alpha) & \lra & \CO_L\langle s',t\rangle/(s't-1) \\
h(s,t) & \longmapsto & h(\alpha s',t). \end{array} \]
By Lemma \ref{ID1} it suffices to show that $\theta$ is injective. 

If $\beta \in \CO_L-\{0\}$ then any element of $\CO_L\langle s,t\rangle/(st-\beta)$ has a representative of the form
\[ \sum_{i=0}^\infty a_i t^i + \sum_{i=1}^\infty b_i s^i .\]
Moreover we claim this representative is unique. Indeed if
\[ \sum_{i=0}^\infty a_i t^i + \sum_{i=1}^\infty b_i s^i  = (st-\beta)\sum_{i,j=0}^\infty c_{i,j} s^it^j \]
then
\[ c_{i+1,j+1}=\beta^{-1} c_{i,j} \]
for all $i,j\in \Z_{\geq 0}$. If $\sum_{i,j=0}^\infty c_{i,j} s^it^j  \in \CO_L\langle s,t\rangle$ this forces $c_{i,j}=0$ for all $i,j \in \Z_{\geq 0}$, which proves the claim.

Applying these observations for $\beta=\alpha$ and for $\beta=1$, the injectivity of $\theta$ follows. 
\end{proof}

\begin{lem}\label{ID} Suppose that $\alpha_1,\dots,\alpha_r \in \CO-\{0\}$. Then 
\[ A=\CO_L\langle s_1,t_1,s_2,\dots,t_r\rangle /(s_1t_1-\alpha_1,\dots,s_rt_r-\alpha_r)\]
 is an integral domain complete in the $l$-adic topology. \end{lem}

\begin{proof}
By the exactness of completion (and noetherianness of the polynomial ring $\CO_L[s_1,t_1,s_2,\dots,t_r]$) we see that $A$ is the $l$-adic completion of
\[ A^0=\CO_L[ s_1,t_1,s_2,\dots,t_r] /(s_1t_1-\alpha_1,\dots,s_rt_r-\alpha_r). \]
The ring $A^0$ is noetherian and flat over $\CO_L$ (because it is free over $\CO_L$ with basis $\{ \prod s_i^{a_i}t_i^{b_i}:\,\, a_ib_i=0\}$) and so $A$ is also flat over $\CO_L$. Thus it suffices to show that
\[ A[1/l] = L\langle  s_1,t_1,s_2,\dots,t_r\rangle /(s_1t_1-\alpha_1,\dots,s_rt_r-\alpha_r) \]
is a domain. 

The ring $A^0$ is also excellent (being a finitely generated $\CO_L$-algebra), Cohen--Macaulay (being a complete intersection in a polynomial ring over $\CO_L$) and normal (being Cohen--Macaulay and regular in codimension $1$). Hence $A$ is normal (\cite[33.I]{mat}), and so $A[1/l]$ is also normal. Thus it suffices to show that
$A[1/l]$ is geometrically connected (or even just connected).

In the case $r=1$ we see that, for any finite extension $L'/L$, the ring 
\[ L'\langle s_1,t_1\rangle/(s_1t_1-\alpha) = (\CO_{L'}\langle s_1,t_1\rangle/(s_1t_1-\alpha))[1/l]\]
 is a domain by Corollary \ref{corID}. Hence $L\langle s_1,t_1 \rangle/(s_1t_1-\alpha_1)$ is geometrically connected. 

In general we have 
\[ A[1/l] \cong L\langle s_1,t_1 \rangle/(s_1t_1-\alpha_1) \hatotimes \cdots \hatotimes L\langle s_r,t_r \rangle/(s_rt_r-\alpha_r), \]
and so the case $r=1$ and Lemma \ref{prod} imply that $A[1/l]$ is geometrically connected.
\end{proof}

Next we turn to representation theory.

\begin{lem}\label{repth} Let $\Gamma$ be a group and $M$ a field of characteristic $0$. Also let 
\[ r: \Gamma \lra \GL_n(\barM) \]
be a semi-simple representation. Suppose that for all $\gamma \in \Gamma$ we have $\tr r(\gamma) \in M$, and that for some $\gamma \in \Gamma$ the characteristic polynomial of $r(\gamma)$ has distinct, $M$-rational roots. Then $r$ is conjugate to a representation into $\GL_n(M)$. \end{lem}

\begin{proof} Let $B$ denote the $M$-span of the image of $r$ in $M_{n \times n}(\barM)$.
Note that the $\barM$-span $B_{\barM}$ of $B$ is a finite dimensional, semi-simple $\barM$-algebra. Let $e_1,\dots,e_r$ be an $\barM$-basis of $B_{\barM}$ consisting of elements of $B$ and let $e_1',\dots,e_r'$ be the dual basis for the trace pairing. Then $B \supset \sum_i Me_i$. If $b \in B$ we have $\tr(be_i) \in M$ for all $i$ and so $B \subset \sum_i Me_i'$. Thus $B=\sum M e_i = \sum M e_i'$ and
\[ B \otimes_M \barM \liso B_{\barM} \subset M_{n \times n}(\barM). \]
In particular $B$ is a finite dimensional semi-simple $M$-algebra and
$\barM^n$ is a faithful $B \otimes_M \barM$-module. 

We have 
\[ B \cong \bigoplus_j M_{m_j}(D_j) \]
where each $D_j$ is a division algebra with centre a finite extension
$Z(D_j)$ of $M$. Let $r_j^2 = \dim_{Z(D_j)}D_j$. As a representation
of $B\otimes_m\barM \cong \bigoplus_{j,\tau} M_{m_j
  r_j}(\barM)$, where $\tau$ runs over $M$-embeddings $Z(D_j)\into
\barM$ for each $j$,  we have
\[ \barM^n \cong \bigoplus_{j,\tau}\barM^{m_j r_j n_{j,\tau}} \]
for some non-negative integers $n_{j,\tau}$. Since $\tr(B)\subset M$,
we see that $n_{j,\tau}=n_j$ is independent of $\tau$. The existence
of an element of $B$ with $n$-distinct $M$-rational eigenvalues
implies that $r_j= n_j= [Z(D_j):M] = 1$ for all $j$. Thus $D_j=M$ for
all $j$, and the lemma follows.
\end{proof}

Next we have a result about unramified tori.

\begin{lem}\label{nrtorus} Suppose that $\tT_1$ and $\tT_2$ are unramified tori over $\Z_l$ and that $\phi:\tT_1 \ra \tT_2$ is a surjection.  Then the cokernel of $\tT_1(\Z_l) \ra \tT_2(\Z_l)$ is finite of order dividing the order of the torsion subgroup of $X^*(\tT_1)/\phi^*X^*(\tT_2)$. \end{lem}

\begin{proof}
We have an exact (in the centre) sequence
\[ \tT_1(\Z_l) \stackrel{\phi}{\lra} \tT_2(\Z_l) \lra \Hom(X^*(\tT_1)/\phi^*X^*(\tT_2),(\hatZZ_l^\nr)^\times)/(\Frob_l-1). \]
(Recall that $(\hatZZ_l^\nr)^{\Frob_l^r=1}=\Z_{l^r}$.)
There is a positive integer $m$ so that $\Frob_l^m$ acts trivially on $X^*(\tT_1)$. There is also a surjection
\[ \begin{array}{r} \Hom(X^*(\tT_1)/\phi^*X^*(\tT_2),(\hatZZ_l^\nr)^\times)/(\Frob_l^m-1) \onto \\ \Hom(X^*(\tT_1)/\phi^*X^*(\tT_2),(\hatZZ_l^\nr)^\times)/(\Frob_l-1). \end{array} \]
Thus it suffices to show that
\[ \Hom(X^*(\tT_1)/\phi^*X^*(\tT_2),(\hatZZ_l^\nr)^\times)/(\Frob_l^m-1) \]
is finite with order dividing the order of the torsion subgroup of $X^*(\tT_1)/\phi^*X^*(\tT_2)$.

There is also an exact (in the centre) sequence
\[  \begin{array}{r} \Hom((X^*(\tT_1)/\phi^*X^*(\tT_2))^\tf,(\hatZZ_l^\nr)^\times) \lra \Hom(X^*(\tT_1)/\phi^*X^*(\tT_2),(\hatZZ_l^\nr)^\times) \lra \\ \Hom((X^*(\tT_1)/\phi^*X^*(\tT_2))^\tor,(\hatZZ_l^\nr)^\times) \end{array} \]
and hence an exact (in the centre) sequence
\[  \begin{array}{l} \Hom((X^*(\tT_1)/\phi^*X^*(\tT_2))^\tf,(\hatZZ_l^\nr)^\times) \lra  \\\Hom(X^*(\tT_1)/\phi^*X^*(\tT_2),(\hatZZ_l^\nr)^\times)/(\Frob_l^m-1)\\ \lra  \Hom((X^*(\tT_1)/\phi^*X^*(\tT_2))^\tor,(\hatZZ_l^\nr)^\times)/B, \end{array}\]
where $B$ denotes the image of $(\Frob_l^m-1) \Hom(X^*(\tT_1)/\phi^*X^*(\tT_2),(\hatZZ_l^\nr)^\times)$ in the group $\Hom((X^*(\tT_1)/\phi^*X^*(\tT_2))^\tor,(\hatZZ_l^\nr)^\times)$.
As 
\[ \Hom((X^*(\tT_1)/\phi^*X^*(\tT_2))^\tor,(\hatZZ_l^\nr)^\times)\]
 is finite with order the prime-to-$l$ part of $\# (X^*(\tT_1)/\phi^*X^*(\tT_2))^\tor$, it suffices to show that
\[  \Hom((X^*(\tT_1)/\phi^*X^*(\tT_2))^\tf,(\hatZZ_l^\nr)^\times)/(\Frob_l^m-1) =(0). \]
However this follows from the fact that $\Frob_l^m-1$ is surjective on $(\hatZZ_l^\nr)^\times$ (which is proved recursively modulo higher and higher powers of $l$). 
\end{proof}

Finally we recall the following observation, which isn't really algebraic.

\begin{lem}\label{dd} The intersection of a finite number of sets of rational primes of Dirichlet density $1$ has Dirichlet density $1$. \end{lem}

\begin{proof} A set of rational primes has Dirichlet density $1$ if and only if its complement has Dirichlet density $0$. Thus the Lemma follows from the fact that the union of a finite number of sets of rational primes of Dirichlet density $0$ has Dirichlet density $0$.
\end{proof}

\subsection{Building fields and characters}\label{sbuild}{$\mbox{}$} \newline

For the convenience of the reader we recall some results about the construction of fields and characters, but we will start by recalling some facts about algebraic characters. 

Suppose that $F$ is a number field, $l$ is a rational prime and $\imath:\barQQ_l \iso \C$. Let $F_0$ denote the maximal CM subfield of $F$. Suppose also that $\chi$ is an algebraic character of $\A_F^\times/F^\times$ with
\[ \chi|_{(F_\infty^\times)^0}: x \longmapsto \prod_{\tau \in \Hom(F,\C)} (\tau x)^{-a_\tau}. \]
Then
 \[\imath\left((r_{l,\imath}(\chi)\circ\Art_F)(x)\prod_{\tau\in\Hom(F,\C)}(\imath^{-1}\tau)(x_l)^{a_\tau}\right)=\chi(x)\prod_{\tau\in\Hom(F,\C)}(\tau x)^{a_\tau}\]
for $x\in \A_F^\times$. Moreover $r_{l,\imath}(\chi)$ is de Rham at all primes above $l$ and, if $\tau:F \into \barQQ_l$ then $\HT_\tau(r_{l,\imath}(\chi))=\{ a_{\imath \circ \tau}\}$. (See \cite{MR0263823}.)
Moreover we have that 
\begin{enumerate}
\item $a_\tau+a_{\tau'} = \wt(\chi)$ for all $\tau,\tau'\in \Hom(F,\C)$ with
$\tau|_{F_0}=\tau'|_{F_0} \circ c$;
\item $a_\tau$ only depends on $\tau|_{F_0}$;
\item $\wt(\chi_1\chi_2)=\wt(\chi_1)+\wt(\chi_2)$;
\item if $\sigma$ is an automorphism of $F$ then $\wt(\chi^\sigma)=\wt(\chi)$ (as $\sigma|_{F_0}$ commutes with $c$);
\item $\wt(\chi \circ \norm_{K'/K}) = \wt(\chi)$;
\item if $F_0$ is totally real then $\wt(\chi)$ is even;
\item $\wt(||\,\,\,||_F)=-2$;
\item for all places $v$ of $F$ the Weil-Deligne representation $\WD(r_{l,\imath}(\chi))|_{G_{F_v}}$ is pure of weight $\wt(\chi)$.
\end{enumerate}
Any algebraic $l$-adic character of $G_F$ arises in this way.

Next we recall Lemma 4.1.2 of \cite{cht}.

\begin{lem}\label{l412} Suppose that $F$ is a number field, that
  $\Favoid/F$ is a finite Galois extension and that $S$ is a finite set of places of F. For $v \in S$ let $E_v
/F_v$ be a finite
Galois extension. Then we can find a finite, soluble Galois extension $E/F$
linearly disjoint from $\Favoid$ such that for each $v \in S$ and each prime $w$ of $E$
above $v$, the extension $E_w/F_v$ is isomorphic to $E_v/F_v$.
\end{lem}

In a somewhat similar vein we have the following result.

\begin{lem} Suppose that $F$ is a number field, that $S$ is a finite set of places of  $F$, that $\Favoid/F$ is a finite Galois extension and that $N$ is a positive integer.  Then we can find a finite cyclic extension $E/F$ of degree $N$ in which all the elements of $S$ split completely and which is linearly disjoint from $\Favoid$ over $F$. 
\end{lem}

\begin{proof} Let $\{ \Favoid_i \}$ denote the intermediate fields between $\Favoid$ and $F$ for which $\Gal(\Favoid_i/F)$ is simple. Augment $S$ to include, for each $i$, a prime which does not split in $\Favoid_i$. Then the linear disjointness of $E$ from $\Favoid$ will be automatic. Choose a prime $v_0$ of $F$ with $v_0 \not\in S$. Now use Lemma 4.1.1 of \cite{cht} to choose a finite order character
\[ \chi: \A_F^\times/F^\times \lra \barQQ^\times \]
such that
\begin{itemize}
\item $\chi|_{\prod_{v \in S}F_v^\times} =1$,
\item and $\chi|_{F_{v_0}^\times}$ has order $N$.
\end{itemize}
Then $\barF^{\ker \chi \circ \Art_F}/F$ is a cyclic extension of degree divisible by $N$ in which all the primes in $S$ split completely. The unique sub-extension $E/F$ of degree $N$ satisfies the requirements of the lemma.
\end{proof}

\begin{cor}\label{cycCM} Suppose that $F$ is an imaginary CM field, that $S$ is a finite set of places of  $F$, that $\Favoid/F$ is a finite Galois extension and that $N$ is a positive integer.  Then we can find a cyclic CM extension $E/F$ of degree $N$ in which all the elements of $S$ split completely and which is linearly disjoint from $\Favoid$ over $F$. 
\end{cor}

\begin{proof} Let $F^+$ denote the maximal totally real subfield of
  $F$. By the Lemma we can find a totally real cyclic extension
  $E^+/F^+$ of degree $N$ in which all the primes of $F^+$ below $S$ split completely and which is linearly disjoint from the normal closure of $\Favoid/F^+$ over $F^+$. Then we can take $E=E^+F$.
\end{proof}

We now turn to building characters, and record three results, all of which have a similar feel but which are slightly different. We start by restating (a special case of) Lemma 2.2 of \cite{hsbt}.

\begin{lem}\label{l22} Suppose that $F$ is an imaginary CM field with maximal totally real subfield $F^+$ and that $S$ is a finite set of primes of $F$. Let
\[ \chi: (\A_{F^+}^\infty)^\times \lra \barQQ^\times \]
and
\[ \psi_S: \CO_{F,S}^\times \lra \barQQ^\times \]
be continuous characters such that
\[ \psi_S|_{(\A_{F^+}^\infty)^\times \cap \CO_{F,S}^\times}=\chi|_{(\A_{F^+}^\infty)^\times \cap \CO_{F,S}^\times}. \]
Also let
\[ \phi_0: F^\times \lra \barQQ^\times \]
be a character such that 
\[ \phi_0|_{(F^+)^\times}=\chi|_{(F^+)^\times}. \]

Then there is a continuous character
\[ \phi: \A_F^\times \lra \barQQ^\times \]
such that
\begin{itemize}
\item $\phi|_{F^\times}=\phi_0$;
\item $\phi|_{(\A^\infty_{F^+})^\times}=\chi$;
\item and $\phi|_{\CO_{F,S}^\times}=\psi_S$.
\end{itemize}
\end{lem}

When we invoke this Lemma we may not specify $S$ or the character $\psi_S$, but instead we may list local conditions to be satisfied by $\phi$ at a finite number of primes. We hope that the reader will have no difficulty in finding a set $S$ and a character $\psi_S$, so that if we apply the Lemma with these choices it produces a character $\phi$ with the desired local properties.

We next record two similar results about algebraic characters of $G_F$.

\begin{lem} \label{l416}
Suppose that $l$ is a rational prime; that $F$ is an imaginary CM field with maximal totally real subfield $F^+$, and that $S$ is a finite set of primes of $F$ containing all primes above $l$ and satisfying $S^c=S$. Let
\[ \chi:G_{F^+} \lra \barQQ_l^\times \]
be a continuous character; and, for $v \in S$, let
\[ \psi_v:G_{F_v} \lra \barQQ_l^\times \]
be a continuous character such that
\[ (\psi_v \psi^c_{v^c})|_{I_{F_v}} = \chi|_{I_{F_v}}. \]
For $v|l$ suppose that the character $\psi_v$ is de Rham.
\begin{enumerate}
\item Suppose further that every element of $S$ is unramified over $F^+$, and that $\chi(c_v)$ is independent of $v|\infty$. Then 
there is a continuous character 
\[ \theta:G_F \lra \barQQ_l^\times \]
such that
\[ \theta \theta^c = \chi|_{G_F} \]
and, for all $v \in S$, 
\[ \theta|_{I_{F_v}}=\psi_v|_{I_{F_v}}. \]

\item Alternatively, suppose further that $l>2$ and that
\[ \bartheta: G_F \lra \barFF_l^\times \]
is a continuous character such that
\begin{itemize}
\item $\bartheta\bartheta^c$ equal to the reduction of $\chi|_{G_F}$,
\item and for $v \in S$ the restriction $\bartheta|_{G_{F_v}}$ equals the reduction of $\psi_v$.
\end{itemize} 
Then there is a continuous character 
\[ \theta:G_F \lra \barQQ_l^\times \]
lifting $\bartheta$ and such that
\[ \theta \theta^c = \chi|_{G_F} \]
and, for all $v \in S$, 
\[ \theta|_{I_{F_v}}=\psi_v|_{I_{F_v}}. \]
\end{enumerate}
\end{lem}

\pfbegin We deduce the first part from Lemma \ref{l22}. Replacing $\chi$ by $\chi \delta_{F/F^+}$ if need be, we may
suppose that $\chi(c_v)=1$ for all $v|\infty$. Note that $\chi$ is
algebraic, and so, because $F^+$ is totally real, all its Hodge--Tate
numbers are equal to some integer $w$. By class field theory we may think of $\chi:\A_{F^+}^\times \ra \barQQ_l^\times$ and $\phi_v:F_v^\times \ra \barQQ_l^\times$ and look for $\theta:\A_F^\times/F^\times \ra \barQQ_l^\times$. If $v|l$ then
\[ \psi_v = \prod_{\tau:F_v \into \barQQ_l} \tau^{-m_\tau} \]
on a some non-empty open subgroup of $F_v^\times$. Moreover $m_\tau+m_{\tau \circ c}=w$ for all $\tau$. 

Let $\chi':\A_{F^+}^\times \ra \barQQ_l^\times$ be defined by
\[ \chi'(\alpha)=\chi(\alpha) (\norm_{F^+/\Q} \alpha_l)^w. \]
Then $\chi'$ has open kernel containing $(F^+_\infty)^\times$. 
Let 
\[ \phi_0=\prod_{\tau:F \into \barQQ_l} \tau^{m_\tau}:F^\times \ra \barQQ_l^\times. \]
We see that
\[ \phi_0|_{(F^+)^\times} = \chi'|_{(F^+)^\times}. \]
Thus in particular $\chi'$ is actually valued in the algebraic closure $\barQQ$ of $\Q$ in $\barQQ_l$. Also define 
\[ \psi_v':F_v^\times \lra \barQQ_l^\times \]
to be $\psi_v$ if $v\ndiv l$ and
\[ \psi_v \prod_{\tau:F_v \into \barQQ_l} \tau^{m_\tau} \]if $v|l$.
In either case $\psi_v'$ has open kernel and is valued in $\barQQ$.

Now apply Lemma \ref{l22} to $\chi'$, $\phi_0$ and $\prod_{v \in S} \psi_v'$. (Note that the norm map $\norm_{F/F^+}$ from $\prod_{v \in S} \cO_{F,v}$ to the intersection of this group with $\A_{F^+}^\times$ is surjective.) We get a  character
\[ \phi: \A_F^\times \lra \barQQ_l^\times \]
with open kernel such that
\begin{itemize}
\item $\phi|_{F^\times}=\phi_0$;
\item $\phi|_{(\A^\infty_{F^+})^\times}=\chi'$;
\item and $\phi|_{\CO_{F,v}^\times}=\psi_v'$ for all $v \in S$.
\end{itemize}
The character
\[ \theta: \A_F^\times \lra \barQQ_l^\times \]
defined by
\[ \theta(\alpha) = \phi(\alpha) \prod_{\tau:F \into \barQQ_l} (\tau \alpha_l)^{-m_\tau} \]
satisfies the requirements of the first part of the present lemma.

The second part follows easily from Lemma 4.1.6 of \cite{cht}. If $\tau:F \into \barQQ_l$ lies above $v \in S$ then we will let $m_\tau$ denote the $\tau$ Hodge-Tate number of $\psi_v$. We need only verify that $m_\tau+m_{\tau c}$ is independent of $\tau$. However the character $\chi$ must be algebraic and for each $\tau:F \into \barQQ_l$ the $\tau$ Hodge-Tate number of $\chi$ is $m_\tau+m_{\tau c}$. The result follows from the facts recalled at the start of this section.
\pfend

Probably the assumption $l>2$ in the second part of this Lemma could be replaced by the assumption that $\chi(c_v)$ is independent of $v|\infty$ as in the first part of the Lemma.

Again we will often invoke this Lemma without specifying $S$ or the characters $\psi_v$, but instead we may list local conditions to be satisfied by $\theta$ at a finite number of primes. We hope that the reader will have no difficulty in finding a set $S$ and characters $\psi_v$, so that if we apply the Lemma with these choices it produces a character $\theta$ with the desired local properties.

\newpage

\bibliographystyle{amsalpha}
\bibliography{barnetlambgeegeraghty}

\end{document}